\documentclass[11pt,reqno]{amsart}
\usepackage{amssymb}
\usepackage{amsmath,amsthm}
\usepackage{amsfonts}
\usepackage{leftidx}
\usepackage{times}
\usepackage{mathrsfs}
\usepackage{enumerate}
\usepackage{abstract}
\usepackage{stmaryrd}
\usepackage{hyperref}
\def\R{\mathbb{R}^{3}}

\def\p{\partial}

\def\Q{Q_{1,2}}

\def\n{|\nabla|^{-1}}
\def\d{|\nabla|}
\def\po{\partial_{1}}
\def\R{\mathbb{R}^{2}}

\def\pt{\partial_{2}}
\newtheorem{theorem}{Theorem}[section]
\newtheorem{lemma}[theorem]{Lemma}

\newtheorem{proposition}[theorem]{Proposition}
\newtheorem*{thm}{Main Theorem}
\theoremstyle{definition}

\theoremstyle{remark}
\newtheorem{remark}[theorem]{Remark}

\numberwithin{equation}{section}

\begin{document}

\title[2D Incompressible Isotropic Elastodynamics]{Global existence for the 2D incompressible isotropic elastodynamics  for small initial data }

\author{Xuecheng Wang}
\address{Mathematics Department, Princeton University, Princeton, New Jersey,08544, USA}
\email{xuecheng@math.princeton.edu}
\thanks{}
\maketitle
\begin{abstract}
We establish the global existence and the asymptotic behavior for the 2D incompressible isotropic
 elastodynamics for sufficiently small, smooth initial data in the Eulerian coordinates formulation.
The main tools used to derive the main 
results are, on the one hand, a modified energy method to  derive the energy estimate and on the other hand, a Fourier transform method with a suitable choice of $Z-$ norm to derive the sharp $L^\infty-$ estimate.  We  mention that the global existence of the same system but in the \emph{Lagrangian} coordinates formulation was recently obtained by Lei in \cite{Lei1}. Our goal is to improve the understanding of the behavior of solutions. Also we present a different approach to study $2D$ nonlinear wave equations from the point of view in  frequency space.\end{abstract}
\setcounter{tocdepth}{1}
\tableofcontents

\section{Introduction}
In this paper, we consider the questions of the global existence and the asymptotic behavior  for the motion of elastic waves for isotropic incompressible materials in 2D. The motion of an elastic
 body is described as a time dependent family of orientation preserving diffeomorphism $x(t,\cdot)$, $0\leq t < T$. Material point $X$
in the reference configuration is deformed to the spatial position $x(t, X)$ at time $t$. Initially, we have $x(0,X)=X$. We use $X(t, x)$ to denote the inverse map of $x(t,X)$ and  we have $X(0,x)=x$. Since we can see more directly the motion of an elastic body in Eulerian coordinates, we work in the Eulerian coordinates formulation.   From now on, if without special annotations,  the derivatives $(\p_t, \nabla)$
are with respect to the Eulerian coordinates $(t,x)$.

For the purposes of fixing notations and seeing how the incompressible condition enters  the picture, we record the following lemma, which can be found in \cite{Lei2}[Lemma 2.1]:
\begin{lemma}\label{constraint}
 Given a family of deformations $x(t, X)$, define the velocity,
 deformation gradient and the displacement gradient as follows:
\[
v(t, x) = \frac{d x(t, X)}{d t}\Bigg|_{X= X(t,x)}, \quad F(t, x) = 
\frac{\p x(t, X)}{\p X } \Bigg|_{X= X(t,x)},\quad G(t,x) = F(t,x) - I.
\]
If $x(t, X)$ is incompressible, then $det F(t,x)\equiv 1, $ 
 \begin{equation}\label{incompressible}
 \nabla\cdot v = 0,\quad (\nabla\cdot F^{\top})_j = (\nabla\cdot G^{\top})_j= \p_{i} G_{i,j}=0,
 \end{equation}
and
\begin{equation}\label{constraint1}
 \p_j G_{i,k} - \p_k G_{i, j} = G_{m, k} \p_{m}G_{i,j} - G_{m,j} \p_{m} G_{i,k}, \quad i,j,k\in \{1,2,\cdots,n\}.
 \end{equation}
\end{lemma}
\begin{proof}
A detailed proof of this lemma can be found in \cite{Lei3}[Section 2]. We only give brief \mbox{remarks} here. From the incompressible condition, we have the volume preserving condition, which is $det (F(t,x))$ $\equiv 1$. After taking a derivative with respect to ``$t$" for this equality, we have
\[
0 =\frac{d}{d t} det(F(t,x))= det(F)(t,x) {trac}( F_{t}F^{-1})\Longrightarrow {trac}(F_{t}F^{-1})=0, 
\]
note that
\[
x_{t}(t, X)= v( t, x(t,X))\Longrightarrow F_{t}=  \nabla v F \Longrightarrow 0=trac(F_{t}F^{-1})=trac(\nabla v)= \nabla\cdot v=0.
\]
It is easy to see that $F(0,x)=I$, hence $\nabla\cdot F^{\top}(0,x)=0$. It can be shown that $\nabla\cdot F^{\top}$ satisfies a transport equation. Since it starts from zero, then it will remain at zero for all the time.

Lastly,  equality (\ref{constraint1}) comes from the facts that $ F_{m,j}\p_{m}F_{i, k}=F_{m,k}\p_{m} F_{i,j}$, which is derived from the commutativity of material derivatives $\p_{X_k}\p_{X_j}=\p_{X_j}\p_{X_k}$ in spatial coordinates formulation and the definition of $G(t,x)$. 
\end{proof}

 With the notation in Lemma \ref{constraint}, the system of incompressible isotropic elastodynamcis can be formulated as follows,
\begin{equation}\label{generalequation}
\left\{\begin{array}{ll}
\p_{t} v + v\cdot \nabla v= -\nabla p  + \nabla\cdot T(F), &\\
\p_{t} G - \nabla v = - v\cdot \nabla G + \nabla v G, &\\
\nabla\cdot v =0,\quad \nabla\cdot G^{\top} =0 , &\\
\end{array} \right. 
\end{equation}
where $T(F)$ is the Cauchy stress tensor and it is derived from the energy functional $W(F)$ as follows,
\begin{equation}\label{CST}
 T(F) := (det F)^{-1} S(F) F^\top = (det F)^{-1}\frac{\p W(F)}{\p F} F^{\top},
\end{equation}
where $S(F)$ is the so-called the Piola-Kirchhoff stress. Readers can refer to \cite{Sideris1}[section 2] for the formal derivation of the system (\ref{generalequation}).

For a general isotropic elastodynamcis, the energy functional $W(F)$ satisfies the following relation,
\begin{equation}\label{energyfunctional} 
 W(F) = W(F Q) = W(Q F),
\end{equation}
for all rotational matrices $Q$ such that $Q = Q^\top$ and $det\,Q = 1$. The first equality in (\ref{energyfunctional}) is resulted from the frame
indifference and the second equality in (\ref{energyfunctional}) is resulted from the isotropy of materials. 

Equality (\ref{energyfunctional})
implies that the energy functional $W(F)$ depends on $F$ through the principal invariants of $F F^{\top}$, which are $tr F F^\top$
and $det F F^\top$ in 2D. If we denote $\tau = \frac{1}{2} tr F F^\top$ and 
$\delta= det F$, then $W(F)= \widetilde{W}(\tau,\delta)$ for some smooth function $\widetilde{W}: \mathbb{R}_+\times \mathbb{R}_+ \rightarrow \mathbb{R}_+$. It follows that the Piola-Kirchhoff stress has the following form,
\begin{equation}\label{PKS}
 S(F)= \frac{\p W(F)}{\p F} = \widetilde{W}_\tau(\tau, \delta) F + \widetilde{W}_{\delta} (\tau,\delta) \delta F^{-\top}.
\end{equation}

In this paper, we mainly study the case of Hookean elasticity. We  derive the corresponding system of equations first and then we will remark on the main differences between the Hookean case and general cases at the end of subsection \ref{diagonalization}. 

The  Hookean strain energy functional has the form $W(F) = \frac{1}{2} |F|^2$, which infers that \mbox{$\widetilde{W}(\tau, \delta)=\tau$}. Recall that 
 $det F(t,x)\equiv 1$. Hence, from (\ref{CST}) and (\ref{PKS}), we know that  $T(F)= F F^\top$. Combining this fact with the system of equations  (\ref{generalequation}), we can derive the  system of evolution for Hookean elasticity in terms of $v$ and $G$ as follows,\begin{equation}\label{mainequation1}
\left\{\begin{array}{ll}
\p_{t} v - \nabla\cdot G = -\nabla p - v \cdot \nabla v + \nabla \cdot (G G^\top), &\\
\p_{t} G - \nabla v = - v\cdot \nabla G + \nabla v G, &\\
\nabla\cdot v =0,\quad \nabla \cdot G^{\top} =0 . &\\
\end{array} \right. 
\end{equation}

\subsection{Diagonalizing the system (\ref{mainequation1})}\label{diagonalization}
 If one wants to analyze the system  on the  Fourier side, it is usually more convenient to symmetrize the system (\ref{mainequation1}). We mainly  conduct this process in this subsection.

 As $v$ and $G^{\top}$ are divergence free, and we are in 2D setting,  we can further reduce the system (\ref{mainequation1}), which has six variables into a system that has three variables.
Assume that $\psi$ is the velocity potential of $v^{\bot}$ and $G_{1}, G_{2}$ are the velocity potentials of $G_{\cdot,1}^{\bot}$ and $G_{\cdot, 2}^{\bot}$ respectively.  
More precisely, 
\begin{equation}\label{velocitypotential}
v=(-\p_2 \psi, \p_1 \psi),\quad G_{\cdot, 1}= (-\p_2 G_1, \p_1 G_1), \quad G_{\cdot, 2}=(-\p_2 G_2, \p_1 G_2).
\end{equation}
After tedious computations (readers can refer to the appendix for details),  we can reduce (\ref{mainequation1}) into the system of equations as follows, 
\begin{equation}\label{mainequation}
\left\{\begin{array}{ll}
\p_{t} \psi  - \p_{1} G_{1} - \p_{2} G_{2}  = \widetilde{\mathcal{N}}_{0}, &\\
\p_{t} G_{1} - \p_{1} \psi=\widetilde{\mathcal{N}}_{1}, & \\
\p_{t} G_{2} - \p_{2} \psi = \widetilde{\mathcal{N}}_{2},  &\\
\end{array} \right.
\end{equation}
where 
\begin{equation}\label{eqn400}
\widetilde{\mathcal{N}}_{0} = - \n ( R_{1} f_{1} + R_{2} f_{2} )
,\quad \widetilde{\mathcal{N}}_{1}=Q_{1,2}(G_1,\psi),\quad \widetilde{\mathcal{N}}_{2}= Q_{1,2}(G_2, \psi),
\end{equation}
\[
f_{i}= Q_{1,2} (\p_{i} \psi, \psi ) - Q_{1,2} (\p_{i} G_{1}, G_{1}) - \Q(\p_{i} G_{2}, G_{2}),\quad i \in \{1,2\},
\]
and $R_{i} = \p_{i}/|\nabla|$ denotes the Riesz transform. The operator $Q_{1,2}(\cdot, \cdot)$ in above equations is one of the celebrated null form bilinear operators that is defined as follows,\begin{equation}\label{nullform}
Q_{1,2} ( f , g ) := \p_{1} f \p_{2} g - \p_{2} f \p_{1}g.
\end{equation}
In terms of potentials, we can reduce the constraint (\ref{constraint1}) as follows,
\begin{equation}\label{constraint3}
\p_{1} G_{2} - \p_{2} G_{1} =  \Q (G_{2},  G_{1}).
\end{equation}

\begin{remark}
Note that the pressure term does not come into play in the first equation of the system (\ref{mainequation}). Since this equation is derived by applying \textit{curl}  on (\ref{mainequation1}), as a result, $\nabla p$ disappears. We can recover the pressure term from the solution $(\psi, G_{1}, G_{2}) $ of (\ref{mainequation}). Here is how it works: Firstly, we derive the original velocity field $v$ and displacement gradient $G$ from potentials $\psi$, $G_1$, and $G_2$; Secondly, we derive the following Poisson's equation  by applying $div$ operator on the first equation of system(\ref{mainequation1}), 
\begin{equation}\label{laplace}
\Delta p = div ( \nabla\cdot (G G^{\top}) - v\cdot \nabla v);
\end{equation}
Lastly, the pressure term can be derived by solving above Poisson's equation (\ref{laplace}).
\end{remark}

After diagonalizing the system (\ref{mainequation}), we  get the following system of equations,
 \begin{equation}\label{maineqn1}
\left\{\begin{array}{l}
\p_{t} \phi_{0}  =  \mathcal{N}_{0}= R_{1} \widetilde{\mathcal{N}}_{2} - R_{2} \widetilde{\mathcal{N}}_{1} , \\
\\
\p_{t} \Phi + i |\nabla| \Phi = \mathcal{N}_{1}=\widetilde{\mathcal{N}}_{0} + i (R_{1}  \widetilde{\mathcal{N}}_{1} + R_{2}  \widetilde{\mathcal{N}}_{2}), \\
\end{array} \right.
\end{equation} 
where $\widetilde{\mathcal{N}}_{0}$, $\widetilde{\mathcal{N}}_1$ and $\widetilde{\mathcal{N}}_2$ are defined in (\ref{eqn400}) and $\phi_0$ and $\Phi$  are defined as follows,
\begin{equation}\label{changevariables}
 \phi_{0}  :=  |\nabla|^{-1} (\p_{1} G_{2} - \p_{2} G_{1} )= R_{1} G_{2} - R_{2} G_{1},\quad \Phi:= \psi +  i (R_{1} G_{1} + R_{2} G_{2}).
\end{equation}
It is easy to see that we can  recover $(\psi, G_{1}, G_{2})$ from $(\phi_{0}, \Phi)$ by the following identities, \begin{equation}\label{equation143}
 \psi = \mathfrak{Re}(\Phi) = \frac{\Phi + \overline{\Phi}}{2},\quad
 G_{1} = R_{2}\phi_{0} - R_{1}(\frac{\Phi - \overline{\Phi}} {2 i} ),\quad G_{2}  = - R_{1} \phi_{0} - R_{2}(\frac{\Phi - \overline{\Phi}}{2 i}).
 \end{equation}
Therefore it is sufficient to study the system (\ref{maineqn1}) instead.

We can rewrite  the constraint equation (\ref{constraint3}) in terms of  $\phi_{0}$ and $\Phi$ as follows,
\begin{equation}\label{constraint5}
\phi_{0} = \mathcal{N}_{2},
\end{equation}
\[
 \mathcal{N}_{2} = \n \left[ \Q(R_{2}\phi_{0}, R_{1} \phi_{0}) + \sum_{\begin{subarray}{c}\mu\in\{+,-\} \\ i=1,2\\ \end{subarray}} \frac{c_{\mu}}{2} \Q(R_{i} \phi_{0}, R_{i} \Phi_{\mu}) + \sum_{\mu, \nu \in \{+,-\}} \frac{c_{\mu}c_{\nu}}{4} \Q(R_{2}\Phi_{\mu}, R_{1}\Phi_{\nu})\right],
 \]
where $c_{+} := - i$, $c_{-}:= i$, $\Phi_{+} := \Phi$ and $\Phi_{-} := \overline{\Phi}$. In later context, we also use  $P_{+}\Phi$ to denote $\Phi$ and use $P_{-} \Phi$ to denote  $\overline{\Phi}$.

\begin{remark}
For general cases, the system (\ref{generalequation}) only
 differs from the system
(\ref{mainequation1}), the Hookean case,  at
 ``cubic and higher" level in the sense of decay rate.  
 To know this fact, we recommend interested readers to \cite{Lei2}[section 10] for more details.

 The approach used here in the Hookean case is robust enough to be applied to general isotropic incompressible elastodynamics cases. Firstly, we  consider the energy estimate part. On one hand, intuitively speaking, those cubic and higher order terms  will not cause any \emph{additional}
 obstructions because of the higher decay rate than quadratic terms. On the other hand, to be rigorous, 
 we usually
confronts the difficulty of ``losing derivatives'' in the energy estimate because of the quasilinear nature.  However, the system (\ref{generalequation}) indeed has the
 requisite symmetry properties to avoid losing derivatives. For more details about this fact, please  refer to \cite{Sideris1, Sideris4}. Lastly, we consider the $L^\infty$-estimate  part. We can estimate those additional cubic and higher order terms by the same method used to handle  the cubic terms inside  the equation satisfied by ${\Phi}$ (see  (\ref{reformulateofPhi})). 

\end{remark}

\subsection{Statement of the main result}\label{mainresult}
Before introducing our main theorem, we define   function spaces as follows,
$$
X_{k} (\R) := \big\{ f: \| f \|_{X_{k}} =  \| f \|_{H^{k}} + 
\| S f \|_{H^{\lfloor k/2\rfloor}}  + \| \Omega f \|_{H^{\lfloor k/2\rfloor}} < \infty\big\}, \quad k \in \mathbb{N},
$$
$$
Z:= \Big\{ f :  \| f \|_{Z} = \| (1 + |\xi|)^{N_{1}+6} \widehat{f}(\xi) \|_{L^{\infty}_{\xi}}< + \infty \Big\},
$$
$$
Z' := \Big\{   f:  \| f\|_{Z'} = \|  f \|_{W^{N_{1} + 4}} < \infty \Big\},\quad
Z'_{1} := \Big\{   f:  \| f\|_{Z'} = \|  f \|_{W^{N_{1} + 2}} < \infty \Big\},
$$
\[
\| f \|_{W^{\gamma}}:= \sum_{k\in\mathbb{Z}} 2^{\gamma\max\{k,0\}}\|P_{k} f\|_{L^\infty},
\]
where ``$S$"  is the scaling vector field and it is 
 defined as $S= t \p_{t} + x_{1} \po + x_{2} \pt = t \p_{t} + r \p_{r}$; and ``$\Omega$ "
is the rotational vector field and it is  defined as $\Omega = x_{2} \po - x_{1} \pt = x^{\bot} \cdot \nabla.$

The main result of this paper is as follows,
\begin{thm}\label{theorem1}
Let $N_{0}=300$,$N_{1}= N_{0}/2$, and a fixed  constant $p_{0}\in (0, 1/1000]$, which is sufficiently small. If initial data $(\tilde{\phi}_{0}, \tilde{\Phi}_{0})$ satisfies the constraint \textup{(\ref{constraint5})} and  the  following estimate,
\begin{equation}
\| (\tilde{\phi}_{0}, \tilde{\Phi}_{0})\|_{X_{N_{0}}} + \| (\tilde{\phi}_{0}, \tilde{\Phi}_{0})\|_{Z} = \epsilon_{0} \leq \bar{\epsilon},
\end{equation}
where $\bar{\epsilon}$ is a sufficiently small constant,  then there exists a unique global solution $(\phi_{0}, \Phi) \in C([0, \infty): X_{N_{0}}(\R))$ of the initial value problem: 
 \begin{equation}\label{maineqn}
\left\{\begin{array}{ll}
\p_{t} \phi_{0}  =  \mathcal{N}_{0} , &\\

\p_{t} \Phi + i |\nabla| \Phi = \mathcal{N}_{1}, &\\
\phi_{0} = \mathcal{N}_{2},\quad \phi_{0}(0)= \tilde{\phi}_{0}, \Phi(0) = \tilde{\Phi}_{0}. &\\

\end{array} \right.
\end{equation} 
Moreover, the following estimate holds,
 \begin{equation}\label{equation64210}
\sup_{t\in[0,\infty)} (1+t)^{-p_{0}} \| \Phi(t)\|_{X_{N_{0}}} + (1+t)^{1/2} \| \Phi(t)\|_{Z'} + (1+t)^{1/2-p_{0}} \| \phi_{0}(t)\|_{X_{N_{0}}} + (1+t) \| \phi_{0}(t)\|_{Z'_{1}}  \lesssim \epsilon_{0}.
\end{equation}
\end{thm}

\subsection{Previous results}
 
The long time behavior of isotropic elastodynamics mainly follows the
paradigm of nonlinear wave equation.  As one can see from system (\ref{maineqn}), it is of quasilinear wave type equation (technically speaking, it is of half wave type).

There is an extensive  literature devoted to the study of wave equations. For the purpose of giving concise introduction, we only list some representative works here.  Even for a
semilinear wave equation with  small smooth localized initial data, John  \cite{john}
showed that it can blow up in finite time. Meanwhile, if there exists ``null structure'' inside the nonlinearity, one might
 expect better behavior of the solutions. From the work of Klainerman \cite{klainerman1} and 
the work of Christodoulou \cite{christo}, we can see the role of ``null structure''. Also, the vector field method introduced by S. Klainerman in \cite{klainerman2} is a powerful tool to study the wave equations. 

A natural question is that whether there exists null structure for general isotropic elastodynamics (not limited to the incompressible case). The compressible isotropic elastodynamcis can be characterized by two  families of waves: fast pressure
waves and slower shear waves. In the incompressible case, the 
pressure wave does not present and the equations for shear waves possess an inherent null structure. So the answer for the incompressible case is yes, but the answer is not always yes for the compressible case. 

For the $3D$ \textit{compressible}
elastodynamcis, on the one hand, counterexamples to global existence were shown in \cite{john1} and \cite{Sideris}. In \cite{john1}, John showed that the nontrivial radial solutions  blow up  for the dynamics of an isotropic homogeneous
hyper-elastic medium with initial data that has sufficiently small compact support, if the equations satisfy a certain ``genuine nonlinearity condition''. And in \cite{Sideris}, Tahvildar-Zadeh showed the formation
of singularities of relativistic dynamics of isotropic hyperelastic solids for large initial data. On the other hand, in \cite{Sideris2}, it was first noticed that there exists a null structure within the class of physically
meaningful nonlinearities arising from the hyperelasticity assumption.

Now, let us focus on the incompressible case. We already know that there exists null structure inside the nonlinearity. Hence, we might expect better behavior of solution,  but does it strong enough to guarantee global solution? In $3D$, it is true, see the works of Sideris-Thomases  \cite{Sideris1, Sideris4}.

Naturally, one might wonder what will happen in $2D$. In \cite{Lei2}, Lei-Sideris-Zhou showed the almost global existence  for small initial data. The methods 
used in \cite{Lei2} are mainly the vector field method and the Alinhac's 
trick (the ghost weighted energy method). To see how the ghost weighted energy method works, interested readers may refer to \cite{Lei2} for details.

Here come the questions: does global solution exist for the $2D$ case? If it does, how to push the almost global existence to the global existence for the
$2D$ incompressible case? One might  try to work harder on the vector field method
to improve the previous result. But here, we are trying to provide another method and improve the understanding
of this problem, by using a modified energy method 
and a Fourier transform method with an appropriate choice of $Z$- norm. Instead of from the point of view in physical space, we consider this problem
in the frequency space.  We hope that the argument developed here can shine some lights on the the \emph{2D compressible case}. 

\subsection{The main idea of the proof}

We will use the     bootstrap argument to derive the global existence. It naturally falls into two parts: energy estimate ($L^2-$type) part and the improved dispersion estimate ($L^\infty-$type). For the energy estimate part, we will use a modified energy method to construct an appropriate modified energy, which grows at most polynomially in time. For the improved dispersion estimate part, we will use a Fourier transform method with a suitable choice of $Z$-norm to get sharp $1/t^{1/2}$ decay rate, which means the decay rates of the nonlinear solution and the linear solution are same.

The idea of modified energy method  was ever used in \cite{pierre1} by P. Germain and N. Masmoudi, where they called it the iterated energy method. In \cite{pierre1}, they used the Duhamel formula and did integration by parts in time once to convert the quadratic terms into cubic with price of ``${1}/{\textup{phase}}$", which effectively equivalent to utilizing normal form transformation. Similar idea has also been used in the work of Hunter-Ifrim-Tataru-Wong \cite{Tataru2}, where they called it the modified energy method.
Later,   this method has been applied further to the study of  water waves in the holomorphic coordinates formulation, see  
Hunter-Ifrim-Tataru \cite{Tataru1} and Ifrim-Tataru \cite{Ifrim} for details. 

Although, the idea of using normal form transformation to cancel out the quadratic terms is clear and straightforward. But, to construct an appropriate modified energy which can actually be used to close the energy estimate is highly non-trivial. For the normal transformation part, it is crucial to identify the strong null structure inside, otherwise, one only gets  a singular bilinear operator, which is not helpful to close the energy estimate. In subsection \ref{identifynullstructure}, we identify that there are at least two degrees of requisite angle inside the symbols, one of them comes from the incompressible condition and the other one comes from symmetries. As a result, the normal form transformation is not singular. After adding cubic correction terms that based on the normal form transformation, due to the quasilinear nature, there is a issue of losing a derivative for the quartic terms. To get around this issue, we also add  quartic correction terms into the energy, which will effectively cancel out the part that causes losing a derivative. Those quartic terms are impossible to find out or even works out  without identifying key cancellations inside the normal form transformation and symmetries inside the system, see subsection \ref{constructionofmodified} and subsection \ref{keycancellations} for more details.

 To make the Fourier transform method works properly, it is  important to choose  an appropriate $Z$- norm and get the improved $Z$- norm estimate, thereafter get the improved dispersion estimate to close the argument. Because the linear decay estimate we will prove in this paper has a similar structure as the one proved in the works of  Ionescu-Pusateri \cite{AIonescu1, AIonescu2}, inspired from their works, we choose the $Z$- norm used in \cite{AIonescu1}. For more details about how this argument works, see   \cite{pierre1,germain2,pierre2,AIonescu4, AIonescu5,AIonescu3,AIonescu1,AIonescu2, IP3,IP4}.

Lastly, we  mention the recent result of Lei \cite{Lei1}, which derives the global existence of 2D incompressible elastodynamics in \emph{Lagrangian coordinates} formulation, the system there will be a second order wave type equation instead of a coupled system. For those interested readers, please refer to \cite{Lei1} for more details. An advantage of using \emph{Lagrangian coordinates} is that, we can formulate the system in a nice way such that the bulk quadratic terms disappear. In the sense of decay rate, the nonlinearity is at \emph{cubic and higher} level. However, in the \emph{Eulerian coordinates} formulation, the bulk quadratic terms  do not disappear and this is a major drawback of working in \emph{Eulerian coordinates}.

\subsection{ Outline:} In section \ref{proofofmaintheorem}, we fix notations and prove the main theorem by assuming Proposition\,\ref{proposition1} and  Proposition \ref{proposition2} hold. In section \ref{energyestimate}, we first construct a modified energy and then use this modified energy to  do energy estimate and prove Proposition \ref{proposition1}. In section \ref{sectiondecay}, we prove the linear decay estimate, which is one of the key lemmas to derive the improved dispersion estimate.  In section \ref{improveddispersion}, we prove Proposition \ref{proposition2}.   In section \ref{asymptotic}, we will describe the asymptotic behavior of solution in a lower regularity Sobolev space. Lastly, in the appendix, we show how to derive (\ref{mainequation}) in details. 

\vspace{1\baselineskip}
\noindent \textbf{Acknowledgement.} The author thank his Ph.D advisor Alexandru Ionescu for helpful discussions and suggestions on improving this manuscript and thank Benoit Pausader for helpful discussions. The author wants to express gratitude to Yu Deng for letting him aware of the reference \cite{Lei2}  and helpful discussions. 
 
\section{Preliminary}\label{proofofmaintheorem}
 \subsection{Notations}
Fix an even smooth bump function $\tilde{\psi}: \mathbb{R}\rightarrow [0,1]$ that supports in $[-3/2,3/2]$ and equals to $1$ in $[-5/4,5/4]$, $k\in \mathbb{Z},  x\in \R$, we define
\begin{displaymath}
\psi_{k}(x) : = \tilde{\psi}(|x|/ 2^{k}) - \tilde{\psi} (|x| / 2^{k-1}),\quad \psi_{\leq k}(x)=\sum_{l\leq k}\psi_{l}(x),\quad \psi_{\geq k}(x) = \sum_{l\geq k} \psi_{\geq k}(x).
\end{displaymath}
The frequency projection operator $P_{k}, P_{\leq k}$, and $ P_{\geq k}$ are defined by the multipliers $\psi_{k}(\xi)$,  $\psi_{\leq k}(\xi)$ and  $\psi_{\geq k}(\xi)$ respectively, i.e.,
\[
\widehat{P_{k} f} (\xi) = \psi_{k}(\xi) \widehat{f}(\xi), \quad \widehat{P_{\leq k} f} (\xi) = \psi_{\leq k}(\xi) \widehat{f}(\xi), \quad\widehat{P_{\geq k} f} (\xi) = \psi_{\geq k}(\xi) \widehat{f}(\xi).
\]

For any real number $k\in \mathbb{R}$, we use $k+$, $k_{+}$ and $k_{-}$  to denote $k+\epsilon$, $\max\{k,0\}$ and $\min\{k, 0\}$ respectively throughout the paper, where $\epsilon$ is an arbitrary small constant. For any two numbers $A$ and $B$ and two absolute constants $c, C$, $c < C$, we denote
\[
A\sim B,\quad \emph{if}\,\,  c A \leq B \leq C A, \quad A \lesssim B, \,B\gtrsim A \quad \emph{if}\,\,  A \leq C B. 
\]
The Fourier transform of $f$ is defined as follows,
\[
\widehat{f}(\xi)= \int_{\R} e^{-ix\cdot \xi} f(x) d x.
\]
Besides $\widehat{f}(\xi)$, we also use $\mathcal{F}(f)(\xi)$ to denote the Fourier transform of $f$. We use $\mathcal{F}^{-1}(g)$ to denote the inverse  Fourier transform of $g$.

For any two vectors $\xi$ and $\eta$, we use notation $\angle(\xi,\eta)$ to denote the angle from $\eta$ to $\xi$. Hence $\angle(\xi,\eta)=-\angle(\eta,\xi)$ and $\angle(\xi, \eta)\in [-\pi, \pi]$. When $\angle(\xi, \eta)$ is very close to $\pm\pi$, i.e., $\xi$ and $\eta$ are almost parallel but in the opposite direction,  we say $\angle(\xi, \eta)$ is small in the sense that $\angle(\xi, -\eta)$ is small in this scenario. We say a quantity has $k$ degrees of angle, if this quantity is of size $\angle(\xi,  \eta)^{k}$ or $\angle(\xi,  -\eta)^{k}$  when $\angle(\xi, \eta)$ is small.  We mention that $\cos(\angle(\xi, \eta))$ is understood in the usual sense as $\cos(\angle(\xi, \eta))= \xi\cdot \eta/(|\xi||\eta|).$

We will use the convention that the symbol $q(\cdot, \cdot)$ of a bilinear operator $Q(\cdot, \cdot)$ is defined in the following  sense throughout this paper,
\[\mathcal{F} [Q(f,g)](\xi)= \frac{1}{4\pi^2} \int_{\mathbb{R}^2} \widehat{f}(\xi-\eta)\widehat{g}(\eta) q(\xi-\eta, \eta) d \eta, \quad \textup{where $f$ and $g$ are two well defined functions.} 
\]

\subsection{Bilinear estimate}
Define a class of symbol and its associated norms as follows,
\[
\mathcal{S}^\infty:=\{ m: \mathbb{R}^4\,\textup{or}\, \mathbb{R}^6 \rightarrow \mathbb{C}, m\,\textup{is continuous and }  \quad \| \mathcal{F}^{-1}(m)\|_{L^1} < \infty\},
\]
\[
\| m\|_{\mathcal{S}^\infty}:=\|\mathcal{F}^{-1}(m)\|_{L^1}, \quad \|m(\xi,\eta)\|_{\mathcal{S}^\infty_{k,k_1,k_2}}:=\|m(\xi, \eta)\psi_k(\xi)\psi_{k_1}(\xi-\eta)\psi_{k_2}(\eta)\|_{\mathcal{S}^\infty},
\]
\[
 \|m(\xi,\eta,\sigma)\|_{\mathcal{S}^\infty_{k,k_1,k_2,k_3}}:=\|m(\xi, \eta,\sigma)\psi_k(\xi)\psi_{k_1}(\xi-\eta)\psi_{k_2}(\eta-\sigma)\psi_{k_3}(\sigma)\|_{\mathcal{S}^\infty}.
\]
\begin{lemma}\label{boundness}
Given $m, m'\in\mathcal{S}^{\infty}$ and two well defined functions $f_1$, $f_2$, and $f_3$,
then the following estimates hold,
\begin{equation}\label{productsymbol}
\| m\cdot m'\|_{\mathcal{S}^\infty}\lesssim \|m \|_{\mathcal{S}^\infty} \| m'
\|_{\mathcal{S}^\infty},
\end{equation}
\[
\| \mathcal{F}^{-1}( \int_{\R} m(\xi, \eta)  \widehat{f_1}(\xi-\eta) \widehat{f_2}(\eta)  d \eta) (x) \|_{L^{r}}  \lesssim \|m\|_{\mathcal{S}^\infty} \| f_1\|_{L^{p}}  \| f_2\|_{L^{q}},\quad \frac{1}{r}=\frac{1}{p}+\frac{1}{q},
\]
\[
\| \mathcal{F}^{-1}( \int_{\R\times\R} m(\xi, \eta,\sigma)  \widehat{f_1}(\xi-\eta) \widehat{f_2}(\eta-\sigma)\widehat{f_3}(\sigma)  d \eta d \sigma) (x) \|_{L^{s}}  \lesssim \|m\|_{\mathcal{S}^\infty} \| f_1\|_{L^{p'}}  \| f_2\|_{L^{q'}} \| f_3\|_{L^{r'}},
\]
where $p', q'$, $r'$, and $s$ satisfy $1/s=1/p'+1/q'+1/r'.$
\end{lemma}
\begin{proof}
 The proof is standard, or one can see \cite{AIonescu2}[Lemma 5.2].
\end{proof}

To estimate the $\mathcal{S}^{\infty}_{k,k_1,k_2}$  or the $\mathcal{S}^{\infty}_{k,k_1,k_2,k_3}$ norms of symbols, we  constantly use the following lemma . 
\begin{lemma}\label{Snorm}
For $i\in\{2,3\}, $ if $f:\mathbb{R}^{2i}\rightarrow \mathbb{C}$ is a smooth function and $k_1,\cdots, k_i\in\mathbb{Z}$, then the following estimate holds,
\begin{equation}\label{eqn61001}
\| \int_{\mathbb{R}^{2i}} f(\xi_1,\cdots, \xi_i) \prod_{j=1}^{i} e^{i x_j\cdot \xi_j} \psi_{k_j}(\xi_j) d \xi_1\cdots  d\xi_i \|_{L^1_{x_1, \cdots, x_i}} \lesssim \sum_{m=0}^{2i+1}\sum_{j=1}^i 2^{m k_j}\|\p_{\xi_j}^m f\|_{L^\infty} .
 \end{equation}
\end{lemma}
\begin{proof}
Let's first consider the case when $i=2$. Through scaling, it is  sufficient to prove above estimate for the case $k_1=k_2=0$. From the integration by parts in $\xi_1$ and $\xi_2$, we have the following pointwise estimate
\[
(1+|x_1|+|x_2|)^{5}\Big|\int_{\R\times\R} e^{i x_1\cdot \xi_1} e^{i x_2 \cdot \xi_2} f(\xi_1, \xi_2)\psi_{0}(\xi_1)\psi_{0}(\xi_2) d \xi_1  d\xi_2 \Big|\lesssim \sum_{m=0}^5\big[\|\p_{\xi_1}^m f\|_{L^\infty} + \| \p_{\xi_2}^m f \|_{L^\infty}\big],
\]
which is sufficient to finish the proof of (\ref{eqn61001}). We can prove the case when $i=3$ very similarly, hence we omit the details here.
\end{proof}

\subsection{Bootstrap assumption and proof of the main theorem}
Recall that  $\phi_0=\mathcal{N}_2$ (see (\ref{constraint5})) and $\mathcal{N}_{2}$ is  quadratic. Intuitively speaking, if we expect that energy is appropriately grow at rate $t^{p_0}$ and the decay rate of $\Phi$ is sharp, then  the $L^{2}$ norm of $\phi_{0}$ decays at a rate $1/t^{1/2-p_0}$ and the $L^{\infty}$ norm of $\phi_{0}$ decays at a rate $1/t$. Therefore, it motivates us to state the following bootstrap assumption,
\[
\sup_{t\in[0,T]} (1+t)^{-p_{0}} \| \Phi(t)\|_{X_{N_{0}}} + (1+t)^{1/2-p_{0}} \| \phi_{0}(t)\|_{X_{N_{0}}}  
\]
\begin{equation}\label{smallness}
+ (1 + t )^{1/2} \| \Phi\|_{Z'} + (1+t ) \| \phi_{0}\|_{Z'_{1}} + (1+t)^{-2p_0}\| \Phi(t)\|_{Z} \lesssim \epsilon_{1}:= \epsilon_0^{5/6}.
\end{equation}
Since the size of initial data is $\epsilon_0$, from the continuity of solution, we know the existence of $T>0$ in above bootstrap assumption.  In later context, without further annotations, the solution is considered in the time interval $[0,T] $ and satisfies above estimate.

 In section \ref{energyestimate}, we will prove the following proposition, 
\begin{proposition}[Energy estimate]\label{proposition1}
Under the bootstrap assumption \textup{(\ref{smallness})}, we have 
\begin{equation}\label{improvedenergyestimate}
\sup_{t\in[0,T]} (1+t)^{-p_0}\|(\phi_0, \Phi)\|_{X_{N_0}}\lesssim \epsilon_0 +\epsilon_1^2\lesssim \epsilon_0.
\end{equation}

\end{proposition}
\begin{proof}
It follows from the result of Lemma \ref{remainderestimate2} directly.
\end{proof}

 In section \ref{improveddispersion}, we will prove the following proposition, 
\begin{proposition}\label{proposition2}
Under the bootstrap assumption \textup{(\ref{smallness})} and the energy estimate \textup{(\ref{improvedenergyestimate})}, we can derive the following improved estimates, 
\begin{equation}\label{step1}
\sup_{t\in[0,T]} (1+t)^{-p_{0}} \| \Phi(t)\|_{X_{N_{0}}} + (1+t)^{1/2} \| \Phi(t)\|_{Z'} + (1+t)^{-2p_0} \|\Phi\|_{Z} \lesssim \epsilon_{0} + \epsilon_{1}^{2}\lesssim \epsilon_0,
\end{equation}
\begin{equation}\label{step2}
\sup_{t\in[0, T]} (1+t)^{1/2-p_{0}} \| \phi_{0}(t)\|_{X_{N_{0}}} + (1+t) \| \phi_{0}(t)\|_{Z'_{1}} \lesssim  (\epsilon_{0} + \epsilon_{1})^{2}\lesssim \epsilon_0^2.
\end{equation}
\end{proposition}

\begin{proof}[Proof of the Main Theorem ]
Combining results in above two propositions, it is easy to see that, under the bootstrap assumption (\ref{smallness}),  we have
\[
\sup_{t\in[0,T]} (1+t)^{-p_{0}} \| \Phi(t)\|_{X_{N_{0}}} + (1+t)^{1/2-p_{0}} \| \phi_{0}(t)\|_{X_{N_{0}}}+(1+t)^{-2p_0} \|\Phi\|_{Z}
\]
\begin{equation}\label{equation016}
  + (1+t)^{1/2} \| \Phi(t)\|_{Z'} + 
(1+t) \| \phi_{0}(t)\|_{Z'_{1}} \lesssim \epsilon_{0}.
\end{equation}
Therefore, we can keep iterating the local result and extend the time interval of existence to the full time interval $[0,+\infty)$, i.e., solution exists globally. Moreover, the desired estimate (\ref{equation64210}) holds.
\end{proof}

\section{Energy Estimate}\label{energyestimate}
This section is devoted to prove Proposition \ref{proposition1}. Firstly, we identify null structures inside the system, which are the bases of doing this argument. Secondly, we identify the most problematic terms, which help us to figure out how to construct a modified energy. Finally, we use this  modified energy to do energy estimate and finish the proof of Proposition \ref{proposition1}.

\subsection{Identifying null structures inside the system}\label{identifynullstructure}
 
The goal of this subsection is to  check the symbols of quadratic terms very carefully to see whether there exist null structures and  how ``strong"  null structures are.  Based on the input types inside the quadratic terms, we decompose the nonlinearities $\mathcal{N}_{0}$ and $\mathcal{N}_{1}$  and the constraint $\mathcal{N}_{2}$ as follows, \[
 \mathcal{N}_{0} =  \sum_{\mu\in \{+,- \}}Q_{0,\mu}(\phi_{0}, \Phi_{\mu})  + \sum_{\mu, \nu\in \{+,-\}} Q_{\mu, \nu}(\Phi_{\mu}, \Phi_{\nu}),
\]
\[
\mathcal{N}_{1} = \tilde{Q}_{0,0} (\phi_{0}, \phi_{0}) + \sum_{\mu \in \{+,-\}} \tilde{Q}_{0,\mu} (\phi_{0}, \Phi_{\mu}) + \sum_{\mu, \nu \in \{+, -\}}\tilde{Q}_{\mu,\nu}(\Phi_{\mu}, \Phi_{\nu}),
\]
\[
\mathcal{N}_{2} = \tilde{Q}_{0,0}^{1} (\phi_{0}, \phi_{0}) + \sum_{\mu \in \{+,-\}} \tilde{Q}_{0,\mu}^{1} (\phi_{0}, \Phi_{\mu}) + \sum_{\mu, \nu \in \{+, -\}}\tilde{Q}_{\mu,\nu}^{1}(\Phi_{\mu}, \Phi_{\nu}),
\]
where
\[
Q_{0,\mu} (\phi_{0}, \Phi_{\mu}) = \frac{1}{ 2 |\nabla|} \Big(\Q( |\nabla| \phi_{0}, \Phi_{\mu}) - \Q(\ R_{1} \phi_{0}, \po \Phi_{\mu}) - \Q( R_{2} \phi_{0}, \pt \Phi_{\mu})  \Big),\quad \mu \in \{+, -\},
\]
\[
Q_{\mu, \nu}(\Phi_{\mu}, \Phi_{\nu}) = \frac{c_{\mu}}{4 |\nabla|} \Big(\Q(R_{1}\Phi_{\mu}, \pt \Phi_{\nu}) - \Q(R_{2}\Phi_{\mu}, \po \Phi_{\nu})\Big),\quad \mu, \nu \in \{+,-\} ,
\]
\[
\tilde{Q}_{0,0} (\phi_{0}, \phi_{0}) =   \sum_{ i = 1,2}\frac{R_{i}}{|\nabla|} \Big( \Q(\p_{i} R_{2}\phi_{0}, R_{2}\phi_{0}) + \Q(\p_{i} R_{1} \phi_{0}, R_{1} \phi_{0}) \Big) ,
\]
\[
 \tilde{Q}_{0,\mu} (\phi_{0}, \Phi_{\mu}) = \sum_{i=1,2}\frac{ c_{\mu}\,R_{i}}{2\, |\nabla|}\Big(  \Q(\p_{i} R_{1}\phi_{0}, R_{2} \Phi_{\mu})  -  \Q(\p_{i} R_{2}\phi_{0}, R_{1} \Phi_{\mu} ) + 
\]
\begin{equation}\label{eqn901}\Q(\p_{i}R_{2}\Phi_{\mu}, R_{1} \phi_{0}) -  \Q(\p_{i} R_{1}\Phi_{\mu}, R_{2} \phi_{0} )\Big) + \frac{i}{2 |\nabla|} \Big(\Q(R_{2} \phi_{0}, \po \Phi_{\mu} ) - \Q(R_{1} \phi_{0}, \pt \Phi_{\mu} )\Big) ,
\end{equation}
\[
\tilde{Q}_{\mu,\nu}(\Phi_{\mu}, \Phi_{\nu}) = \sum_{i, j=1,2}\frac{R_{i}}{4\, |\nabla|} \Big( c_{\mu} c_{\nu} \Q(\p_{i} R_{j} \Phi_{\mu}, R_{j} \Phi_{\nu}) - \Q(\p_{i} \Phi_{\mu}, \Phi_{\nu})  \Big) + 
\]
\begin{equation}\label{eqn900}
\frac{i\, c_{\mu} }{4 |\nabla|} \Big( \Q(|\nabla| \Phi_{\mu}, \Phi_{\nu}) - \Q(R_{1} \Phi_{\mu}, \po\Phi_{\nu} ) -\Q(R_{2}\Phi_{\mu}, \pt \Phi_{\nu}) \Big),\quad \mu, \nu \in \{+,-\},
\end{equation}
\[
\tilde{Q}_{0,0}^{1} (\phi_{0}, \phi_{0}) = \frac{1}{2|\nabla|} [ \Q(R_{2}\phi_{0}, R_{1} \phi_{0})-\Q(R_{1}\phi_{0}, R_{2} \phi_{0})], \]
\[  \tilde{Q}_{0,\mu}^{1} (\phi_{0}, \Phi_{\mu}) = \sum_{i=1}^{2}\frac{c_{\mu}}{2} \n \Q(R_{i} \phi_{0}, R_{i} \Phi_{\mu}),
\]
\begin{equation}\label{eqn902}
\tilde{Q}_{\mu,\nu}^{1}(\Phi_{\mu}, \Phi_{\nu}) =\frac{c_{\mu}c_{\nu}}{8} \n[\Q(R_{2}\Phi_{\mu}, R_{1}\Phi_{\nu})-\Q(R_{1}\Phi_{\mu}, R_{2}\Phi_{\nu})].
\end{equation}
After tedious calculations, we can show  that the associated symbols of above bilinear operators are given as follows,
\begin{equation}\label{eqn904}
m_{0,\mu}(\xi - \eta, \eta) = -\frac{\xi \cdot (\xi-\eta)}{2 |\xi||\xi-\eta|} \big( (\xi - \eta) \times \eta \big),\quad m_{\mu, \nu} (\xi - \eta, \eta) = \frac{c_{\mu}}{4 | \xi |} \frac{1}{|\xi - \eta|} \big( (\xi - \eta) \times \eta \big)^{2},\end{equation}

\begin{equation}\label{eqn905}
\tilde{m}_{0, \mu} (\xi -\eta, \eta)= -\frac{c_{\mu}}{2 |\xi |^{2}}\Big(\frac{\xi \cdot (\xi - \eta)}{|\xi - \eta| |\eta|} -  \frac{\xi \cdot \eta}{|\xi - \eta| |\eta|} \Big)  \Big( (\xi - \eta) \times \eta \Big)^{2} - \frac{i}{2 |\xi | |\xi - \eta|} \Big( (\xi - \eta) \times \eta \Big)^{2},
\end{equation}

\begin{equation}\label{equation13203}
\tilde{m}_{\mu, \nu}(\xi -\eta, \eta) = \Big(- \frac{c_{\mu} c_{\nu}}{4 |\xi|} \frac{\xi \cdot (\xi -\eta)}{|\xi |} \frac{(\xi - \eta) \cdot \eta}{|\xi - \eta| |\eta|} - \frac{\xi \cdot (\xi - \eta)}{4 |\xi|^{2}} -\frac{ i c_{\mu}\,\xi\cdot (\xi-\eta)}{4|\xi| |\xi-\eta|} \Big) \Big( (\xi - \eta) \times \eta \Big),
\end{equation}
\begin{equation}\label{eqn3002}
\tilde{m}_{0,0}(\xi - \eta, \eta) = - \frac{ \xi \cdot (\xi - \eta)}{|\xi |^{2}}\frac{(\xi - \eta) \cdot \eta}{|\xi - \eta| |\eta|} \big( (\xi - \eta) \times \eta \big),\quad \tilde{m}_{0,0}^{1}(\xi-\eta, \eta) =  \frac{-\big((\xi-\eta) \times \eta \big)^2}{2|\xi| |\xi-\eta||\eta|} ,
\end{equation}
\begin{equation}\label{eqn3001}
\tilde{m}_{0,\mu}^{1}(\xi-\eta, \eta) = - \frac{c_{\mu}}{2} \frac{(\xi-\eta)\cdot  \eta}{|\xi-\eta| |\eta| |\xi|} \big((\xi-\eta) \times \eta \big),\quad  \tilde{m}_{\mu,\nu}^{1}(\xi-\eta, \eta) = \frac{c_{\mu}c_{\nu}}{8}  \tilde{m}_{0,0}^{1}(\xi-\eta, \eta).
\end{equation}
Recall that $c_{+}=-i$ and $c_{-}=i$ as defined in the introduction. As an example, we will show detail computations for (\ref{eqn901}). which is very typical. All other cases can be computed in the same way.  From explicit formula in (\ref{eqn901}), we have
\[
\tilde{m}_{0, \mu}(\xi-\eta, \eta) = \sum_{j=1,2}\frac{ic_{\mu} \xi_{j}}{2|\xi|^2}\Big[i(\xi-\eta)_{j}[(\xi-\eta)\times \eta] \frac{(\xi-\eta)_1 \eta_2-(\xi-\eta)_2\eta_1}{|\xi-\eta||\eta|} 
\]
\[
- i \eta_{j} [(\xi-\eta)\times \eta] \frac{(\xi-\eta)_1\eta_{2} - (\xi-\eta)_2\eta_1}{|\xi-\eta||\eta|} \Big] + \frac{i}{2 |\xi|} \big[ [(\xi-\eta)\times\eta] \frac{(\xi-\eta)_2 \eta_1-(\xi-\eta)_1 \eta_2}{|\xi-\eta|}  \big] 
\]
\[
= \frac{-c_{\mu}}{2 |\xi|^2}\Big[ \frac{\xi\cdot(\xi-\eta)}{|\xi-\eta||\eta|}- \frac{\xi\cdot \eta}{|\xi-\eta||\eta|}\Big] \Big( (\xi-\eta)\times\eta\Big)^2 - \frac{i}{2|\xi||\xi-\eta|}\Big( (\xi-\eta)\times\eta \Big)^2,
\]
therefore (\ref{eqn905}) holds.

 From above detailed  formulas of symbols, we can see that  all symbols vanishes when $(\xi-\eta)\parallel \eta$. Hence, indeed, there are null structures inside nonlinearities. But does those null structures, especially the  one in the symbol $\tilde{m}_{\mu, \nu}(\cdot, \cdot)$, strong enough?  The answer depends on how strong we need them to be. 
 
In the later modified energy estimate part, we will see that it is very crucial to gain at least two degrees of angle for $\tilde{m}_{\mu, \nu}(\cdot, \cdot)$ in certain scenarios. Otherwise,  the symbol will be singular after dividing the phase.

 In the following, we will show that we can gain one more degree of angle from symmetries. However, this angle depends on the fact that whether $(\xi-\eta)$ and $\eta$ are in the same direction.  To be more precise, we divide into different cases based on the types of phases, which are determined by the types of quadratic terms. More precisely, the phases are defined as follows,
\[
\Phi_{\mu, \nu}(\xi, \eta):= |\xi|-\mu |\xi-\eta| - \nu |\eta|, \quad \mu, \nu\in\{+,-\}.
\]
 
Before we proceed, we mention that the types of phase and the discussion here are also related to the normal form transformations that we will do later (see also subsection \ref{normalform}).

\begin{enumerate}
\item[(i)] For the phase of type $|\xi|-|\xi-\eta|-|\eta|$, the corresponding symbol is $\tilde{m}_{+,+}(\cdot, \cdot)$. In this case, the phase vanishes when $\angle(\xi, \eta)=\angle(\xi-\eta, \eta)=0$ and $|\xi|$ has comparable size of $|\xi-\eta|$. 

\item[(ii)] For the phase of type $|\xi|-|\xi-\eta|+|\eta|$ or $|\xi|+|\xi-\eta|-|\eta|$, the corresponding symbol is $\tilde{m}_{+,-}(\cdot, \cdot)$ or $\tilde{m}_{-, +}(\cdot, \cdot)$. For the first type, the phase vanishes when $\angle(\xi, -\eta)=\angle(\xi-\eta, -\eta)=0$. For the second type, the phase vanishes when $\angle(\xi, \eta)=\angle(\xi-\eta, -\eta)=0$. In whichever case, we have $\angle(\xi-\eta, -\eta)=0.$

\item[(iii)] For the phase of type $|\xi|+|\xi-\eta|+|\eta|$, it does not vanish and has a lower bound regardless whether $\xi$ and $\eta$ are parallel or not. For this case, it is not necessary to gain two degrees of angle.
\end{enumerate}
For case (i), since two inputs of $\tilde{Q}_{+,+}(\cdot, \cdot)$ are of the same type,  we can utilize self symmetry to see that the symbol of  bilinear form $\tilde{Q}_{+,+}(\cdot, \cdot)$ is also given as follows,
\[
\tilde{m}'_{+,+}(\xi-\eta, \eta):=[\tilde{m}_{+,+}(\xi-\eta, \eta) + \tilde{m}_{+,+}(\eta, \xi-\eta)]/2
\]
\[
= \Big(- \frac{c_{+} c_{+}}{8 |\xi|} \frac{\xi \cdot (\xi -\eta)}{|\xi |} \frac{(\xi - \eta) \cdot \eta}{|\xi - \eta| |\eta|} - \frac{\xi \cdot (\xi - \eta)}{8 |\xi|^{2}} -\frac{ i c_{+}\,\xi\cdot (\xi-\eta)}{8 |\xi| |\xi-\eta|} \Big) \Big( (\xi - \eta) \times \eta \Big) 
\]
\[
+\Big(- \frac{c_{+} c_{+}}{8|\xi|} \frac{\xi \cdot \eta}{|\xi |} \frac{(\xi - \eta) \cdot \eta}{|\xi - \eta| |\eta|} - \frac{\xi \cdot \eta}{8  |\xi|^{2}} -\frac{ i c_{+}\,\xi\cdot\eta}{8 |\xi| |\eta|} \Big) \Big( \eta \times (\xi-\eta) \Big)
\]
\begin{equation}\label{twosame}
= \frac{-\xi\cdot(\xi-2\eta)}{8 |\xi|^2}\underbrace{(1-\cos(\angle(\xi-\eta, \eta)))\Big((\xi-\eta)\times\eta\Big)}_{\text{three degrees of angle $\angle(\xi-\eta, \eta)$}} - \frac{\xi}{8|\xi|}\cdot \underbrace{\Big(\frac{\xi-\eta}{|\xi-\eta|}-\frac{\eta}{|\eta|}\Big) \Big( (\xi - \eta) \times \eta \Big)}_{\text{two degrees of angle $\angle(\xi-\eta, \eta)$}}. 
\end{equation}
For case (ii), we  couple term $\tilde{Q}_{+,-}(\Phi, \overline{\Phi})$ with term $\tilde{Q}_{-,+}(\overline{\Phi}, \Phi)$ and  define $\widetilde{Q}_{+,-}(\Phi, \overline{\Phi}):= \tilde{Q}_{+,-}(\Phi, \overline{\Phi}) +\tilde{Q}_{-,+}(\overline{\Phi}, {\Phi}) $. Its  corresponding symbol is given as follows, 
\[
\tilde{m}'_{+,-}(\xi-\eta, \eta)= \tilde{m}_{+,-}(\xi-\eta, \eta) + \tilde{m}_{-,+}(\eta, \xi-\eta)
\]
\[
= \Big(- \frac{c_{+} c_{-}}{4 |\xi|} \frac{\xi \cdot (\xi -\eta)}{|\xi |} \frac{(\xi - \eta) \cdot \eta}{|\xi - \eta| |\eta|} - \frac{\xi \cdot (\xi - \eta)}{4 |\xi|^{2}} -\frac{ i c_{+}\,\xi\cdot (\xi-\eta)}{4 |\xi| |\xi-\eta|} \Big) \Big( (\xi - \eta) \times \eta \Big) 
\]
\[
+\Big(- \frac{c_{-} c_{+}}{4|\xi|} \frac{\xi \cdot \eta}{|\xi |} \frac{(\xi - \eta) \cdot \eta}{|\xi - \eta| |\eta|} - \frac{\xi \cdot \eta}{4|\xi|^{2}} -\frac{ i c_{-}\,\xi\cdot\eta}{4 |\xi| |\eta|} \Big) \Big( \eta \times (\xi-\eta) \Big)
\]
\begin{equation}\label{twoopposite}
= \frac{-\xi\cdot(\xi-2\eta)}{4|\xi|^2}\underbrace{(1+\cos(\angle(\xi-\eta, \eta)))\Big((\xi-\eta)\times\eta\Big)}_{\text{three degrees of angle $\angle(\xi-\eta, -\eta)$}} - \frac{\xi}{4 |\xi|}\cdot \underbrace{\Big(\frac{\xi-\eta}{|\xi-\eta|}+\frac{\eta}{|\eta|}\Big) \Big( (\xi - \eta) \times \eta \Big)}_{\text{two degrees of angle $\angle(\xi-\eta, -\eta)$}}. 
\end{equation}
To make notations consistent, we define \begin{equation}\label{equation13202}
\tilde{m}'_{-,-}(\xi-\eta, \eta):= \tilde{m}_{-,-}(\xi-\eta, \eta), \quad \widetilde{Q}_{\mu, \nu}(\Phi_{\mu}, \Phi_{\nu}):=  \tilde{Q}_{\mu, \nu}(\Phi_{\mu}, \Phi_{\nu}), \quad (\mu, \nu)\in\{(+,+), (-,-)\}.
\end{equation}
To sum up, we can reformulate the equation satisfied by $\Phi$ as follows,
\begin{equation}\label{mainequation10}
\p_t \Phi + i |\nabla |\Phi = \mathcal{N}_1=\tilde{Q}_{0,0}(\phi_{0}, \phi_{0}) + \sum_{\mu\in\{+,-\}} \tilde{Q}_{0,\mu} (\phi_{0}, \Phi_{\mu}) + \sum_{(\mu,\nu)\in\mathcal{S}} \widetilde{Q}_{\mu, \nu}(\Phi_{\mu}, \Phi_{\nu}),
\end{equation}
where $\mathcal{S}:=\{(+,+),(+,-), (-,-)\}.$  Inside the symbol of $\widetilde{Q}_{+,+}(\cdot, \cdot)$, we can gain two degrees of angle when frequencies of two inputs are parallel and in the \emph{same direction}. Inside the symbol of $\widetilde{Q}_{+,-}(\cdot, \cdot)$, we can gain two degrees of angle when frequencies of two inputs are parallel and in the \emph{opposite direction}. We can always gain one degree of angle regardless whether two frequencies are in the same direction or not.

\subsection{Normal form transformation}\label{normalform}
The first step of constructing the modified energy is to find out  the normal form transformation. More precisely, we are looking for a normal form transformation $\Phi \rightarrow \tilde{\Phi}$, such that the equation satisfied by $\tilde{\Phi}$ is cubic and higher. Let
\begin{equation}\label{normalformtran}
\tilde{\Phi} = \Phi + \sum_{(\mu, \nu) \in \mathcal{S}} A_{\mu, \nu}(\Phi_{\mu}, \Phi_{\nu}),
\end{equation}
where $A_{\mu, \nu}(\cdot, \cdot)$, $(\mu, \nu)\in \mathcal{S}$,  is an unknown bilinear operators to be determined.
\noindent Recall the equation (\ref{mainequation10}) satisfied by $\Phi$, then we have the following,
\[
\p_{t} \tilde{\Phi} + i |\nabla| \tilde{\Phi} =  \sum_{(\mu, \nu) \in \mathcal{S}}\widetilde{Q}_{\mu,\nu}(\Phi_{\mu}, \Phi_{\nu}) +  \sum_{(\mu, \nu) \in \mathcal{S}} i |\nabla |A_{\mu, \nu}(\Phi_{\mu}, \Phi_{\nu} )- \]
\[ \Big(\sum_{(\mu, \nu) \in \mathcal{S}} A_{\mu, \nu}( i a_{\mu} |\nabla| \Phi_{\mu}, \Phi_{\nu})  + A_{\mu, \nu}( \Phi_{\mu}, i |\nabla| a_{\nu} \Phi_{\nu})\Big) + \textit{cubic and higher} ,
\]
where $a_{+} =1$ and $a_{-} =-1$. To cancel out quadratic terms, it is sufficient if the following equality holds for all $(\mu, \nu)\in \mathcal{S}$,
\begin{equation}\label{condition1}
\widetilde{Q}_{\mu, \nu} (\Phi_{\mu}, \Phi_{\nu}) + i |\nabla| A_{\mu, \nu} (\Phi_{\mu}, \Phi_{\nu}) - i a_{\mu} A_{\mu, \nu}(|\nabla| \Phi_{\mu}, \Phi_{\nu}) - i a_{\nu} A_{\mu, \nu} (\Phi_{\mu}, |\nabla| \Phi_{\nu}) = 0,
\end{equation}
which gives us,
\begin{equation}\label{symbolofnormal}
a_{\mu, \nu} (\xi -\eta, \eta) = \frac{i}{ (|\xi | -a_{\mu} |\xi -\eta| - a_{\nu}|\eta|)} \tilde{m}'_{\mu, \nu}(\xi -\eta, \eta).
\end{equation}

\subsection{The  $S^{\infty}$ norm estimate of symbols}\label{guide}

In this section, we first discuss how to estimate the $\mathcal{S}^{\infty}_{k,k_1,k_2}$ norm for a general symbol which depends on the angular variable, and then we estimate the $S^\infty_{k,k_1,k_2}$ norm of $a_{\mu, \nu}(\xi-\eta, \eta)$, $(\mu, \nu)\in \mathcal{S}$. The method stated here can be easily generalized to the three independent variables setting. One can estimate $\mathcal{S}^{\infty}_{k,k_1,k_2,k_3}$ norm of a symbol very similarly.  For later stated $\mathcal{S}^{\infty}$ norm estimates of  symbols, readers can refer to this subsection for help to see the validities of those stated estimates.

The essential tool we use is Lemma \ref{Snorm}. Beside the pointwise estimate, we also have to estimate the derivatives of symbols. Since the angular variable of symbol may appear in the denominator, it's not so straightforward to see the upper bound of $S^\infty_{k,k_1,k_2}$ norm directly. Hence, we provide the following guide of doing estimates to readers, which consists of two essential steps and an example.

\textbf{Step $1$:\quad} Choose the independent variables.\\
Since there are only two independent variables among three variables: $\xi$, $\xi-\eta$, and $\eta$. It's important to choose the right two independent variables, otherwise the estimate can be unnecessarily rough. We choose the least two of the three variables as the independent variables. For example, for a symbol as follows,
\[
m(\xi, \eta)\psi_k(\xi)\psi_{k_1}(\xi-\eta)\psi_{k_2}(\eta),
\]
variables $\xi$ and $\eta$ are not always  the independent variables. If $|\xi-\eta|\leq |\xi|\sim |\eta|$, then we let $\xi-\eta$ and $\xi$ to be independent variables first and then apply Lemma \ref{Snorm}.

\textbf{Step 2:\quad} View the angular part as a whole part when we have angular variable in the denominator.

The main point of this step is that we can reformulate the aforementioned symbols as follows,
\[
m(\xi, \eta)= \tilde{m}(\xi,\eta) f(\frac{\eta}{|\eta|}\times \frac{\xi-\eta}{|\xi-\eta|} , (1,0)\times\frac{\eta}{|\eta|})=\tilde{m}(\xi,\eta) f(\sin(\angle(\xi-\eta, \eta)) , \sin(\angle(\eta, (1,0)))),
\]
where $m(\xi, \eta)$ is one of the aforementioned symbol, $\tilde{m}(\xi, \eta)$ is a regular symbol \footnote{For a regular symbol, after choosing the independent variables as in Step $1$, the right hand side of
(\ref{eqn61001}) is comparable to the $L^\infty$ norm of itself.} and $f:\mathbb{R}^2\rightarrow \mathbb{R}$ is a smooth function.

\textbf{An example:}\quad We choose $a_{+,+}(\xi-\eta, \eta)$ as a representative example, other symbols can be done similarly. From (\ref{twosame}) and (\ref{symbolofnormal}), we have
\[
a_{+,+}(\xi-\eta,\eta) = \frac{i\tilde{m}_{+,+}'(\xi-\eta, \eta)}{|\xi|-|\xi-\eta|-|\eta|}= \frac{i(|\xi|+|\xi-\eta|+|\eta|)}{2|\xi-\eta||\eta|(\cos(\angle(\xi-\eta, \eta))-1)} \tilde{m}_{+,+}'(\xi-\eta, \eta)
\]
\[
=  \frac{\xi\cdot(\xi-2\eta)}{8 |\xi|^2}\frac{i(|\xi|+|\xi-\eta|+|\eta|)}{2|\xi-\eta||\eta|} \Big((\xi-\eta)\times\eta\Big) 
+\Big( \frac{-i(|\xi|+|\xi-\eta|+|\eta|)}{2(\cos(\angle(\xi-\eta, \eta))-1)} \frac{\xi}{16\pi|\xi|}\cdot \Big(\frac{\xi-\eta}{|\xi-\eta|}-\frac{\eta}{|\eta|}\Big)\Big) \times \]
\[
\Big(\frac{ (\xi - \eta)}{|\xi-\eta|} \times \frac{\eta}{|\eta|}\Big)
= \underbrace{\frac{i\xi\cdot(\xi-2\eta)(|\xi|+|\xi-\eta|+|\eta|) }{16 |\xi|^2}\Big(\frac{(\xi-\eta)}{|\xi-\eta|}\times\frac{\eta}{|\eta|}\Big) }_{\text{Part I: A regular symbol}}
\]
\[
+\underbrace{\frac{-i(|\xi|+|\xi-\eta|+|\eta|)\xi_1}{16 |\xi|} \frac{\cos(\angle(\xi-\eta,(1,0)))-\cos(\angle(\eta, (1,0)))}{(\cos(\angle(\xi-\eta, \eta))-1)} \Big(\frac{ (\xi - \eta)}{|\xi-\eta|} \times \frac{\eta}{|\eta|}\Big)}_{\text{Part II}}
\]
\[
+ \underbrace{\frac{-i(|\xi|+|\xi-\eta|+|\eta|)\xi_2}{16 |\xi|} \frac{\sin(\angle(\xi-\eta,(1,0)))- \sin(\angle(\eta,(1,0)))}{(\cos(\angle(\xi-\eta, \eta))-1)} \Big(\frac{ (\xi - \eta)}{|\xi-\eta|} \times \frac{\eta}{|\eta|}\Big)}_{\text{Part III}},
\]
where $\xi_{j}, j\in\{1,2\}$, is the $j$-th component of vector $\xi$ and we used the following fact in above computation,
\[
\frac{\xi-\eta}{|\xi-\eta|}-\frac{\eta}{|\eta|} = (\cos(\angle(\xi-\eta,(1,0)))-\cos(\angle(\eta, (1,0))), \sin(\angle(\xi-\eta,(1,0)))- \sin(\angle(\eta,(1,0))) ).
\]
It remains to check ``
Part II" and ``Part III".  Using the Step $2$, we rewrite them as follows, 
\[
\textup{Part II} = \underbrace{\frac{-i(|\xi|+|\xi-\eta|+|\eta|)\xi_1}{16 |\xi|}}_{\text{A regular symbol}}  \underbrace{f( \frac{\eta}{|\eta|}\times \frac{\xi-\eta}{|\xi-\eta|}), (1,0)\times \frac{\eta}{|\eta|})}_{\text{Angular part}}, 
\]
\[
\textup{Part III} = \underbrace{\frac{-i(|\xi|+|\xi-\eta|+|\eta|)\xi_2}{16 |\xi|}}_{\text{A regular symbol}}  \underbrace{g( \frac{\eta}{|\eta|}\times \frac{\xi-\eta}{|\xi-\eta|}), (1,0)\times \frac{\eta}{|\eta|})}_{\text{Angular part}}, 
\]
where
\[
f(x,y)= \frac{-x\big(\cos(\sin^{-1}(y)+\sin^{-1}(x))- \cos(\sin^{-1}(y)) \big)}{(\cos(\sin^{-1}(x))-1)}, \]
\[g(x,y)= \frac{-x\big(\sin(\sin^{-1}(y)+\sin^{-1}(x))- y) \big)}{(\cos(\sin^{-1}(x))-1)}, \]
and $ x, y \in[-1,1]$. We have the following expansions when $x, y$ are very close to $0$,
\[ f(x,y)= x+ 2y -\frac{y^2x+x^2 y}{2} + o(x^4) + o(y^4), \quad \textup{when $|x|,|y|\ll 1$},
\]
\[ g(x,y)= -2 + y^2 + xy+\frac{4x^2-2x^2y^2+x^4}{8}+\mathcal{O}(x^4) +\mathcal{O}(y^4), \quad\textup{when $|x|,|y|\ll 1$}.
\]

We first use the rules in Step $1$ to find out the independent variables and then use the Chain rule and Leibniz's rule, as a result, we can see that the ``Angular parts" of ``Part II" and ``Part III"
are also  regular symbols. From above discussion, it's easy to see the following estimate holds 
\[
\big| a_{+,+}(\xi-\eta,\eta)\psi_k(\xi)\psi_{k_1}(\xi-\eta)\psi_{k_2}(\eta)\big| \lesssim 2^{\max\{k_1,k_2\}}.
\]
Then, we can check the derivatives of $a_{+,+}(\xi-\eta, \eta)$ with respect to the independent variables, from Lemma \ref{Snorm}, eventually the following estimate holds,
\[
\| a_{+,+}(\xi-\eta, \eta)\|_{\mathcal{S}_{k,k_1,k_2}^\infty}\lesssim 2^{\max\{k_1,k_2\}}.
\]

We can perform similar analysis for  all other symbols and have the following lemma, 
\begin{lemma}\label{sizenormal}
For any admissible $k, k_1, k_2\in\mathbb{Z}$, we have the following estimate for any $(\mu, \nu)\in \mathcal{S}$,
\begin{equation}\label{sizeofnormal1}
\| a_{\mu,\nu}(\xi-\eta, \eta)\|_{\mathcal{S}^\infty_{k,k_1,k_2}} \lesssim 2^{\max\{k_1,k_2\}}, \quad \| a_{-,-}(\xi-\eta, \eta)\|_{\mathcal{S}^\infty_{k,k_1,k_2}} \lesssim 2^{\min\{k_1,k_2\}}
\end{equation}
\begin{equation}\label{equation12001}
\| a_{+,-}(\xi-\eta, \eta)\|_{\mathcal{S}^\infty_{k,k_1,k_2}} \lesssim 2^{k_1}, \quad \textup{if $k_1\leq k_2-5$},
\end{equation}
\begin{equation}\label{equation12000}
\| q(\xi, \eta)\|_{\mathcal{S}^\infty_{k,k_1,k_2}} \lesssim 2^{k_1+k_2},\quad \| \widetilde{q}(\xi, \eta)\|_{\mathcal{S}^\infty_{k,k_1,k_2}} \lesssim 2^{\min\{k_1, k_2\}},
\end{equation}
where $q(\cdot, \cdot)$ is the symbol of bilinear operator $Q\in \{Q_{0,\mu}, Q_{\mu,\nu}, 
\tilde{Q}_{0,0}, \tilde{Q}_{0,\mu}, \tilde{Q}_{\mu,\nu}\}$ and $\tilde{q}(\cdot, \cdot)$ is the symbol of  bilinear operator  $\tilde{Q} 
\in \{ \tilde{Q}_{0}^{1}, \tilde{Q}_{0,\mu}^{1}, \tilde{Q}_{\mu, \nu}^{1} \} $.

\end{lemma}
\begin{proof}
The desired estimates (\ref{sizeofnormal1}), (\ref{equation12001}), and (\ref{equation12000}) follows from Lemma \ref{Snorm} and above discussion. We mention that for estimate (\ref{equation12001}) and the second estimate of (\ref{sizeofnormal1}), we used the facts that $|\xi|+|\xi-\eta|+|\eta|$ is always big and $|\xi|-|\xi-\eta|+|\eta|$ is not small when $|\xi-\eta|\ll |\eta|$. More precisely, from Lemma \ref{Snorm}, the following estimate holds, 
\[
\| \frac{1}{|\xi|+|\xi-\eta|+|\eta|}\|_{\mathcal{S}^\infty_{k,k_1,k_2}} \lesssim 2^{-\max\{k_1,k_2\}}, 
\]
\[
\| \frac{1}{|\xi|-|\xi-\eta|+|\eta|}\|_{\mathcal{S}^\infty_{k,k_1,k_2}} \lesssim 2^{-\max\{k_1,k_2\}},\quad \textup{if $k_1\leq k_2-5$}.
\]
Combining above estimates with (\ref{equation12000}), from (\ref{productsymbol}) in Lemma \ref{boundness}, we can see these improved estimates hold.
\end{proof}

\subsection{The usual energy estimate}

 To find out what cubic correction terms to add, we first do the usual energy estimate. Due to the quasilinear nature of the system (\ref{maineqn}),   we have to avoid losing derivatives when doing energy estimate.\subsubsection{Energy in terms of $\phi_0$ and $\Phi$.}
We define the usual energy as follows, 
\begin{equation}\label{equation1100}
E(t):=  E^{N_0}(t)+E^{N_1}(t)  + E^{0}(t), \quad E^{N_0}(t):= \sum_{k+j=N_0, 0\leq k,j \in \mathbb{Z}}\frac{1}{2}\big[ \int_{\R} |\p_1^{k}\p_2^{j}\phi_0|^2 +  |\p_1^{k}\p_2^{j}\Phi|^2 \big],
\end{equation}
\[
 E^{N_1}(t):=  \sum_{k+j=N_1, 0\leq k,j \in \mathbb{Z}}\frac{1}{2}\big[ \int_{\R} |\p_1^{k}\p_2^{j} S\phi_0|^2 +  |\p_1^{k}\p_2^{j}S\Phi|^2 +|\p_1^{k}\p_2^{j} \Omega \phi_0|^2 +  |\p_1^{k}\p_2^{j}\Omega \Phi|^2 \big], 
\]
\begin{equation}\label{equation10010}
E^{0}(t):= \frac{1}{2}\Big[\int_{\R} |\phi_0|^2 + |\Phi|^2 + |S\phi_0|^2 +|\Omega \phi_0|^2+ |S\Phi|^2 + |\Omega \Phi|^2\Big]. 
\end{equation}
For  a tuple of nonnegative integers $\alpha=(\alpha_1, \alpha_2, \alpha_3, \alpha_4), |\alpha| \geq 0$ and $|\alpha_3|+|\alpha_4|\leq 1$,  we use $\Gamma^{\alpha}$ to denote $\p_{1}^{\alpha_1}\p_2^{\alpha_2} S^{\alpha_3}\Omega^{\alpha_4}$ and use $f^\alpha$ to denote $\Gamma^\alpha f$ for a well defined function $f$. For a bilinear term $T(f,g)$, we use
\[
T^{(\beta, \gamma)}(f,g):= T(f^{\beta}, g^{\gamma}) + T(f^{\gamma}, g^{\beta}), \]
to denote the terms that one of the inputs is hit by $\Gamma^\beta$ and the other input is hit by $\Gamma^\gamma$. As
\begin{equation}\label{eqn1203}
[S, \p_{t} + i |\nabla| ] =- (\p_{t} + i |\nabla|),\quad [\Omega, \p_{t}  + i |\nabla|] = 0, 
\end{equation}
after applying $\Gamma^\alpha,\Gamma^{\alpha}S,$ and $ \Gamma^{\alpha}\Omega$
to the system of equations (\ref{maineqn}), we can derive equations satisfied by $\phi^\alpha_0 \in\{ \Gamma^\alpha \phi_0, \Gamma^{\alpha} S \phi_{0}, \Gamma^{\alpha}\Omega \phi_{0}\}$ and $  \Phi^\alpha\in \{ \Gamma^\alpha \Phi, \Gamma^{\alpha}S\Phi, \Gamma^{\alpha} \Omega \Phi \}$ as follows,
\begin{equation}\label{maineqn6}
 \left\{\begin{array}{ll}
  \p_t \phi^\alpha_0 = \mathcal{N}^{(\alpha,0)}_0  +  \mathbf{Err}_0^{\alpha} &\\
&\\
\p_t \Phi^\alpha + i |\nabla| \Phi^\alpha =  \mathcal{N}^{(\alpha,0)}_1 + \mathbf{Err}_1^{\alpha}, & \\
 \end{array}\right.
\end{equation}
where 
$\mathbf{Err}_0^{\alpha} $ and $ \mathbf{Err}_1^{\alpha} $ are  good error terms, which  consist of terms in which two inputs are not hit by the entire $\Gamma^\alpha$
 derivative and the commutator terms if ``$S$" or ``$\Omega$" is applied. More precisely, for $i\in\{0,1\}$,
 \begin{equation}\label{equation1480}
 \mathbf{Err}_{i}^{\alpha}= \sum_{|\gamma|\leq |\beta|< |\alpha|, \beta +\gamma =\alpha} { \alpha \choose \beta}\mathcal{N}^{(\beta,\gamma)}_{i} + \textup{commutator terms if $|\alpha_3|+|\alpha_4|=1$},\quad { \alpha \choose \beta} := \prod_{j=1}^4{ \alpha_j \choose \beta_j}.
 \end{equation}

We mention that  we need to utilize the  commutation rules for the vector fields to derive the system (\ref{maineqn6}). For readers' conveniences, we derive and discuss   those commutations rules before ending this subsubsection.

For  two smooth well defined functions $h_{1}$ and $ h_{2}$ and a bilinear operator $Q(\cdot,\cdot)$ with symbol $q(\xi,\eta)$, which is  homogeneous of degree $c$, i.e.,
\begin{equation}\label{equation1220}
q(\lambda\xi, \lambda\eta)= \lambda^c q(\xi, \eta),
\end{equation}
 we have 
 \begin{equation}\label{equation1021}
S Q( h_{1}, h_{2}) = Q( S h_{1}, h_{2} ) + Q( h_{1}, S h_{2}) - c Q(h_{1}, h_{2}), \quad c \in \{1, 2\}.
 \end{equation}
 To prove (\ref{equation1021}), it is  sufficient to  consider the `$x\cdot \nabla $' part of scaling vector `$S$' as `$t\p_{t}$' part of `$S$' distributes as usual derivatives. We have
\[
\mathcal{F} \big( x\cdot \nabla Q(h_{1}, h_{2}) \big) (\xi) =  [-\xi \cdot \nabla_{\xi}-2I] \Big( \int_{\R} q(\xi, \eta) \widehat{h}_{1}(\xi-\eta) \widehat{h}_{2} (\eta) \, \, d \, \eta \Big)
\]

\[
=  \int_{\R} [(-\xi \cdot \nabla_{\xi}-2)q(\xi, \eta)] \widehat{h}_{1}(\xi-\eta)  \widehat{h}_{2} (\eta) \, d \, \eta + 
\]
\[  \int_{\R} q(\xi, \eta) \eta \cdot \nabla_{\eta} \widehat{h}_{1}(\xi-\eta)  \widehat{h}_{2} (\eta) \, d \, \eta -  \int_{\R} q(\xi, \eta) (\xi-\eta)\cdot \nabla_{\xi} \widehat{h}_{1}(\xi-\eta)  \widehat{h}_{2} (\eta) \, d \, \eta 
\]

\[
= \int_{\R} q(\xi, \eta) [- (\xi-\eta)\cdot \nabla_{\xi} -2]\widehat{h}_{1}(\xi-\eta)  \widehat{h}_{2} (\eta) \, d \, \eta +  \int_{\R} q(\xi, \eta)  \widehat{h}_{1}(\xi-\eta)  [-\eta \cdot \nabla_{\eta} -2] \widehat{h}_{2} (\eta) \, d \, \eta 
\]

\begin{equation}\label{equation1022}
- \int_{\R} ( \xi \cdot \nabla_{\xi} + \eta \cdot \nabla_{\eta} ) q(\xi, \eta) \widehat{h}_{1}(\xi-\eta)  \widehat{h}_{2} (\eta)  \, d \, \eta.
\end{equation}
After taking one derivative with respect to $\lambda$ for identity  (\ref{equation1220}) and evaluating it at $\lambda =1$, we have
\begin{equation}
\xi \cdot \nabla_{\xi} q(\xi, \eta) + \eta \cdot \nabla_{\eta} q (\xi, \eta) = c q (\xi, \eta). 
\end{equation}
Therefore, 
\begin{equation}
x\cdot \nabla Q(h_{1}, h_{2}) = Q (x\cdot \nabla h_{1}, h_{2}) + Q(h_{1}, x\cdot \nabla h_{2}) - c\,Q(h_{1}, h_{2}),
\end{equation}
which implies  that (\ref{equation1021}) holds.  Very similarly, we have the following identity for the rotational vector field,
\begin{equation}\label{equation1023}
\Omega Q( h_{1}, h_{2}) = Q(\Omega h_{1}, h_{2}) + Q(h_{1}, \Omega h_{2}) - Q' (h_{1}, h_{2}),
\end{equation}
where bilinear operator $Q'(\cdot, \cdot)$ is defined by the following symbol, 
\begin{equation}\label{equation1362}
q' (\xi ,\eta ) = \xi^{\bot} \cdot \nabla_{\xi} q (\xi, \eta) + \eta^{\bot} \cdot \nabla_{\eta} q (\xi, \eta).
\end{equation}
The size of symbol $q'(\xi, \eta)$ is comparable to size of $q(\xi ,\eta)$. Moreover, $q'(\xi,\eta)$ has null structure as long as $q(\xi,\eta)$ has null structure. To see this point, we only have to check the case when both $\nabla_{\xi}$ and $\nabla_{\eta}$ hits the angular part, for example,  $(\xi-\eta)\times \eta $, we have\begin{equation}
\xi^{\bot} \cdot \nabla_{\xi } \big( (\xi-\eta)\times \eta\big) + \eta^{\bot} \cdot \nabla_{\eta}  \big( (\xi-\eta)\times \eta\big) = \xi\cdot \eta - \xi\cdot \eta =0,
\end{equation}
which infers that bilinear operator $Q'(\cdot, \cdot)$ also has two degrees of angle inside.

  From (\ref{equation1021}) and (\ref{equation1023}), we know that, modulo the good error commutator terms,  we can distribute the vector fields $S$ and $\Omega$ as usual derivatives. 
  
  We remark that commutator terms come from two sources: (i) from the commutation rules in (\ref{eqn1203}); (ii) from the commutation rules in (\ref{equation1021}) and (\ref{equation1023}). It is safe to put those commutator terms into the error terms,
 because the commutator terms only depend on the $\phi_0$ and $\Phi$ and their  top regularities are all $N_0$, which is much bigger than $N_1$.  
 
\subsubsection{The usual energy estimate} Recall the definition of $E^{N_0}(t)$ and the system of equations (\ref{maineqn6}), we have
\begin{equation}\label{equation1230}
 \frac{d}{d t} E^{N_0} (t) =  \sum_{\begin{subarray}{l}
 |\alpha|= N_{0}, \alpha_3=\alpha_4=0\\
 \end{subarray}} \textup{Re} \Big( \int \overline{\phi^\alpha_0} \mathbf{Err}_0^{\alpha}  + \overline{\Phi^\alpha} \mathbf{Err}_1^{\alpha} \Big) + \textup{Re} \Big( \int \overline{\phi^\alpha_0} \mathcal{N}^{(\alpha,0)}_0   + \overline{\Phi^\alpha} \mathcal{N}^{(\alpha,0)}_1  \Big).
\end{equation}
We will first show that cancellation happens for the second integral of the right hand side of (\ref{equation1230}), hence it does not lose derivatives.

From the system of equation satisfied by $\psi$, $G_1$, and $G_2$ in (\ref{changevariables}) and (\ref{eqn400}), we have
\[
\mathcal{N}^{(\alpha,0)}_0  = R_1 \widetilde{\mathcal{N}}_{2}^{(\alpha, 0)} -R_2 \widetilde{\mathcal{N}}_{1}^{(\alpha,0)}, \quad \mathcal{N}^{(\alpha,0)}_1= \widetilde{\mathcal{N}}_{0}^{(\alpha, 0)} + i \big[R_1 \widetilde{\mathcal{N}}_{1}^{(\alpha,0)} + R_2 \widetilde{\mathcal{N}}_{2}^{(\alpha,0)}\big].
\]
\begin{equation}\label{equation62220}
\widetilde{\mathcal{N}_{0}}= Q(\psi, \psi)- Q(G_1,G_1)-Q(G_2,G_2), \quad\widetilde{\mathcal{N}}_{1}= Q_{1,2}(G_1,\psi), \quad \widetilde{\mathcal{N}}_{2}= Q_{1,2}(G_2,\psi),
\end{equation}
where bilinear operator $Q(\cdot, \cdot)$ is defined by the symbol as follows, 
\[
q(\xi-\eta, \eta) = -\frac{\xi\cdot(\xi-\eta)}{|\xi|^2} (\xi-\eta)\times \eta.
\]

 We first reformulate the second integral of (\ref{equation1230}) in terms of $\psi$,$G_1$, and $G_2$, because  operators in (\ref{equation62220}) are much easier than operators in (\ref{mainequation10}) (computations are less involving) and we can use the fact that  $\psi $, $G_1,$ and $G_2$
are all real. More precisely, from (\ref{changevariables}), we have
\[
\textup{Re} \Big( \int \overline{\phi^\alpha_0} \mathcal{N}^{(\alpha,0)}_0   + \overline{\Phi^\alpha} \mathcal{N}^{(\alpha,0)}_1  \Big) = \int \psi^{\alpha} \widetilde{\mathcal{N}}_{0}^{(\alpha, 0)} + G_1^{\alpha} \widetilde{\mathcal{N}}_{1}^{(\alpha,0)} + G_2^{\alpha} \widetilde{\mathcal{N}}_{2}^{(\alpha,0)}  = \int \psi^{\alpha}[Q(\psi^{\alpha}, \psi) 
\]
\begin{equation}\label{equation9000}
+ Q(\psi, \psi^{\alpha})]+\sum_{i=1,2} G_i^\alpha Q_{1,2}(G_i^{\alpha}, \psi)
+  \psi^{\alpha}\big[-Q(G_i^\alpha, G_i) - Q(G_i,G_i^{\alpha})\big] 
+ G_i^{\alpha} Q_{1,2}(G_i, \psi^{\alpha}),
\end{equation}
After utilizing symmetries on the Fourier side to switch the role of $\xi$ and $\xi-\eta$, we have
\[
(\ref{equation9000})= 
\sum_{i=1,2}\int \int \Big[ \overline{\widehat{G_i^\alpha}(\xi) }\widehat{G_i^\alpha}(\xi-\eta)\widehat{\psi}(\eta) q_1(\xi-\eta, \eta) + \overline{\widehat{\psi^{\alpha}}(\xi)}\widehat{G_i^\alpha}(\xi-\eta)\widehat{G_i}(\eta) q_2(\xi-\eta)\Big] d \xi d \eta  
\]
\begin{equation}\label{equation61220}
+  \int\int  \overline{\widehat{\psi^{\alpha}}(\xi)}\widehat{\psi^{\alpha}}(\xi-\eta)\widehat{\psi}(\eta) q_3(\xi-\eta, \eta) d\xi d \eta,
\end{equation}
where
\[
q_1(\xi-\eta, \eta)= \big[-(\xi-\eta)\times (\eta) - \xi \times (-\eta)\big]/2= 0,\quad
q_2(\xi-\eta, \eta)= -(-\eta)\times \xi- q(\xi-\eta, \eta)-q(\eta, \xi-\eta)
\]
\begin{equation}\label{equation1241}
=\eta \times \xi+ \frac{\xi\cdot(\xi-2\eta)}{ |\xi|^2} (\xi-\eta)\times \eta= \frac{-2\xi\cdot \eta}{|\xi|^2} (\xi-\eta)\times \eta,
\end{equation}
\[
q_3(\xi-\eta, \eta) =  [q(\xi-\eta, \eta)+q(\eta, \xi-\eta)]/2 + [q(\xi, -\eta) + q(-\eta, \xi)]/2\]
\begin{equation}\label{equation1008}
= \big( \frac{\xi\cdot \eta}{ |\xi|^2}+ \frac{(\xi-\eta)\cdot \eta}{ |\xi-\eta|^2}\big) (\xi-\eta)\times \eta= -\big(q_2(\xi-\eta, \eta)+ q_2(\xi, -\eta))/2.
\end{equation}
From Lemma \ref{Snorm}, the following estimate holds from above explicit formulas, 
\begin{equation}\label{equation1483}
\| q_2(\xi-\eta, \eta)\|_{\mathcal{S}^{\infty}_{k,k_1,k_2}} + \| q_3(\xi-\eta, \eta)\|_{\mathcal{S}^{\infty}_{k,k_1,k_2}} \lesssim 2^{2k_2}.
\end{equation}
After replacing $\psi$, $G_1$, and $G_2$ by $\phi_0$ and $\Phi$ through (\ref{equation143}), we have 
\[
(\ref{equation61220}) = \int \sum_{i=1,2} \psi^\alpha Q_2(G_i^\alpha, G_i) + \psi^\alpha Q_3(\psi^\alpha, \psi)= 
\textup{Re}\Big( \int \sum_{\mu, \nu, \kappa}\frac{1}{8} \Phi^{\alpha}_{\mu}Q_3(\Phi_{\nu}^{\alpha}, \Phi_{\kappa})+\sum_{i=1,2} \frac{1}{8}(\Phi^{\alpha} 
\]
\[
+\overline{\Phi^\alpha}) Q_{2}\big(\Gamma^\alpha((-1)^{i}2R_{3-i} \phi_0 + c_{+} R_i \Phi +  c_{-} R_i\overline{\Phi}),(-1)^{i}2R_{3-i} \phi_0 +  c_{+}R_i\Phi+c_{-}R_i \overline{\Phi} ) \Big),
\]
where bilinear operators $Q_2(\cdot, \cdot)$ and $Q_3(\cdot, \cdot)$ are defined by the symbols $q_2(\cdot, \cdot)$ and $q_3(\cdot, \cdot)$ respectively.  To sum up, we have
\[
\frac{d}{d t} E^{N_0} (t) =  \sum_{\begin{subarray}{l}
 |\alpha|= N_{0}\\
 \alpha_3=\alpha_4=0\\
 \end{subarray}} \textup{Re} \Big( \int \overline{\phi^\alpha_0} \mathbf{Err}_0^{\alpha}  + \overline{\Phi^\alpha} \mathbf{Err}_1^{\alpha} \Big) + \textup{Re}\Big( \int \sum_{\mu, \nu, \kappa\in\{+,-\}}\frac{1}{8} \Phi^{\alpha}_{\mu}Q_3(\Phi_{\nu}^{\alpha}, \Phi_{\kappa}) \]
\begin{equation}\label{eqn1009}
+\sum_{i=1,2} \frac{1}{8}(\Phi^{\alpha} + \overline{\Phi^\alpha})Q_{2}\big(\Gamma^\alpha((-1)^{i}2R_{3-i} \phi_0 +  c_{+}R_i\Phi +  c_{-} R_i\overline{\Phi}),(-1)^{i}2R_{3-i} \phi_0 +  c_{+}R_i\Phi+  c_{-} R_i\overline{\Phi} ) \Big).
\end{equation}

With minor modifications, we can perform the same procedure for $E^{N_1}(t)$ and have the following, 
\[
 \frac{d}{d t} E^{N_1} (t) =  \sum_{\begin{subarray}{l}
 |\alpha|= N_{1}+1\\
 \alpha_3+\alpha_4=1\\
 \end{subarray}} \textup{Re} \Big( \int \overline{\phi^\alpha_0} \mathbf{Err}_0^{\alpha}  + \overline{\Phi^\alpha} \mathbf{Err}_1^{\alpha} \Big)+ \textup{Re}\Big( \int \sum_{\mu, \nu, \kappa\in\{+,-\}}\frac{1}{8} \Phi^{\alpha}_{\mu}Q_3(\Phi_{\nu}^{\alpha}, \Phi_{\kappa})
 \]
 \begin{equation}\label{equation1003}
+\sum_{i=1,2} \frac{1}{8}(\Phi^{\alpha} + \overline{\Phi^\alpha})Q_{2}\big(\Gamma^\alpha((-1)^{i}2R_{3-i} \phi_0 +  c_{+}R_i\Phi +  c_{-}R_i \overline{\Phi}),(-1)^{i}2R_{3-i} \phi_0 +  c_{+} R_i\Phi +  c_{-} R_i\overline{\Phi} ) \Big).
\end{equation}

\subsection{Identifying the worst cubic terms} The worst cubic terms inside the derivative of energy $E(t)$, which decay slowly, only depend on $\Phi$.  We identify them in this section.  From (\ref{equation10010}), we have 
\[
\textup{ worst cubic terms of } \frac{d}{d t} E^0(t) = \sum_{(\mu, \nu)\in \mathcal{S}} \textup{Re} \Big(\int \overline{\Phi} \widetilde{Q}_{\mu, \nu}(\Phi_{\mu}, \Phi_{\nu})+\overline{\Omega\Phi} \big[\widetilde{Q}_{\mu, \nu}(\Omega\Phi_{\mu}, \Phi_{\nu})+ \widetilde{Q}_{\mu, \nu}(\Phi_{\mu}, \Omega\Phi_{\nu}) \]
\begin{equation}\label{equation1248}
-\widetilde{Q}'_{\mu, \nu}(\Phi_{\mu}, \Phi_{\nu})\big]+
 \overline{S\Phi}\big[\widetilde{Q}_{\mu, \nu}(S\Phi_{\mu}, \Phi_{\nu}) + \widetilde{Q}_{\mu, \nu}(\Phi_{\mu}, S\Phi_{\nu})-\widetilde{Q}_{\mu, \nu}(\Phi_{\mu}, \Phi_{\nu})\big]\Big),
\end{equation}
where the bilinear operator $\widetilde{Q}'_{\mu, \nu}(\cdot, \cdot)$, follows from (\ref{equation1362}), is defined by the following symbol 
\[
(\xi^{\bot}\cdot \nabla_{\xi} + \eta^{\bot}\cdot \nabla_{\eta}) \tilde{m}'_{\mu, \nu}(\xi-\eta ,\eta).
\]
From (\ref{eqn1009}),  we have
\[
\textup{worst cubic terms of } \frac{d}{d t} E^{N_0}(t) = \sum_{\begin{subarray}{l}
 |\alpha|= N_{0}, \alpha_3=\alpha_4=0\\
 \end{subarray}} \sum_{|\gamma|\leq |\beta|< |\alpha|, \beta +\gamma =\alpha} { \alpha \choose \beta}\sum_{(\mu, \nu)\in \mathcal{S}} \textup{Re} \Big( \int \overline{\Phi^\alpha} \widetilde{Q}_{\mu, \nu}(\Phi_{\mu}^{\beta}, \Phi_{\nu}^{\gamma}) \Big)
\]
\[
+\sum_{\mu, \nu, \kappa\in\{+,-\}} \frac{1}{8}\textup{Re} \Big(\int  \sum_{i=1,2}  c_{\nu} c_{\kappa}
\Phi_{\mu}^{\alpha} Q_2(R_i\Phi_{\nu}^{\alpha}, R_i \Phi_{\kappa}) +  \Phi_{\mu}^{\alpha} Q_3(\Phi_{\nu}^{\alpha}, \Phi_{\kappa})\Big).
\]
From the identity (\ref{equation1008}), we have
\[
\sum_{\mu, \nu, \kappa\in\{+,-\}} \frac{1}{8}\textup{Re} \Big(\int  \sum_{i=1,2}  c_{\nu} c_{\kappa}
\Phi_{\mu}^{\alpha} Q_2(R_i\Phi_{\nu}^{\alpha}, R_i \Phi_{\kappa}) +  \Phi_{\mu}^{\alpha} Q_3(\Phi_{\nu}^{\alpha}, \Phi_{\kappa})\Big)\]
\[= \sum_{\mu, \nu, \kappa\in\{+,-\}} \frac{1}{8}\textup{Re} \Big(\int  \sum_{i=1,2}  c_{\nu}
 c_{\kappa}\Phi_{\mu}^{\alpha} Q_2(R_i\Phi_{\nu}^{\alpha}, R_i \Phi_{\kappa}) -  \Phi_{\mu}^{\alpha} Q_2(\Phi_{\nu}^{\alpha}, \Phi_{\kappa})/2-\Phi_{\nu}^{\alpha} Q_2(\Phi_{\mu}^{\alpha}, \Phi_{-\kappa})/2\Big)\]
 \[
=  \sum_{\mu, \nu, \kappa\in\{+,-\}}  \frac{1}{8}
\textup{Re} \Big(\int\sum_{i=1,2}  c_{\nu} c_{\kappa}
\Phi_{\mu}^{\alpha} Q_2(R_i\Phi_{\nu}^{\alpha}, R_i \Phi_{\kappa})-\Phi_{\mu}^{\alpha} Q_2(\Phi_{\nu}^{\alpha}, \Phi_{\kappa})\Big) \]
\begin{equation}\label{equation1244}
= \sum_{\mu, \nu, \kappa\in\{+,-\}}  \frac{1}{8} \textup{Re} \Big(\int \Phi_{\mu}^{\alpha} Q^4_{\nu, \kappa}(\Phi_{\nu}^{\alpha},  \Phi_{\kappa})\Big)=  \sum_{ \nu, \kappa\in\{+,-\}} \frac{1}{4} \textup{Re} \Big(\int \overline{\Phi^{\alpha}} Q^4_{\nu, \kappa}(\Phi_{\nu}^{\alpha},  \Phi_{\kappa})\Big),
\end{equation}
where the bilinear operator  $Q_4(\cdot, \cdot)$ is defined by the following symbol,
\begin{equation}\label{equation1352}
q^4_{\nu, \kappa}(\xi-\eta, \eta)= -q_2(\xi-\eta, \eta)\big( 1+c_{\nu}c_{\kappa} \frac{(\xi-\eta)\cdot \eta}{|\xi-\eta||\eta|} \big)= -q_2(\xi-\eta, \eta)(1-a_{\nu}a_{\kappa}\cos(\angle(\xi-\eta, \eta)).
\end{equation}
To sum up, we have
\[
\textup{worst cubic terms of  } \frac{d}{d t} E^{N_0}(t) = \sum_{\begin{subarray}{l}
 |\alpha|= N_{0}\\
  \alpha_3=\alpha_4=0\\
 \end{subarray}} \sum_{\begin{subarray}{l}
 |\beta|, |\gamma|< |\alpha|\\
 \beta + \gamma=\alpha\\
 \end{subarray}} { \alpha \choose \beta}\sum_{(\mu, \nu)\in \mathcal{S}} \textup{Re} \Big( \int \overline{\Phi^\alpha} \widetilde{Q}_{\mu, \nu}(\Phi_{\mu}^{\beta}, \Phi_{\nu}^{\gamma}) \Big) 
\]
\begin{equation}\label{equation1249}
+  \sum_{ \nu, \kappa\in\{+,-\}}  \frac{1}{4} \textup{Re} \Big(\int \overline{\Phi^{\alpha}} Q^{4}_{\nu, \kappa}(\Phi_{\nu}^{\alpha},  \Phi_{\kappa})\Big).
\end{equation}
Now we  consider  the worst cubic terms   of the derivative of $E^{N_1}(t)$. Except the value of $\alpha$ is different and the presence of commutators terms, most cubic terms are same as terms in (\ref{equation1249}). Very similarly, we have
\[
\textup{worst cubic terms of  } \frac{d}{d t} E^{N_1}(t) = \sum_{\begin{subarray}{l}
 |\alpha|= N_{1}+1\\
  \alpha_3+\alpha_4=1\\
 \end{subarray}}\Big[
 \sum_{\begin{subarray}{l}
 |\gamma|, |\beta|< |\alpha|\\
  \beta +\gamma =\alpha\\
 \end{subarray}} { \alpha \choose \beta}\sum_{(\mu, \nu)\in \mathcal{S}} \textup{Re} \Big( \int \overline{\Phi^\alpha} \widetilde{Q}_{\mu, \nu}(\Phi_{\mu}^{\beta}, \Phi_{\nu}^{\gamma}) \Big)\]
\[
+  \sum_{ \nu, \kappa\in\{+,-\}}  \frac{1}{4} \textup{Re} \Big(\int 
\overline{\Phi^{\alpha}} Q^4_{\nu, \kappa}(\Phi_{\nu}^{\alpha},  \Phi_{\kappa})\Big)
 \Big]+\sum_{\begin{subarray}{l}
 |\alpha|= N_{1}\\
  \alpha_3=\alpha_4=0\\
 \end{subarray}}
 \sum_{\begin{subarray}{l}\\
  \beta +\gamma =\alpha\\
 \end{subarray}}{ \alpha \choose \beta}\sum_{(\mu, \nu)\in \mathcal{S}}  \textup{Re}\Big( \int  -\overline{\Gamma^\alpha S\Phi} \widetilde{Q}_{\mu, \nu}(\Phi_{\mu}^\beta, \Phi_{\nu}^\gamma) 
\]
\begin{equation}\label{equation1251}
-\overline{\Gamma^\alpha \Omega\Phi} \widetilde{Q}'_{\mu, \nu}(\Phi_{\mu}^\beta, \Phi_{\nu}^\gamma)\Big).
\end{equation}
To sum up, from (\ref{equation1100}), (\ref{equation1248}), (\ref{equation1249}) and (\ref{equation1251}), we have
\[
\textup{worst cubic terms of } \frac{d}{d t} E (t)=   \sum_{\begin{subarray}{l}
(|\alpha_1|+|\alpha_2|, |\alpha_3|+|\alpha_4|)\in \\
\{(N_0,0), (N_1,1),(0,0),(0,1)\}\\
 \end{subarray}} \sum_{\begin{subarray}{l}
 |\beta|, |\gamma|\leq \max\{|\alpha|-1, 1\}\\
 \beta + \gamma=\alpha\\
 \end{subarray}} { \alpha \choose \beta}\times\]
\[
\sum_{(\mu, \nu)\in \mathcal{S}} \textup{Re} \Big( \int \overline{\Phi^\alpha} \widetilde{Q}_{\mu, \nu}(\Phi_{\mu}^{\beta}, \Phi_{\nu}^{\gamma}) \Big)
+  \sum_{\begin{subarray}{l}
(|\alpha_1|+|\alpha_2|, |\alpha_3|+|\alpha_4|)\\
\in \{(N_0,0), (N_1,1)\}\\
 \end{subarray}} \sum_{ \nu, \kappa\in\{+,-\}}  \frac{1}{4} \textup{Re} \Big(\int \overline{\Phi^{\alpha}} Q^4_{\nu, \kappa}(\Phi_{\nu}^{\alpha},  \Phi_{\kappa})\Big)
\]
\begin{equation}\label{equation1300}
+\sum_{\begin{subarray}{l}
 |\alpha|\in\{0,N_{1}\}\\
  \alpha_3=\alpha_4=0\\
 \end{subarray}}
 \sum_{\begin{subarray}{l}\\
  \beta +\gamma =\alpha\\
 \end{subarray}}{ \alpha \choose \beta} \sum_{(\mu, \nu)\in \mathcal{S}}  \textup{Re}\Big( \int  -\overline{\Gamma^\alpha S\Phi} \widetilde{Q}_{\mu, \nu}(\Phi_{\mu}^\beta, \Phi_{\nu}^\gamma) -\overline{\Gamma^\alpha \Omega\Phi} \widetilde{Q}'_{\mu, \nu}(\Phi_{\mu}^\beta, \Phi_{\nu}^\gamma)\Big).
\end{equation}
\subsection{Construction of the modified energy}\label{constructionofmodified}
\subsubsection{First level correction.}
 From the construction of normal form transformation in subsection \ref{normalform}, to cancel out (\ref{equation1300}), it is sufficient to add the following cubic terms to the usual energy, 
\[
E_{\textit{FCorr}}(t)=   \sum_{\begin{subarray}{l}
(|\alpha_1|+|\alpha_2|, |\alpha_3|+|\alpha_4|)\in \\
\{(N_0,0), (N_1,1),(0,0),(0,1)\}\\
 \end{subarray}} \sum_{\begin{subarray}{l}
 |\beta|, |\gamma|\leq \max\{|\alpha|-1,1\}\\
 \beta + \gamma=\alpha\\
 \end{subarray}} { \alpha \choose \beta}\sum_{(\mu, \nu)\in \mathcal{S}} \textup{Re} \Big( \int \overline{\Phi^\alpha} A_{\mu, \nu}(\Phi_{\mu}^{\beta}, \Phi_{\nu}^{\gamma}) \Big)
\]
\[
+  \sum_{\begin{subarray}{l}
(|\alpha_1|+|\alpha_2|, |\alpha_3|+|\alpha_4|)\\
\in \{(N_0,0), (N_1,1)\}\\
 \end{subarray}} \sum_{ \nu, \kappa\in\{+,-\}}  \frac{1}{4} \textup{Re} \Big(\int \overline{\Phi^{\alpha}} A^4_{\nu, \kappa}(\Phi_{\nu}^{\alpha},  \Phi_{\kappa})\Big)
\]
\begin{equation}\label{equation1309}
+\sum_{\begin{subarray}{l}
 |\alpha|\in\{0,N_{1}\}\\
  \alpha_3=\alpha_4=0\\
 \end{subarray}}
 \sum_{\begin{subarray}{l}\\
  \beta +\gamma =\alpha\\
 \end{subarray}}{ \alpha \choose \beta} \sum_{(\mu, \nu)\in \mathcal{S}}  \textup{Re}\Big( \int  -\overline{\Gamma^\alpha S\Phi} A_{\mu, \nu}(\Phi_{\mu}^\beta, \Phi_{\nu}^\gamma) -\overline{\Gamma^\alpha \Omega\Phi} A'_{\mu, \nu}(\Phi_{\mu}^\beta, \Phi_{\nu}^\gamma)\Big),
\end{equation}
where the bilinear operator $A'_{\mu, \nu}(\cdot,\cdot)$ and the bilinear operator $A^{4}_{\nu, \kappa}(\cdot, \cdot)$ are defined by the symbol $a'_{\mu, \nu}(\cdot, \cdot)$ and $a^4_{\nu, \kappa}(\xi-\eta, \eta)$ as follows,
\[
a'_{\mu, \nu}(\xi-\eta, \eta)= \frac{i(\xi^{\bot}\cdot \nabla_{\xi} + \eta^{\bot}\cdot \nabla_{\eta}) \tilde{m}'_{\mu, \nu}(\xi-\eta ,\eta)}{|\xi| - a_{\mu} |\xi-\eta| - a_{\nu}|\eta|},
\]
\begin{equation}\label{equation1502}
a^4_{\nu, \kappa}(\xi-\eta, \eta)= \frac{iq^4_{\nu, \kappa}(\xi-\eta,\eta)}{|\xi|-a_{\nu}|\xi-\eta|-a_{\kappa}|\eta|}=  \frac{i(|\xi|+a_{\nu}|\xi-\eta|+a_{\kappa}|\eta|)q_2(\xi-\eta, \eta)}{2a_{\nu}a_{\kappa}|\xi-\eta||\eta|}. 
\end{equation}
 As mentioned in the commutation part of the rotational vector field $\Omega$, symbol $(\xi^{\bot}\cdot \nabla_{\xi} + \eta^{\bot}\cdot \nabla_{\eta}) \tilde{m}'_{\mu, \nu}(\xi-\eta ,\eta)$  has two degrees of angle inside, hence $a'_{\mu, \nu}(\xi-\eta, \eta)$ is a regular symbol. Similar to the proof of Lemma \ref{sizenormal}, by Lemma \ref{Snorm},  we can derive the following estimates, 
\begin{equation}\label{equation1363}
\| a'_{\mu, \nu}(\xi-\eta, \eta)\|_{\mathcal{S}^\infty_{k,k_1,k_2}}\lesssim 2^{\max\{k_1,k_2\}}, \quad  \|  a^4_{\nu, \kappa}(\xi-\eta, \eta) \|_{\mathcal{S}^\infty_{k,k_1,k_2}}\lesssim 2^{k_2},
\end{equation}
\begin{equation}\label{equation12006}
\| a_{-,-}^4(\xi-\eta, \eta)\|_{\mathcal{S}^\infty_{k,k_1,k_2}} + \| a_{-,+}^4(\xi-\eta, \eta)\|_{\mathcal{S}^\infty_{k,k_1,k_2}} \lesssim 2^{2k_2-k_1}, \quad \textup{if $k_2\leq k_1-5$}.
\end{equation}
\subsubsection{Second level correction.} Due to the quasilinear nature, it is possible to lose a derivative after taking derivative with respect to time $t$ for $E_{\textit{FCorr}}(t)$. We first identify those problematic terms inside $d E_{\textit{FCorr}}(t)/d t$ by checking which terms are possible to lose one derivative. 

Note that inputs inside $A_{\mu, \nu}(\cdot, \cdot)$  are not hit by entire derivatives and we lose at most another one derivative after taking derivative with respect to time. That is to say, only the following terms are possible to lose one derivative after taking derivative with respect to time, 
\[
\textup{Possible problematic terms}:= \sum_{\begin{subarray}{l}
(|\alpha_1|+|\alpha_2|, |\alpha_3|+|\alpha_4|)\\
\in \{(N_0,0), (N_1,1)\}\\
 \end{subarray}} \sum_{\begin{subarray}{l}
 \beta + \gamma=\alpha,  |\gamma|=1, \\
 |\gamma_3|=|\gamma_4|=0\\
 \end{subarray}} \sum_{(\mu, \nu)\in \mathcal{S}} { \alpha \choose \beta} \textup{Re} \Big( \int \overline{\Gamma^{\gamma}\Phi^\beta}\times  \]
\[
 \big(A_{\mu, \nu}(\Phi^{\beta}_{\mu}, \Phi^{\gamma}_{\nu})+ A_{\mu, \nu}(\Phi^{\gamma}_{\mu}, \Phi^{\beta}_{\nu})\big) \Big)
+ \sum_{\nu, \kappa\in\{+,-\}} \frac{1}{4} \textup{Re} \Big(\int \overline{\Phi^{\alpha}} \big(A^4_{\nu, \kappa}(\Phi^{\alpha}_{\nu}, \Phi_{\kappa})\Big).
\]
 Recall (\ref{sizeofnormal1}), we know that $A_{-,-}(\cdot, \cdot)$ does not lose derivative. That is to say, $A_{-,-}(\p_t \Phi^{\beta}, \Phi^\gamma)$  actually does not lose derivative. Recall (\ref{equation12001}), we can see that $A_{+,-}(\Phi^{\gamma}, \p_t \Phi^\beta)$ actually does not lose derivative, because losing derivative is only relevant when $\Phi^\beta$ has relatively higher frequency and for this case $A_{+,-}(\cdot, \cdot)$ doesn't lose derivative. From (\ref{equation12006}), we can actually gain one derivative from the symbol of  $A_{-,\kappa}^{4}(\p_t \Phi^\alpha_{-}, \Phi_{\kappa})$ when $\p_t \Phi^\alpha_{-}$ has larger frequency, hence it does not lose derivative.  To sum up, we can rule out $A_{-,-}({\Phi^\beta_{-}}, {\Phi^{\gamma}_{-}})$, $A_{-,-}({\Phi^\gamma_{-}}, {\Phi^{\beta}_{-}})$,  $A_{+,-}(\Phi^\gamma, {\Phi^{\beta}_{-}})$, and $A_{-,\kappa}^4(\Phi^\alpha_{-}, \Phi_{\kappa})$. The problematic terms are given as follows,
\[
\textup{Problematic terms}:= \sum_{\begin{subarray}{l}
(|\alpha_1|+|\alpha_2|, |\alpha_3|+|\alpha_4|)\\
\in \{(N_0,0), (N_1,1)\}\\
 \end{subarray}} \sum_{\begin{subarray}{l}
 \beta + \gamma=\alpha,  |\gamma|=1, \\
 |\gamma_3|=|\gamma_4|=0\\
 \end{subarray}} { \alpha \choose \beta} \textup{Re} \Big( \int \overline{\Gamma^{\gamma}\Phi^\beta} \big(A_{+, +}(\Phi^{\beta}, \Phi^{\gamma})+ A_{+, +}( \Phi^{\gamma}, \Phi^{\beta})\]
\begin{equation}\label{equation62000}
 +A_{+, -}(\Phi^{\beta}, \overline{\Phi^{\gamma}}\big) \Big)
+  \frac{1}{4} \textup{Re} \Big(\int \overline{\Phi^{\alpha}} \big(A^4_{+, +}(\Phi^{\alpha},  \Phi) +A^4_{+, -}(\Phi^{\alpha},  \overline{\Phi}) \big)\Big).
\end{equation}

To get around the difficulty of losing another derivative after taking derivative in time, we need utilize symmetries inside the system. 
 Utilizing cancellations inside (\ref{equation12310}), (\ref{equation12311}) and (\ref{equation12005}), we can  rule out some more terms further after replacing $\Phi$ by its real part and imaginary part.  As a result,   the problematic  cubic terms inside (\ref{equation62000}) are, intuitively speaking, either of type $\p_j(|\Phi^\beta|^2) \p_j \textup{Im}(\Phi)$ or of type $|\Phi^\alpha|^2 \textup{Im}(\Phi)$. For those types cubic terms, we only need to worry about the case when $\p_t$ hits $|\Phi^\beta|^2$ or $|\Phi^\alpha|^2$.  Recall the symmetries we identified in (\ref{equation9000}). We know that there are symmetries inside $\p_t(|\phi_0^\alpha|^2 +|\Phi^\alpha|)$ and this whole term doesn't lose derivative. This observation suggests us to add cubic terms of type $\p_j(|\phi^\beta_0|^2) \p_j \textup{Im}(\Phi)$ or of type $|\phi^\alpha_0|^2 \textup{Im}(\Phi)$ to (\ref{equation62000}). Essentially speaking, $\textup{Im}(\Phi)$ plays very little role. But, due to its presence, it will complicate computations a lot. 

 To sum up, with above intuition, we define the following second level correction terms, 
\[
E_{SCorr}(t): = \sum_{\begin{subarray}{l}
(|\alpha_1|+|\alpha_2|, |\alpha_3|+|\alpha_4|)\\
\in \{(N_0,0), (N_1,1)\}\\
 \end{subarray}} \sum_{\begin{subarray}{l}
 \beta + \gamma=\alpha,  |\gamma|=1, \\
 |\gamma_3|=|\gamma_4|=0\\
 \end{subarray}} { \alpha \choose \beta} \textup{Re} \Big( \int \overline{\Gamma^{\gamma}\phi_0^\beta} \big(A_{+, +}(\phi_0^{\beta}, \Phi^{\gamma})+ A_{+, +}( \Phi^{\gamma}, \phi_0^{\beta})\]
\begin{equation}\label{equation1340}
 +A_{+, -}(\phi_0^{\beta}, \overline{\Phi^{\gamma}}\big) \Big)
+  \frac{1}{4} \textup{Re} \Big(\int \overline{\phi_0^{\alpha}} \big(A^4_{+, +}(\phi_0^{\alpha},  \Phi) +A^4_{+, -}(\phi_0^{\alpha},  \overline{\Phi}) \big)\Big).
\end{equation}
Now, we can see the correspondence of $\Phi^\alpha$($\Phi^\beta$) in (\ref{equation62000})  and $\phi_0^\alpha$($\phi_0^\beta$) in (\ref{equation1340}). The modified energy that will be used to do energy estimate is defined as follows,
\begin{equation}\label{equation1620}
 E^{modi} (t) = E(t) + E_{FCorr}(t) + E_{SCorr}(t).
\end{equation}

\begin{lemma}\label{sizeofmodifiedenergy}
Under the bootstrap assumption \textup{(\ref{smallness})}, we have
\begin{equation}\label{equation9}
\sup_{t\in[0, T]} \big|E_{FCorr}(t) \big| + \big|E_{SCorr}(t) \big| \lesssim \epsilon_0^2.
\end{equation}
\end{lemma}

\begin{proof}
Let $h_1$ and $h_2$ be  two well defined functions. From   (\ref{sizeofnormal1}) and (\ref{equation1363}) and Lemma \ref{boundness}, we have the following estimate for  a bilinear operator $A(\cdot, \cdot)\in \{A_{\mu, \nu}(\cdot, \cdot), A'_{\mu, \nu}(\cdot, \cdot)\}$,
\[
\| A(h_1, h_2)\|_{H^{s}} \lesssim \Big( \sum_{k_1\leq k_2 -10}2^{2s k_2} \| P_{k_2}[A(P_{k_1} h_1, P_{k_2}h_2)]\|_{L^2}^2 \Big)^{1/2}+ \Big( \sum_{k_2\leq k_1 -10}2^{2s k_1}\times
\]
\[
 \| P_{k_1}[A(P_{k_1} h_1, P_{k_2}h_2)]\|_{L^2}^2\Big)^{1/2} + \sum_{|k_1-k_2|\leq 10} \sum_{k\leq k_1+20} 2^{s k} \| P_k[A(P_{k_1} h_1, P_{k_2}h_2)]\|_{L^2}
\]
\[
\lesssim \Big( \sum_{k_1\leq k_2 -10}2^{2(s+1) k_2}\| P_{k_2} h_2\|_{L^2}^2\| P_{k_1} h_1\|_{L^\infty}^2\Big)^{1/2} + \Big( \sum_{k_2\leq k_1 -10}2^{2(s+1) k_1}\| P_{k_2} h_2\|_{L^\infty}^2\| P_{k_1} h_1\|_{L^2}^2\Big)^{1/2}
\]
\begin{equation}\label{equation1410}
+ \| h_2\|_{H^{s+1}}\|h_1\|_{W^{1}}\lesssim \| \| h_2\|_{H^{s+1}}\|h_1\|_{W^{1}} + \| h_1\|_{H^{s+1}}\|h_2\|_{W^{1}}.
\end{equation}
For fixed $\alpha$ such that $ (|\alpha_1|+|\alpha_2|, |\alpha_3|+|\alpha_4|)=(N_0,0)$, then from estimate (\ref{equation1363}) and Lemma \ref{Snorm}, we have
\[
\|A^4_{\mu, \nu}(\Phi_{\nu}^{\alpha},  \Phi_{\kappa})\|_{L^2} \lesssim \Big( \sum_{k_1\leq k_2-10} 2^{2N_0k_1 + 2k_2}\| P_{k_2} \Phi\|_{L^2}^2\| P_{k_1} \Phi\|_{L^\infty}^2\Big)^{1/2}\]
\begin{equation}\label{equation1470}
+ \Big( \sum_{k_2\leq k_1 -10}2^{2(N_0+1) k_2}\| P_{k_2} \Phi\|_{L^2}^2\| P_{k_1} \Phi\|_{L^\infty}^2\Big)^{1/2} + \|\Phi\|_{H^{N_0}}\|\Phi\|_{W^{1+}}\lesssim \|\Phi\|_{H^{N_0}}\|\Phi\|_{Z'}.\end{equation}
Very similarly, we have the following estimate when  $ (|\alpha_1|+|\alpha_2|, |\alpha_3|+|\alpha_4|)=(N_1,1)$,
\[
\|A^4_{+, \kappa}(\phi_0^{\alpha},  \Phi_{\kappa})\|_{L^2} + \|A^4_{\mu, \nu}(\Phi_{\nu}^{\alpha},  \Phi_{\kappa})\|_{L^2}
\]
\begin{equation}\label{equation1472}
\lesssim \big(\| S(\Phi, \phi_0)\|_{H^{N_1}} + \| \Omega(\Phi, \phi_0)\|_{H^{N_1}}\big)\|\Phi\|_{W^{N_1+2}}\lesssim \|\Phi\|_{X_{N_0}}\|\Phi\|_{Z'}.
\end{equation}
To sum up, from estimates (\ref{equation1410}), (\ref{equation1470}) and (\ref{equation1472}), we can estimate terms  inside (\ref{equation1309}) and (\ref{equation1340}) one by one. As a result, we have
\begin{equation}
\big| E_{FCorr}(t)\big| + \big| E_{SCorr}(t)\big| \lesssim \| (\phi_0, \Phi)\|_{X_{N_0}}^2 \| \Phi\|_{Z'} \lesssim (1+t)^{-1/2+2p_0}\epsilon_1^3\lesssim \epsilon_0^2. 
\end{equation}
Therefore, our desired estimate (\ref{equation9}) holds.
\end{proof}
\begin{remark}
Since above dyadic decomposition and multilinear  type estimates are standard processes, we will not repeat the detail proofs of similar estimates again but give the stated estimates directly.  
\end{remark}

\subsection{Key cancellations}\label{keycancellations}
\subsubsection{key cancellations inside the symbols of $A_{+,+}(\cdot, \cdot)$ and $A_{+,-}(\cdot, \cdot)$}  From (\ref{twosame}), we have
\[
\tilde{m}'_{+,+}(\xi-\eta, \eta) + \tilde{m}'_{+,+}(\eta, \xi-\eta) =  \frac{-\xi\cdot(\xi-2\eta)}{4 |\xi|^2}(1-\cos(\angle(\xi-\eta, \eta)))\Big((\xi-\eta)\times\eta\Big) \]
\[- \big(\frac{\xi}{4|\xi|}-\frac{\xi-\eta}{4|\xi-\eta|}\big)\cdot \Big(\frac{\xi-\eta}{|\xi-\eta|}-\frac{\eta}{|\eta|}\Big) \Big( (\xi - \eta) \times \eta \Big) -\frac{\xi-\eta}{4|\xi-\eta|}\cdot \Big(\frac{\xi-\eta}{|\xi-\eta|}-\frac{\eta}{|\eta|}\Big) \Big( (\xi - \eta) \times \eta \Big)
\] 
\[
= \frac{\xi\cdot \eta - |\xi|^2}{2|\xi^2|}(1-\cos(\angle(\xi-\eta, \eta)))\big((\xi-\eta)\times \eta \big)- \big(\frac{\xi}{4|\xi|}-\frac{\xi-\eta}{4|\xi-\eta|}\big)\cdot \Big(\frac{\xi-\eta}{|\xi-\eta|}-\frac{\eta}{|\eta|}\Big) \Big( (\xi - \eta) \times \eta \Big).
\]
Very similarly, from (\ref{twoopposite}), we have the following decomposition for $\tilde{m}'_{+,-}(\xi-\eta, \eta)$,
\[
\tilde{m}'_{+,-}(\xi-\eta, \eta) = \frac{\xi\cdot \eta - |\xi|^2}{2|\xi^2|}(1+\cos(\angle(\xi-\eta, \eta)))\big((\xi-\eta)\times \eta \big)\]
\[- \big(\frac{\xi}{4|\xi|}-\frac{\xi-\eta}{4|\xi-\eta|}\big)\cdot \Big(\frac{\xi-\eta}{|\xi-\eta|}+\frac{\eta}{|\eta|}\Big) \Big( (\xi - \eta) \times \eta \Big).
\]
Therefore,
\[
 a_{+,+}(\xi-\eta, \eta) +a_{+,+}(\eta, \xi-\eta) 
=\frac{i(|\xi|+|\xi-\eta|+|\eta|)\times\big(\tilde{m}'_{+,+}(\xi-\eta, \eta) + \tilde{m}'_{+,+}(\eta, \xi-\eta)\big) }{|\xi-\eta||\eta|(\cos(\angle(\xi-\eta, \eta))-1)} 
\]
\[ 
= i\frac{|\xi|+|\xi-\eta|+|\eta|}{2|\xi-\eta||\eta|}\xi \times \eta + i \frac{-\xi\cdot \eta(|\xi|+|\xi-\eta|+|\eta|)}{2|\xi^2||\xi-\eta||\eta|}\big((\xi-\eta)\times \eta \big) 
\]
\begin{equation}\label{equation12300}
- i\big(\frac{\xi}{4|\xi|}-\frac{\xi-\eta}{4|\xi-\eta|}\big)\cdot \Big(\frac{\xi-\eta}{|\xi-\eta|}-\frac{\eta}{|\eta|}\Big) \Big( (\xi - \eta) \times \eta \Big) \frac{|\xi|+|\xi-\eta|+|\eta|}{|\xi-\eta||\eta|(\cos(\angle(\xi-\eta, \eta))-1)},
\end{equation}
\[
a_{+,-}(\xi-\eta,\eta)= \frac{ i \tilde{m}'_{+,-}(\xi-\eta, \eta)}{|\xi|-|\xi-\eta|+|\eta|}=\frac{i(|\xi|+|\xi-\eta|-|\eta|)\tilde{m}'_{+,-}(\xi-\eta, \eta)}{|\xi-\eta||\eta|(\cos(\angle(\xi-\eta, \eta))+1)}
\]
\[ 
= -i\frac{|\xi|+|\xi-\eta|-|\eta|}{2|\xi-\eta||\eta|} \xi\times \eta  + i\frac{\xi \cdot \eta(|\xi|+|\xi-\eta|-|\eta|)}{2|\xi|^2|\xi-\eta||\eta|} \xi \times \eta
\]
\begin{equation}\label{equation12301}
- i \big(\frac{\xi}{4|\xi|}-\frac{\xi-\eta}{4|\xi-\eta|}\big)\cdot \Big(\frac{\xi-\eta}{|\xi-\eta|}+\frac{\eta}{|\eta|}\Big) \Big( (\xi - \eta) \times \eta \Big) \frac{|\xi|+|\xi-\eta|-|\eta|}{|\xi-\eta||\eta|(\cos(\angle(\xi-\eta, \eta))+1)}.
\end{equation}
The main point of above decompositions is to separate out the leading term, which essentially causes the loss of a derivatives. Note that the difference between $\xi/|\xi|$ and $(\xi-\eta)/|\xi-\eta|$ is less than $|\eta|/|\xi|$ when $|\eta|\ll |\xi|$. Therefore, when $|\eta|\ll |\xi|$, the leading term of $a_{+,+}(\xi-\eta, \eta) + a_{+,+}(\eta, \xi-\eta)$ is $i\xi \times \eta/|\eta|$ and the leading term of $a_{+,-}(\xi-\eta, \eta)$ is $-i\xi \times \eta/|\eta|$. Now it is easy to see the leading terms are cancelled out. More precisely, from Lemma \ref{Snorm}, (\ref{equation12300}), and (\ref{equation12301}), the following estimates holds, 
\begin{equation}\label{equation12305}
\|a_{+,+}(\xi-\eta, \eta) +a_{+,+}(\eta, \xi-\eta) - i\xi \times \frac{\eta}{|\eta|}\|_{\mathcal{S}^\infty_{k,k_1,k_2}} \lesssim 2^{k_2}, \quad \textup{if $k_2\leq k_1-10$},
\end{equation}  
\begin{equation}\label{equation12306}
\| a_{+,-}(\xi-\eta, \eta) + i \xi \times \frac{\eta}{|\eta|}\|_{\mathcal{S}^\infty_{k,k_1,k_2}} \lesssim 2^{k_2}, \quad \textup{if $k_2\leq k_1-10$}.
\end{equation}
From (\ref{equation12306}), as a by product, the following estimate also holds
\begin{equation}\label{equation12307}
\| \overline{a_{+,-}(\xi, -\eta)} + i \xi \times \frac{\eta}{|\eta|}\|_{\mathcal{S}^\infty_{k,k_1,k_2}} \lesssim 2^{k_2}, \quad \textup{if $k_2\leq k_1-10$}.
\end{equation}
From (\ref{equation12305}), (\ref{equation12306}), and (\ref{equation12307}), the following estimates hold,
\begin{equation}\label{equation12310}
\|a_{+,+}(\xi-\eta, \eta) +a_{+,+}(\eta, \xi-\eta) + a_{+,-}(\xi-\eta, \eta)\|_{\mathcal{S}^\infty_{k,k_1,k_2}} \lesssim 2^{k_2}, \quad \textup{if $k_2\leq k_1-10$},
\end{equation}
\begin{equation}\label{equation12311}
\|a_{+,+}(\xi-\eta, \eta) +a_{+,+}(\eta, \xi-\eta) + \overline{a_{+,-}(\xi, -\eta)} \|_{\mathcal{S}^\infty_{k,k_1,k_2}} \lesssim 2^{k_2}, \quad \textup{if $k_2\leq k_1-10$}.
\end{equation}
After changing of coordinates in (\ref{equation12311}), we have the following byproduct,
\begin{equation}\label{equation64430}
\| a_{+,+}(\xi, -\eta) +a_{+,+}(-\eta,\xi)  +\overline{ a_{+,-}(\xi-\eta,\eta)}\|_{\mathcal{S}^{\infty}_{k,k_1,k_2}}\lesssim 2^{k_2},\quad \textup{if $k_2\leq k_1-10$}.
\end{equation}
\subsubsection{key cancellations inside the symbols of $A_{+,+}^4(\cdot, \cdot)$ and $A_{+,-}^4(\cdot, \cdot)$}  
Recall the detail formula of the symbol of $A_{\nu,\kappa}^4(\cdot, \cdot)$ in (\ref{equation1502}) and  the explicit formula of $q_2(\xi-\eta, \eta)$ in (\ref{equation1241}), we have the following identity
\[
a_{+,+}^4(\xi-\eta, \eta) + a_{+,-}^4(\xi-\eta, \eta)= \frac{i q_2(\xi-\eta, \eta)
}{|\xi-\eta|}= \frac{-2i \xi\cdot \eta}{|\xi|^2|\xi-\eta|} (\xi-\eta)\times \eta,
\]
\[
a_{+,+}^4(\xi-\eta, \eta) -\overline{a_{+,-}^4(\xi, -\eta)} = \frac{-i(|\xi|+|\xi-\eta|+|\eta|)}{|\xi-\eta||\eta|}\Big( \frac{\xi\cdot \eta}{ |\xi|^2} (\xi \times \eta) - \frac{(\xi-\eta)\cdot \eta}{|\xi-\eta|^2}(\xi\times \eta)
\Big)
\]
\[
+\Big(\frac{i(|\xi|+|\xi-\eta|-|\eta|)}{|\xi||\eta|}- \frac{i (|\xi|+|\xi-\eta|+|\eta|)}{|\xi-\eta||\eta|})\frac{(\xi-\eta)\cdot \eta}{ |\xi-\eta|^2}(\xi\times \eta),
\]
From above computations and Lemma \ref{Snorm}, the following estimate holds if $k_2\leq k_1-10$,
\begin{equation}\label{equation12005}
\| a_{+,+}^4(\xi-\eta, \eta) + a_{+,-}^4(\xi-\eta, \eta)\|_{\mathcal{S}^\infty_{k,k_1,k_2}} + \|a_{+,+}^4(\xi-\eta, \eta) -\overline{a_{+,-}^4(\xi, -\eta)}\|_{\mathcal{S}^\infty_{k,k_1,k_2}}\lesssim 2^{2k_2-k_1}. 
\end{equation}
\subsection{Energy estimate for the modified energy}
We first prove the following lemma, which consists of some bilinear estimates that will be used later.
\begin{lemma}\label{lemmal2}
For bilinear operator $Q\in\{Q_{0,\mu}, Q_{\mu,\nu}, 
\tilde{Q}_{0,0}, \tilde{Q}_{0,\mu}, \tilde{Q}_{\mu,\nu}\}$,  $\tilde{Q} 
\in \{ \tilde{Q}_{0}^{1}, \tilde{Q}_{0,\mu}^{1}, \tilde{Q}_{\mu, \nu}^{1} \} $ and any two smooth functions $h_{1}, h_{2} \in H^{1}\cap W^{1+} $, we have 
\begin{equation}\label{equation114}
\| Q(h_{1}, h_{2}) \|_{L^{2}} \lesssim \min\{ \| h_{1} \|_{H^{1}} \| h_{2} \|_{W^{1+}}, \|h_{2} \|_{H^{1}} \|h_{1} \|_{W^{1+}} \},
\end{equation}
\begin{equation}\label{equation117}
\|  \tilde{Q}(h_{1},h_{2}) \|_{L^{2}} \lesssim  \min\{ \| h_{1} \|_{L^{2}} \| h_{2}\|_{W^{1+}}, \| h_{2} \|_{L^{2}} \| h_{1} \|_{W^{1+}} \},\end{equation}
\[
\| \tilde{Q}(h_{1},h_{2}) \|_{L^{\infty}} \lesssim \min\big\{  \| h_{1} \|_{W^{1+}}  \| h_{2} \|_{W^{1+}}, \min\{
 \| h_{1}\|_{L^{\infty}} \| h_{2} \|_{W^{1+}} + \| h_{1} \|_{L^{\infty}}^{3/4}\times\]
\begin{equation} \label{equation168} 
\| h_{1} \|_{H^{5}}^{1/4} \| h_{2}\|_{L^{\infty}}^{3/4} \| h_{2}\|_{L^{2}}^{1/4},  \| h_{2}\|_{L^{\infty}} \| h_{1} \|_{W^{1+}} + \| h_{2} \|_{L^{\infty}}^{3/4} \| h_{2} \|_{H^{5}}^{1/4} \| h_{1}\|_{L^{\infty}}^{3/4} \| h_{1}\|_{L^{2}}^{1/4}  \} \big\}.
\end{equation}
\end{lemma}

\begin{proof}
Similar to the proof of estimate (\ref{equation1410}), from (\ref{equation12000}) in Lemma \ref{sizeofnormal1} and Lemma \ref{boundness}, it is easy to see  the desired  estimate (\ref{equation114}) and (\ref{equation117}) holds.  Note that, from (\ref{equation12000}) in Lemma \ref{sizeofnormal1} and bilinear estimate in Lemma \ref{boundness}, the following estimate holds, 
\[
\| \tilde{Q}(h_{1}, h_{2})\|_{L^{\infty}}  = \sum_{|k_{1} - k_{2}| \leq 4}2^{k_{1}} 2^{-(1+)k_{1,+} }  \| P_{k_{1}} h_{1} \|_{L^{\infty}} \| P_{k_{2}} h_{2} \|_{L^{\infty}} + 
\]
\begin{equation}\label{equation380}
\sum_{ |k_{1} - k_{2} |\geq 4 } 2^{\min\{k_{1}, k_{2}\} } 2^{-(1{+})(k_{1,+} + k_{2,+})}  \| P_{k_{1}} h_{1} \|_{W^{1{+}}} 
 \| P_{k_{2}} h_{2} \|_{W^{1+}} 
\lesssim \| h_{1} \|_{W^{1+}}  \| h_{2} \|_{W^{1{+}}}.
\end{equation}
There is another way to estimate the $L^\infty$-norm of  $\tilde{Q}(h_{1}, h_{2})$, which is as follows,
\[
\| \tilde{Q}(h_{1}, h_{2})\|_{L^{\infty} } 
\lesssim \| h_{1}\|_{L^{\infty}} \| h_{2} \|_{W^{1+}} +\sum_{k_{2}\leq k_{1} - 4} 2^{k_{2} } 2^{ - 5 k_{1,+}/4}  (\| P_{k_{1}} h_{1} \|_{L^{\infty}} \| P_{k_{2}} h_{2} \|_{L^{\infty}})^{3/4}\]
\begin{equation}\label{equation381}
\times \| P_{k_{2}} h_{2} \|_{L^{2}}^{1/4} \| P_{k_{1}} h_{1}\|_{H^{5}}^{1/4} \lesssim  \| h_{1}\|_{L^{\infty}} \| h_{2} \|_{W^{1+}} + \| h_{1} \|_{L^{\infty}}^{3/4} \| h_{1} \|_{H^{5}}^{1/4} \| h_{2}\|_{L^{\infty}}^{3/4} \| h_{2}\|_{L^{2}}^{1/4}.
\end{equation}
Since the upper bound we used is symmetric, we can switch the role of $h_{1}$ and $h_{2}$ in estimates (\ref{equation380}) and (\ref{equation381}) to see estimate (\ref{equation168}) holds.
\end{proof}

From the construction of cubic correction terms inside the modified energy, we know that  cubic terms inside $d E^{modi}(t)/d t$ that  does not depend on $\phi_0 $ are cancelled out. It would be sufficient to close the energy estimate, if we can prove the following two Lemmas.
\begin{lemma}\label{remainderestimate1}
Under the bootstrap assumption \textup{(\ref{smallness})}, we have
\begin{equation}\label{equation7}
\sup_{t\in[0,T]} (1+t)^{-2 p_0+1} \big|\textup{cubic terms inside $ \frac{d }{d t} E(t)$ that depend on $\phi_0$}\big|
\lesssim \epsilon_0^2.
\end{equation}
\end{lemma}
\begin{proof}
Recall the bootstrap assumption  (\ref{smallness}), we know that $\phi_0$ decays $1/t^{1/2}$ faster than $\Phi$.  With this $1/t^{1/2}$ faster decay rate, the decay rate for those cubic terms is sufficient. Although there are many cubic terms inside $d E(t)/d t$ that depends on $\phi_0$,  we can use (\ref{equation1363}), $L^2-L^\infty$ type estimate in Lemma \ref{boundness}, and the estimates in Lemma \ref{lemmal2} to estimate all of them. From (\ref{equation1100}), (\ref{equation1480}),(\ref{eqn1009}), and (\ref{equation1003}), we can identify all of them. As representative examples, we  estimate two of them in details as follows,
\begin{enumerate}
\item[(i)] 
 For any tuples $\beta, \gamma$, such that $|\beta|, |\gamma|< |\alpha|\in\{N_0, N_1+1\}$, we have the following estimate by estimate (\ref{equation114}) in Lemma \ref{lemmal2}, the following estimate holds, 
\[
\Big| \int \overline{\Phi^\alpha}(t)  \tilde{Q}_{0,\mu} (\phi_{0}^{\beta}(t), \Phi_{\mu}^{\gamma}(t)) d x \Big|\lesssim \| \Phi(t)\|_{X_{N_0}}\big[\|\phi_0(t)\|_{X_{N_0}} \|\Phi(t)\|_{Z'} + \| \Phi(t)\|_{X_{N_0}}\|\phi_0(t)\|_{Z'_1}\big]
\]
\[\lesssim (1+t)^{2p_0-1}\epsilon_1^3\lesssim (1+t)^{2p_0-1}\epsilon_0^2.
\] 
\item[(ii)]
From  (\ref{equation1483}) and $L^2-L^\infty$ type bilinear estimate in Lemma \ref{boundness}, the following estimate holds, 
\[
\Big| \int \overline{\Phi^\alpha}(t)  Q_2(R_2\phi_{0}^{\alpha}(t), R_1\Phi(t)) d x \Big|\lesssim  \| \Phi(t)\|_{X_{N_0}}\big[\|\phi_0(t)\|_{X_{N_0}} \|\Phi(t)\|_{Z'} + \| \Phi(t)\|_{X_{N_0}}\|\phi_0(t)\|_{Z'_1}\big]\]
\[\lesssim (1+t)^{2p_0-1}\epsilon_1^3\lesssim (1+t)^{2p_0-1}\epsilon_0^2.
\] 
\end{enumerate}
\end{proof}

\begin{lemma}\label{remainderestimate2}
Under the bootstrap assumption \textup{(\ref{smallness})}, we have
\begin{equation}\label{equation155}
\sup_{t\in[0,T]} (1+t)^{-2 p_0+1} \big| \textup{quartic terms of } \frac{d }{d t} (E_{FCorr}(t)+ E_{SCorr}(t))\big|
\lesssim \epsilon_0^2.
\end{equation}
Therefore, after combing above estimate with the results of Lemma \textup{\ref{sizeofmodifiedenergy}}  and Lemma \textup{\ref{remainderestimate1}}, we have 
\begin{equation}\label{equation8}
 \sup_{t\in[0,T]} (1+t)^{-p_0} \| (\phi_0,\Phi)\|_{X_{N_0}}\lesssim \epsilon_0.
\end{equation}
\end{lemma}
\subsection{Proof of Lemma \ref{remainderestimate2}}
Since the decay rate of quartic terms  is sufficient, we only have to avoid losing derivatives to close the argument. Recall the problematic terms we found in (\ref{equation62000}) and the second correction terms in (\ref{equation1340}), to close the argument, it is sufficient to prove the following terms do not lose a derivative,
\[
\sum_{\begin{subarray}{l}
(|\alpha_1|+|\alpha_2|, |\alpha_3|+|\alpha_4|)\\
\in \{(N_0,0), (N_1,1)\}\\
 \end{subarray}} \sum_{\begin{subarray}{l}
 \beta + \gamma=\alpha,  |\gamma|=1, \\
 |\gamma_3|=|\gamma_4|=0\\
 \end{subarray}} { \alpha \choose \beta} \big( J_{\beta, \gamma}^1(\Phi^\gamma) + J_{\beta, \gamma}^2(\Phi^\gamma) \big)  + \frac{1}{4}( J_{\alpha}^1 + J_{\alpha}^2),
\]
where
\[
J_{\beta, \gamma}^1(\Phi^\gamma):=\textup{Re} \Big( \int \overline{\Gamma^{\gamma}\mathcal{N}_1^{(\beta, 0)}} \big(A_{+, +}(\Phi^{\beta}, \Phi^{\gamma})+ A_{+, +}( \Phi^{\gamma}, \Phi^{\beta}) + A_{+, -}(\Phi^{\beta}, \overline{\Phi^{\gamma}})\big)
\]
\[
+\overline{\Gamma^{\gamma}\mathcal{N}_0^{(\beta, 0)}} \big(A_{+, +}(\phi_0^{\beta}, \Phi^{\gamma}) + A_{+, +}( \Phi^{\gamma}, \phi_0^{\beta})+ A_{+,-}(\phi_0^\beta, \overline{\Phi^\gamma})\Big), 
\]
\[
J_{\beta, \gamma}^2(\Phi^\gamma):=\textup{Re} \Big( \int \overline{\Gamma^{\gamma}\Phi^\beta} \big(A_{+, +}(\mathcal{N}_1^{(\beta, 0)}, \Phi^{\gamma})+ A_{+, +}( \Phi^{\gamma},\mathcal{N}_1^{(\beta, 0)})+ A_{+, -}(\mathcal{N}_1^{(\beta, 0)}, \overline{\Phi^{\gamma}})
\big)
\]
\[
+\overline{\Gamma^{\gamma}\phi_0^\beta} \big(A_{+, +}(\mathcal{N}_0^{(\beta, 0)}, \Phi^{\gamma})+ A_{+, +}( \Phi^{\gamma},\mathcal{N}_0^{(\beta, 0)})+  A_{+,-}(\mathcal{N}_0^{(\beta, 0)}, \overline{\Phi^\gamma})\big)\Big),\]
\[
J^1_\alpha:= \textup{Re} \Big(\int \overline{\mathcal{N}_1^{(\alpha,0)}} \big(A^4_{+, +}(\Phi^{\alpha},  \Phi) +A^4_{+, -}(\Phi^{\alpha},  \overline{\Phi}) \big) + \overline{\mathcal{N}_0^{(\alpha,0)}} \big(A^4_{+, +}(\phi_0^{\alpha},  \Phi) +A^4_{+, -}(\phi_0^{\alpha},  \overline{\Phi}) \big)\Big),
\]
\[
J^2_\alpha:= \textup{Re} \Big(\int \overline{\Phi^\alpha} \big(A^4_{+, +}(\mathcal{N}_1^{(\alpha,0)},  \Phi) +A^4_{+, -}(\mathcal{N}_1^{(\alpha,0)},  \overline{\Phi}) \big) + \overline{\phi_0^{\alpha}} \big(A^4_{+, +}(\mathcal{N}_0^{(\alpha,0)},  \Phi) +A^4_{+, -}(\mathcal{N}_0^{(\alpha,0)},  \overline{\Phi}) \big)\Big).
\]
We mention that $J^1_{\beta,\gamma}(\Phi^\gamma)$ and $J^2_{\beta, \gamma}(\Phi^\gamma)$  differs slightly, only the roles of $\Phi^\beta$ and $\mathcal{N}_1^{(\beta, 0)}$ are switched, and the same situation happens for $J_\alpha^1$ and $J_{\alpha}^2$. In the following two subsubsections, we mainly reveal cancellations inside $J^1_{\beta,\gamma}(\Phi^\gamma)$ and $J_\alpha^1$ in details. 

Although computations that readers will see are tedious, there are three main ideas behind: (i) The cancellation  (\ref{equation62002}) holds, as symbols of bilinear operators $A_{\mu, \nu}(\cdot, \cdot)$ and $A_{\mu, \nu}^{4}(\cdot, \cdot)$ are all imaginary and even in the sense of (\ref{equation12100}). With this fact, we can split variables into real part and imaginary part. As a result, we have the same inputs, which are all real, inside aforementioned bilinear operators. (ii)  With  cancellations in (\ref{equation12310}), (\ref{equation12311}), and (\ref{equation12005}), we know that those bilinear operators do not play many roles. Essentially speaking, the only difference between  quartic terms inside $J^1_{\beta, \gamma}(\Phi^\gamma)$ and $J_\alpha^1$ and cubic terms inside (\ref{equation9000}) is another real function, which does not affect too much.
(iii) There are cancellations inside the cubic terms of (\ref{equation9000}), which have been shown before in (\ref{equation61220}).

\subsubsection{Estimate of $J^1_{\beta,\gamma}(\Phi^\gamma)$ and $J^2_{\beta, \gamma}(\Phi^\gamma)$}
Since $\Gamma^\gamma\in \{\p_1, \p_2\}$, we write $\Gamma^\gamma$ as $\p_j$ for some $j\in\{1,2\}$.  We first consider $J^1_{\beta, \gamma}(\Phi^\gamma)$ and decompose it into two parts by splitting  the input $\p_j \Phi$ into imaginary part and real part. More precisely, we have
\begin{equation}\label{equation63330}
J^{i}_{\beta,\gamma}(\Phi^\gamma)=  J^{i}_{\beta,\gamma}(\p_j\textup{Re}(\Phi))+J^{i}_{\beta,\gamma}(i\p_j\textup{Im}(\Phi)), \quad i\in\{1,2\}.
\end{equation}
From the explicit formula of $a_{\mu, \nu}(\cdot, \cdot)$ in (\ref{symbolofnormal}), the following facts hold,
\begin{equation}\label{equation12100}
a_{\mu, \nu}(\xi-\eta, \eta)= a_{\mu, \nu}(\eta-\xi, -\eta),\quad  \textup{Re}(a_{\mu, \nu}(\xi-\eta, \eta))=0, \quad (\mu, \nu)\in \mathcal{S}.
\end{equation}
which further give us the following identity,
\begin{equation}\label{equation62002}
\textup{Re}\Big(\int f A_{\mu, \nu}(g, h)\Big) =0, \quad \textup{if\,} f, g, \textup{and\,} h\, \textup{\,are all real functions}.
\end{equation}

\noindent $\bullet$\quad   \textit{Estimate of $J^{1}_{\beta,\gamma}(\p_j\textup{Re}(\Phi))$}. \quad  From (\ref{equation62002}) and the fact that $\textup{Im}(\phi_0)=0$, the following identity holds after replacing $\Phi$ in terms of $\psi$, $G_1$, and $G_2$,
\[
 J^{1}_{\beta,\gamma}(\p_j\textup{Re}(\Phi))  = \textup{Re} \Big( \int \p_j\textup{Re}(\mathcal{N}_1^{(\beta, 0)}) \big(A_{+, +}(i \textup{Im}(\Phi^\beta), \p_j \textup{Re}(\Phi))+ A_{+, +}(\p_j \textup{Re}(\Phi),i \textup{Im}(\Phi^\beta))
\]
\[+A_{+, -}(i\textup{Im}(\Phi^\beta), \p_j \textup{Re}(\Phi))\Big)
+ \textup{Re} \Big( \int -i\p_j\textup{Im}(\mathcal{N}_1^{(\beta, 0)})\big(A_{+, +}( \textup{Re}(\Phi^\beta), \p_j \textup{Re}(\Phi)) \]
\[ + A_{+, +}(\p_j \textup{Re}(\Phi),\textup{Re}(\Phi^\beta))+ A_{+, -}(\textup{Re}(\Phi^\beta), \p_j \textup{Re}(\Phi))\big)
\Big).
 \]
After writing above terms in Fourier side, we have the following estimate, 
\[
| J^{1}_{\beta,\gamma}(\p_j\textup{Re}(\Phi)) | \lesssim \Big| \int \overline{\widehat{\textup{Re}\mathcal{N}_1^{(\beta,0)}}(\xi)} \widehat{\textup{Im}(\Phi^\beta)}(\xi-\eta)\widehat{\textup{Re}(\Phi)}(\eta) b_1(\xi,\eta)d \eta\Big|
\]
\[
+ \Big| \int \overline{\widehat{\textup{Im}\mathcal{N}_1^{(\beta,0)}}(\xi)} \widehat{\textup{Re}(\Phi^\beta)}(\xi-\eta)\widehat{\textup{Re}(\Phi)}(\eta) b_1(\xi,\eta)d \eta\Big|,
\]
where
\[
b_1(\xi, \eta)= i\xi_j \eta_j\big(a_{+,+}(\xi-\eta, \eta) + a_{+,+}(\eta, \xi-\eta) + a_{+,-}(\xi-\eta, \eta)).
\]
Recall (\ref{equation12310}). From (\ref{productsymbol}) in Lemma \ref{boundness}, the following estimate holds, 
\[
\| b_1(\xi, \eta)\|_{\mathcal{S}^\infty_{k,k_1,k_2}} \lesssim 2^{2k_2+k_1}, \quad \textup{if $k_2\leq k_1-10$},
\]
which further gives us the following estimate, 
\[
| J^{1}_{\beta,\gamma}(\p_j\textup{Re}(\Phi)) | \lesssim \|\mathcal{N}_1^{(\beta,0)} \|_{L^2} \| \Phi^\beta\|_{H^1} \| \Phi\|_{Z'}\]
\begin{equation}\label{equation12320}
\lesssim \|(\phi_0, \Phi)\|_{X_{N_0}}^2 \big(\|\phi_0\|_{Z'_1}+\|\Phi\|_{Z'}\big)^2 \lesssim (1+t)^{-2p_0+1}\epsilon_0^3.
\end{equation}

\noindent $\bullet$\quad   \textit{Estimate of $J^{1}_{\beta,\gamma}(i\p_j\textup{Im}(\Phi))$}. \quad
Using (\ref{equation62002}), we have the following identity for  $J^{1}_{\beta,\gamma}(i\p_j\textup{Im}(\Phi))$,
\[
 J^{1}_{\beta,\gamma}(i\p_j\textup{Im}(\Phi))= \textup{Re} \Big( \int \p_j\textup{Re}(\mathcal{N}_1^{(\beta, 0)}) \big( A_{+, +}( \textup{Re}(\Phi^\beta), i\p_j \textup{Im}(\Phi))+ A_{+, +}(  i\p_j \textup{Im}(\Phi),\textup{Re}(\Phi^\beta))
\]
 \[
 +A_{+, -}( \textup{Re}(\Phi^\beta), -i\p_j \textup{Im}(\Phi))\big) 
 + \p_j\textup{Re}(\mathcal{N}_0^{(\beta, 0)})\big( A_{+, +}(\textup{Re}(\phi_0^{\beta}),i\p_j \textup{Im}(\Phi)) + A_{+, +}(i\p_j \textup{Im}(\Phi),
 \]
 \[
\textup{Re}(\phi_0^{\beta}))+  A_{+, -}(\textup{Re}(\phi_0^{\beta}),-i\p_j \textup{Im}(\Phi)) \big) + \p_j\textup{Im}(\mathcal{N}_1^{(\beta, 0)}) \big(A_{+, +}( \textup{Im}(\Phi^\beta), i\p_j \textup{Im}(\Phi))
 \]
 \[
+ A_{+, +}(  i\p_j \textup{Im}(\Phi), \textup{Im}(\Phi^\beta)) + A_{+, -}( \textup{Im}(\Phi^\beta), -i\p_j \textup{Im}(\Phi))\big) + \p_j\textup{Im}(\mathcal{N}_0^{(\beta, 0)}) \big(A_{+, +}( \textup{Im}(\phi_0^\beta), 
 \]
 \begin{equation}\label{equation64320}
i\p_j \textup{Im}(\Phi))+  A_{+, +}(  i\p_j \textup{Im}(\Phi), \textup{Im}(\phi_0^\beta)) + A_{+, -}( \textup{Im}(\phi_0^\beta), -i\p_j \textup{Im}(\Phi))\big) \Big). 
 \end{equation}
 
  From (\ref{changevariables}) and (\ref{maineqn}),  we rewrite  (\ref{equation64320}) in terms of $\psi$, $G_1$, and $G_2$ and have the following identity, 
\begin{equation}\label{equation63332}
J^{1}_{\beta,\gamma}(i\p_j\textup{Im}(\Phi)) = J^{1,1}_{\beta,\gamma}(i\p_j\textup{Im}(\Phi))+ J^{1,2}_{\beta,\gamma}(i\p_j\textup{Im}(\Phi)), 
\end{equation}
where
\[J^{1,1}_{\beta,\gamma}(i\p_j\textup{Im}(\Phi)):= \textup{Re}\Big( \int  \p_j R_1 \widetilde{\mathcal{N}}_1^{(\beta, 0)} \big( A_{+,+}(R_2 G_2^\beta, i\p_j \textup{Im}(\Phi)) + A_{+,+}(i\p_j \textup{Im}(\Phi), R_2 G_2^\beta) \]
\[
+ A_{+,+}(R_2 G_2^\beta, -i\p_j \textup{Im}(\Phi))\big) - \p_j R_2 \widetilde{\mathcal{N}}_1^{(\beta, 0)} \big( A_{+,+}(R_1 G_2^\beta, i\p_j \textup{Im}(\Phi)) + A_{+,+}(i\p_j \textup{Im}(\Phi), R_1 G_2^\beta)
\]
\[+ A_{+,-}(R_1 G_2^\beta, -i\p_j\textup{Im}(\Phi))\big) +  \p_j R_2 \widetilde{\mathcal{N}}_2^{(\beta, 0)} \big( A_{+,+}(R_1 G_1^\beta, i\p_j \textup{Im}(\Phi)) + A_{+,+}( i\p_j \textup{Im}(\Phi),R_1 G_1^\beta)\]
\[
+ A_{+,-}(R_1 G_1^\beta, -i\p_j \textup{Im}(\Phi))  \big)  -  \p_j R_1 \widetilde{\mathcal{N}}_2^{(\beta, 0)} \big( A_{+,+}(R_2 G_1^\beta, i\p_j \textup{Im}(\Phi)) + A_{+,+}( i\p_j \textup{Im}(\Phi),R_2 G_1^\beta)\]
\[
+ A_{+,-}(R_2 G_1^\beta, -i\p_j \textup{Im}(\Phi))  \big),  
\]
\[
 J^{1,2}_{\beta,\gamma}(i\p_j\textup{Im}(\Phi)):= 
\textup{Re}\Big( \int \p_j \widetilde{\mathcal{N}}_0^{(\beta, 0)} \big(A_{+,+}(\psi^\beta, i \p_j \textup{Im}(\Phi)) + A_{+,+}(i \p_j \textup{Im}(\Phi),\psi^\beta)   \]
\[+A_{+,-}(\psi^\beta, -i \p_j \textup{Im}(\Phi)) \big) +  \sum_{i=1,2} \p_j R_i \widetilde{\mathcal{N}}_1^{(\beta, 0)} \big(A_{+,+}(R_i G_1^\beta, i\p_j\textup{Im}(\Phi)) + A_{+,+}(i\p_j\textup{Im}(\Phi),R_i G_1^\beta)\]
\[ 
+ A_{+,-}(R_i G_1^\beta, -i\p_j\textup{Im}(\Phi))\big)+ \p_j R_i \widetilde{\mathcal{N}}_2^{(\beta, 0)} \big(A_{+,+}(R_i G_2^\beta, i\p_j \textup{Im}(\Phi)) \]
\[
+A_{+,+}( i\p_j \textup{Im}(\Phi),R_i G_2^\beta)+ A_{+,-}(R_i G_2^\beta, -i\p_j \textup{Im}(\Phi)) \Big). \]
We write $J^{1,1}_{\beta, \gamma}(\p_j \textup{Im}(\Phi))$ on the Fourier side and have the following,
 \[
| J^{1,1}_{\beta,\gamma}(i\p_j\textup{Im}(\Phi))|\lesssim \Big|\int \int \overline{\widehat{\widetilde{\mathcal{N}}_1^{(\beta,0)}}(\xi)} \widehat{G_2^\beta}(\xi-\eta)\widehat{ \textup{Im}(\Phi)}(\eta) m_1(\xi-\eta, \eta) d \eta d \xi \Big|
 \]
 \[
 + \Big| \int \int \overline{\widehat{\widetilde{\mathcal{N}}_2^{(\beta,0)}}(\xi)} \widehat{G_1^\beta}(\xi-\eta)\widehat{ \textup{Im}(\Phi)}(\eta) m_1(\xi-\eta, \eta) d \eta d \xi \Big|, \]
where
\[
m_1(\xi-\eta, \eta)= i\xi_j \eta_j \frac{\xi}{|\xi|}\times\big(\frac{\xi-\eta}{|\xi-\eta|}\big) \big(a_{+,+}(\xi-\eta, \eta) + a_{+,+}(\eta, \xi-\eta)- a_{+,-}(\xi-\eta, \eta)\big)\]
\[
=  i\xi_j \eta_j \frac{-\xi\times \eta}{|\xi||\xi-\eta|}\big( a_{+,+}(\xi-\eta, \eta)+ a_{+,+}(\eta, \xi-\eta)- a_{+,-}(\xi-\eta, \eta)\big) .\]
From above computation, we can see the cancellation comes from the Riesz operators. From Lemma \ref{Snorm}, the following estimate holds,
\[
\| m_1(\xi-\eta, \eta)\|_{\mathcal{S}^{\infty}_{k,k_1,k_2}}\lesssim 2^{2k_2+k_1}, \quad \textup{if $k_2\leq k_1-10$}.
\]
From above estimate and $L^2-L^\infty$ type bilinear estimate  in Lemma \ref{boundness}, the following estimate holds, \[
| J^{1,1}_{\beta,\gamma}(i\p_j\textup{Im}(\Phi))|\lesssim \|(\widetilde{\mathcal{N}}_1^{(\beta, 0)}, \widetilde{\mathcal{N}}_2^{(\beta, 0)})\|_{L^2}\big(\|\Phi \|_{Z'}\| (G_1, G_2)\|_{X_{N_0}} + \|\Phi\|_{X_{N_0}}\|(G_1,G_2)\|_{Z_1'}\big) 
\]
\begin{equation}\label{equation63317}
\lesssim \|(\phi_0, \Phi)\|_{X_{N_0}}^2 \big(\|\phi_0\|_{Z'_1}+\|\Phi\|_{Z'}\big)^2 \lesssim (1+t)^{-2p_0+1}\epsilon_0^3.
\end{equation}
We proceed to consider $J^{1,2}_{\beta, \gamma}(\p_j \textup{Im}(\Phi))$.  Recall  (\ref{equation62220}), similar to computations in (\ref{equation61220}),  we  write $J^{1,2}_{\beta, \gamma}(\p_j \textup{Im}(\Phi))$  on the Fourier side. As a result, the following estimate holds,
\[
| J^{1,2}_{\beta,\gamma}(i\p_j\textup{Im}(\Phi))|\lesssim \Big| \int \int \int  \overline{\widehat{\psi^{\beta}}(\xi-\sigma) \psi(\sigma)} \widehat{\psi^{\beta}}(\xi-\eta)\widehat{\textup{Im}(\Phi)}(\eta) m_2(\xi, \eta, \sigma) d \sigma d \eta d \xi \Big|
\]
\[
+ \sum_{i=1,2} \Big| \int \int \int \overline{\widehat{G_i^\beta}(\xi-\sigma)\widehat{\psi}(\sigma)} \widehat{G_i^\beta}(\xi-\eta)\widehat{\textup{Im}(\Phi)}(\eta) m_3(\xi, \eta, \sigma) d \sigma d \eta d \xi \Big|
\]
\begin{equation}\label{equation63221}
+ \Big| \int \int \int \overline{\widehat{G_i^\beta}(\xi-\sigma)\widehat{G_i}(\sigma)} \widehat{\psi^\beta}(\xi-\eta)\widehat{\textup{Im}(\Phi)}(\eta) m_4(\xi, \eta, \sigma) d \sigma d \eta d \xi \Big|,
\end{equation}
where
\[
m_2(\xi,\eta, \sigma) = i \xi_j\eta_j \big(q(\xi-\sigma, \sigma)  + q(\sigma, \xi-\sigma)\big)\big(a_{+,+}(\xi-\eta, \eta)+ a_{+,+}(\eta, \xi-\eta) \big)
\]
\begin{equation}\label{equation12201}
+i (\xi_j-\eta_j-\sigma_j)(-\eta_j) \big( q(\xi-\eta, -\sigma) + q(-\sigma, \xi-\eta)\big) \overline{a_{+,-}(\xi-\sigma, -\eta)},
\end{equation}
\[
m_3(\xi, \eta,\sigma)= i \xi_j \eta_j {(\xi-\sigma)\times \sigma}  \frac{\xi\cdot(\xi-\eta)}{|\xi||\xi-\eta|}\big(a_{+,+}(\xi-\eta, \eta)+ a_{+,+}(\eta, \xi-\eta)\big)
\]
\begin{equation}\label{equation63310}
+i(\xi_j-\eta_j-\sigma_j)(-\eta_j) {(\xi-\eta)\times (-\sigma)} \frac{(\xi-\eta-\sigma)\cdot(\xi-\sigma)}{|\xi-\eta-\sigma||\xi-\sigma|}  \overline{a_{+,-}(\xi-\sigma, -\eta)},
\end{equation}
\[
m_4(\xi, \eta,\sigma)= i\xi_j \eta_j(-q(\xi-\sigma, \sigma)-q(\sigma, \xi-\sigma))(a_{+,+}(\xi-\eta, \eta)+ a_{+,+}(\eta, \xi-\eta)-a_{+,-}(\xi-\eta, \eta))
\]
\[
+ i(\xi_j-\eta_j-\sigma_j)(-\eta_j)\big(- (-\sigma)\times(\xi-\eta)\big)\frac{(\xi-\eta-\sigma)\cdot(\xi-\eta)}{|\xi-\eta-\sigma||\xi-\eta|}\big(\overline{a_{+,-}(\xi-\sigma, -\eta)}\]
\begin{equation}\label{equation63315}
-\overline{a_{+,+}(\xi-\sigma, -\eta)}- \overline{a_{+,+}(-\eta, \xi-\sigma)}\big)
\end{equation}
 Since losing a derivative for (\ref{equation63221})   is only relevant when $|\eta|, |\sigma|\ll |\xi|$,  we assume $|\eta|, |\sigma|\ll |\xi|$ in the rest of this subsubsection. 
Now, we are ready to see cancellation inside  $m_i(\xi, \eta, \sigma)$,$i\in\{2,3,4\}$. We decompose them as follows, \[
 m_2(\xi, \eta, \sigma)= i\xi_j\eta_j\underbrace{(q(\xi-\sigma, \sigma)+q(\sigma,\xi-\sigma)+q(\xi, -\sigma)+q(-\sigma, \xi))}_{\textup{cancellation from (\ref{equation1008}) and (\ref{equation1483})} }\big(a_{+,+}(\xi-\eta, \eta)\]
 \[+ a_{+,+}(\eta, \xi-\eta) \big)-i\xi_j\eta_j\big(q(\xi, -\sigma)+q(-\sigma, \xi)\big)\underbrace{(a_{+,+}(\xi-\eta, \eta)+a_{+,-}(\eta, \xi-\eta)+\overline{a_{+,-}(\xi,-\eta)})}_{\text{cancellation from (\ref{equation12311})  }}
 \]
 \[
 +\underbrace{i(\eta_j+\sigma_j)\eta_j\big( q(\xi-\eta, -\sigma) + q(-\sigma, \xi-\eta)\big) \overline{a_{+,-}(\xi-\sigma, -\eta)}}_{\text{rough estimate will do}} -i\xi_j\eta_j \overline{a_{+,-}(\xi, -\eta)}\times
\]
\[
\underbrace{(q(\xi-\eta, -\sigma) -q(\xi, -\sigma) + q(-\sigma, \xi-\eta)-q(-\sigma, \xi) )}_{\text{cancellation from the smallness of the difference between $\xi-\eta$ and $\xi$ }} +  i\underbrace{(\overline{a_{+,-}(\xi,-\eta)} -\overline{a_{+,-}(\xi-\sigma, -\eta)})}_{\text{ from the smallness of the difference between $\xi-\sigma$ and $\xi$ }}
\]
\begin{equation}\label{}
\cdot \xi_j\eta_j \big(q(\xi-\eta, -\sigma) + q(-\sigma, \xi-\eta)\big) .
\end{equation}
From above decomposition,  (\ref{equation1483}) and (\ref{equation12311}), and Lemma \ref{Snorm}, the following estimate holds,
\begin{equation}\label{equation63311}
\|m_2(\xi, \eta, \sigma)\|_{\mathcal{S}^\infty_{k,k_1,k_2,k_3}}\lesssim 2^{3\max\{k_2,k_3\}+2k_1}, \quad \textup{if $\max\{k_2, k_3\}\leq k_1-10$}.
\end{equation}

Now, we proceed to reveal cancellations inside $m_3(\xi,\eta, \sigma)$ and $m_{4}(\xi, \eta, \sigma)$. For simplicity, we only highlight those symbols  that do not satisfy (\ref{equation63311}) type estimate. As a result, we have
\[
m_3(\xi, \eta, \sigma)= i\big( \underbrace{a_{+,+}(\xi-\eta, \eta)+ a_{+,+}(\eta, \xi-\eta)+ \overline{a_{+,-}(\xi, -\eta)}}_{\text{cancellation from  (\ref{equation12311}) } }- \overline{a_{+,-}(\xi, -\eta)}+ \overline{a_{+,-}(\xi-\sigma, -\eta)}\big) 
\]
\[
 \cdot \xi_j \eta_j (\xi\times\sigma) + \textup{other terms inside (\ref{equation63310}) that satisfy (\ref{equation63311}) type estimate},
\]
\[
m_4(\xi, \eta, \sigma)= i\xi_j \eta_j \underbrace{\big[ \big(\sigma\times \xi \big)-q(\xi-\sigma, \sigma)-q(\sigma, \xi-\sigma)\big]}_{\text{cancellation from (\ref{equation1241}) and (\ref{equation1483})}}\big(a_{+,+}(\xi-\eta, \eta)+a_{+,+}(\eta, \xi-\eta)
\]
\[
-a_{+,-}(\xi-\eta, -\eta)\big)- i \xi_j\eta_j \big(\sigma \times \xi \big)\Big(\big(\underbrace{a_{+,+}(\xi-\eta, \eta)+ a_{+,+}(\eta, \xi-\eta)+\overline{a_{+,-}(\xi, -\eta)}\big)}_{\text{cancellation from (\ref{equation12311})}} \]
\[
-\underbrace{\big(\overline{a_{+,+}(\xi, -\eta)} + \overline{a_{+,+}(\xi, -\eta)}+ a_{+,-}(\xi-\eta, \eta) \big)}_{\text{cancellation from (\ref{equation64430}) }} +\underbrace{\big(-\overline{a_{+,-}(\xi, -\eta)}+\overline{a_{+,-}(\xi-\sigma, -\eta)}}_{\text{cancellation from the smallness of between $\xi$ and $\xi-\sigma$}}\]
\[
\underbrace{-\overline{a_{+,+}(\xi-\sigma, -\eta) } + \overline{a_{+,+}(\xi, -\eta)} -\overline{a_{+,+}(-\eta, \xi-\sigma)} +\overline{a_{+,+}(-\eta, \xi)}\big)}_{\text{cancellation from the smallness of $\xi$ and $\xi-\sigma$}}\Big)
\]
\[
+ \textup{other terms inside (\ref{equation63315}) that satisfy (\ref{equation63311}) type estimate}.
\]
From above decomposition,  (\ref{equation1483}) and (\ref{equation12311}), and Lemma \ref{Snorm}, the following estimate holds if $\max\{k_2, k_3\}\leq k_1-10$,
\begin{equation}\label{equation63316}
\|m_3(\xi, \eta, \sigma)\|_{\mathcal{S}^\infty_{k,k_1,k_2,k_3}} + \|m_4(\xi, \eta, \sigma)\|_{\mathcal{S}^\infty_{k,k_1,k_2,k_3}}\lesssim 2^{3\max\{k_2,k_3\}+2k_1}.
\end{equation}
From  (\ref{equation63311}) and (\ref{equation63316}) and $L^2-L^\infty-L^\infty$ type trilinear estimate  in Lemma \ref{boundness}, we have
\begin{equation}\label{equation63333}
| J^{1,2}_{\beta,\gamma}(i\p_j\textup{Im}(\Phi))|\lesssim \big(\|\Phi \|_{Z'}\| (G_1, G_2)\|_{X_{N_0}} + \|\Phi\|_{X_{N_0}}\|(G_1,G_2)\|_{Z_1'}\big)^2\lesssim  (1+t)^{-1+2p_0} \epsilon_0^3. 
\end{equation}
To sum up, from (\ref{equation63330}),  (\ref{equation12320}),  (\ref{equation63332}),  (\ref{equation63317}), and  (\ref{equation63333}) , the following estimate holds,
\[
|J_{\beta, \gamma}^1(\Phi^\gamma)| \lesssim (1+t)^{-1+2p_0}\epsilon_0^3.
\]
Note the following two facts, one can estimate $J_{\beta, \gamma}^2(\Phi^\gamma) $ very similarly with minor modifications.
\begin{enumerate}
\item[(i)] We can represent $J^2_{\beta, \gamma}(\Phi^\gamma)$ in the  following way, 
\[
\textup{Re}\Big( \int \overline{\Gamma^{\gamma}f} \big(T(g, \Gamma^{\gamma} h ) \Big)= \textup{Re}\Big( \int \overline{\Gamma^{\gamma} g} \tilde{T}(f, \overline{\Gamma^\gamma h})  \Big) + \textup{a term that doesn't lose derivative},
\]
where the first term is similar to $J^1_{\beta, \gamma}(\Phi^\gamma)$ and $\tilde{T}(\cdot, \cdot)$ is defined by the following symbol,
\[
\tilde{t}(\xi-\eta, \eta):= \overline{t(\xi, -\eta)}, \quad  \textup{$t(\xi-\eta, \eta)$ is the symbol of bilinear operator $T\in \{A_{+,+}, A_{+,-} \}$}.
\]
\item[(ii)] From (\ref{equation12310}) and (\ref{equation12311}), same types of cancellations also happen for $\tilde{t}(\xi-\eta, \eta)$ when $|\eta|\ll |\xi|$.
\end{enumerate}

\subsubsection{Estimating  $J^1_{\alpha}$ and $J^2_{\alpha}$} We first  estimate  $J_\alpha^1$ and decompose it into two parts as follows, 
\[
J_{\alpha}^1 = J_\alpha^1(\textup{Re}(\Phi)) + J_\alpha^1(i\textup{Im}(\Phi)),\quad J^1_\alpha(\textup{Re}(\Phi)):= \textup{Re} \Big(\int \overline{\mathcal{N}_1^{(\alpha,0)}} \big(A^4_{+, +}(\Phi^{\alpha},  \textup{Re}(\Phi)) 
\]
\[
+A^4_{+, -}(\Phi^{\alpha},  \textup{Re}(\Phi)) \big) + \overline{\mathcal{N}_0^{(\alpha,0)}} \big(A^4_{+, +}(\phi_0^{\alpha},  \textup{Re}(\Phi)) +A^4_{+, -}(\phi_0^{\alpha},  \textup{Re}(\Phi)) \big)\Big),
\]
\[
J^1_\alpha(i\textup{Im}(\Phi)):= \textup{Re} \Big(\int \overline{\mathcal{N}_1^{(\alpha,0)}} \big(A^4_{+, +}(\Phi^{\alpha},  i\textup{Im}(\Phi)) +A^4_{+, -}(\Phi^{\alpha},  -i\textup{Im}(\Phi)) \big) \]
\[+ \overline{\mathcal{N}_0^{(\alpha,0)}} \big(A^4_{+, +}(\phi_0^{\alpha},  i\textup{Im}(\Phi)) +A^4_{+, -}(\phi_0^{\alpha},  -i\textup{Im}(\Phi)) \big)\Big).
\]
From (\ref{equation12005}) and $L^2-L^\infty$ type bilinear estimate in Lemma \ref{boundness}, the following estimate holds, 
\[
|J_{\alpha}^1(\textup{Re}(\Phi))|\lesssim \|\mathcal{N}_1^{(\alpha,0)}\|_{H^{-1}} \|\Phi\|_{X_{N_0}}\|\Phi\|_{Z'}
\]
\begin{equation}\label{equation63421}
\lesssim \|(\phi_0, \Phi)\|_{X_{N_0}}^2 (\|\phi_0\|_{Z'_1} +\|\Phi\|_{Z'})^2\lesssim (1+|t|)^{-1+2p_0}\epsilon_0^3.
\end{equation}
Note that (\ref{equation63330}) also holds for $a_{\mu, \nu}^4(\cdot, \cdot)$. Hence  identity (\ref{equation62002}) also holds for bilinear  operator $A_{\mu, \nu}^4(\cdot, \cdot)$. As a result, we have
\[
J^1_\alpha(i\textup{Im}(\Phi))= J^{1,1}_\alpha(i\textup{Im}(\Phi))+ J^{1,2}_\alpha(i\textup{Im}(\Phi)),\]
\[
J^{1,1}_\alpha(i\textup{Im}(\Phi)):= \textup{Re}\Big( \int  R_{2} \widetilde{\mathcal{N}}_2^{(\alpha, 0)}\big( A_{+,+}^4( R_1 G_1^\alpha, i \textup{Im}(\Phi) + A_{+,-}^4(R_1 G_1^\alpha, -i \textup{Im}(\Phi))\big)
\]
\[
- R_{1} \widetilde{\mathcal{N}}_2^{(\alpha, 0)}\big( A_{+,+}^4( R_2 G_1^\alpha, i \textup{Im}(\Phi) + A_{+,-}^4(R_2 G_1^\alpha, -i \textup{Im}(\Phi))\big) +  R_{1} \widetilde{\mathcal{N}}_1^{(\alpha, 0)}\big( A_{+,+}^4( R_2 G_2^\alpha, i \textup{Im}(\Phi) \]
\[+ A_{+,-}^4(R_2 G_2^\alpha, -i \textup{Im}(\Phi))\big)- R_{2} \widetilde{\mathcal{N}}_1^{(\alpha, 0)}\big( A_{+,+}^4( R_1 G_2^\alpha, i \textup{Im}(\Phi) + A_{+,-}^4(R_1 G_2^\alpha, -i \textup{Im}(\Phi))\big)\Big),
\]
\[ J^{1,2}_\alpha(i\textup{Im}(\Phi)):= \textup{Re}\Big( \int \widetilde{\mathcal{N}}_0^{(\alpha, 0)}\big(A_{+,+}^4(\psi^\alpha, i \textup{Im}(\Phi)) + A_{+,-}^4(\psi^\alpha, -i \textup{Im}(\Phi))\big) \]
\[
+ \sum_{i,j=1,2}  R_{i} \widetilde{\mathcal{N}}_j^{(\alpha, 0)}\big( A_{+,+}^4( R_i G_j^\alpha, i \textup{Im}(\Phi) + A_{+,-}^4(R_i G_j^\alpha, -i \textup{Im}(\Phi))\big)\Big).
\]
Very similar to estimate  of  $J^{1,1}_{\beta,\gamma}(\p_j\textup{Im}(\Phi))$ in (\ref{equation63317}), the cancellation comes from the Riesz operator. With minor modifications, we can prove the following estimate, 
\[
|J_{\alpha}^{1,1}(i \textup{Im}(\Phi))|\lesssim \|(\widetilde{\mathcal{N}}_1^{(\alpha, 0)}, \widetilde{\mathcal{N}}_2^{(\alpha, 0)})\|_{H^{-1}}\|(G_1, G_2, \Phi)\|_{X_{N_0}} (\|(G_1,G_2)\|_{Z_1'} + \|\Phi\|_{Z'})
\]
\begin{equation}\label{equation63420}
\lesssim \|(\phi_0, \Phi)\|_{X_{N_0}}^2 \big( \|\phi_0\|_{Z_1'} +\|\Phi\|_{Z'}\big) \lesssim (1+t)^{-1+2p_0}\epsilon_0^3.
\end{equation}

Very Similar to what we did for $J^{1,2}_{\beta, \gamma}(i\p_j \textup{Im}(\Phi))$ in (\ref{equation63221}), 
we can write $J_{\alpha}^{1,2}(i\textup{Im}(\Phi))$ on the Fourier side and rewrite associated symbols as follows, 
\[
| J^{1,2}_{\alpha}(i\textup{Im}(\Phi))|\lesssim \Big| \int \int \int  \overline{\widehat{\psi^{\beta}}(\xi-\sigma) \psi(\sigma)} \widehat{\psi^{\beta}}(\xi-\eta)\widehat{\textup{Im}(\Phi)}(\eta) m_5(\xi, \eta, \sigma) d \sigma d \eta d \xi \Big|
\]
\[
+ \sum_{i=1,2} \Big| \int \int \int \overline{\widehat{G_i^\beta}(\xi-\sigma)\widehat{\psi}(\sigma)} \widehat{G_i^\beta}(\xi-\eta)\widehat{\textup{Im}(\Phi)}(\eta) m_6(\xi, \eta, \sigma) d \sigma d \eta d \xi \Big|
\]
\[
+ \Big| \int \int \int \overline{\widehat{G_i^\beta}(\xi-\sigma)\widehat{G_i}(\sigma)} \widehat{\psi^\beta}(\xi-\eta)\widehat{\textup{Im}(\Phi)}(\eta) m_7(\xi, \eta, \sigma) d \sigma d \eta d \xi \Big|,
\]
where
\[
m_5(\xi,\eta, \sigma)=  i\big(q(\xi-\sigma, \sigma)+q(\sigma, \xi-\sigma)\big)a_{+,+}^4(\xi-\eta, \eta) + i \big(q(\xi-\eta, -\sigma)+q(-\sigma , \xi-\eta)\big) \overline{a_{+,-}^4(\xi-\sigma, -\eta)}
\]
\[
= i\underbrace{\big(q(\xi-\sigma, \sigma)+q(\sigma, \xi-\sigma)+ q(\xi, -\sigma)+q(-\sigma, \xi)\big)}_{\textup{cancellation from (\ref{equation1008}) and (\ref{equation1483})} }a_{+,+}^4(\xi-\eta, \eta) 
\]
\[
+ i \big(q(\xi, -\sigma)+q(-\sigma, \xi) \big)\underbrace{\big(-a_{+,+}^4(\xi-\eta, \eta) +  \overline{a_{+,-}^4(\xi,-\eta)} \big)}_{\textup{cancellation from (\ref{equation12005})}} + \textup{good errors},
\]
\[
m_6(\xi, \eta, \sigma)= i \big((\xi-\sigma)\times \sigma\big)\frac{\xi\cdot(\xi-\eta)}{|\xi||\xi-\eta|}a_{+,+}^4(\xi-\eta, \eta) +  i\big((\xi-\eta)\times (-\sigma)\big) \frac{(\xi-\eta-\sigma)\cdot(\xi-\sigma)}{|\xi-\eta-\sigma||\xi-\sigma|}\]
\[ \cdot  \overline{a_{+,-}^4(\xi-\sigma,-\eta)}=  i \big(\xi\times \sigma\big)\underbrace{\big(a_{+,+}^4(\xi-\eta, \eta)- \overline{a_{+,-}^4(\xi-\sigma,-\eta)}\big)}_{\textup{cancellation from (\ref{equation12005})}} + \textup{good errors},
\]
\[
m_7(\xi, \eta,\sigma)= i(-q(\xi-\sigma, \sigma)-q(\sigma, \xi-\sigma)) \big(a_{+,+}^4(\xi-\eta, \eta)- a_{+,-}^4(\xi-\eta, \eta)\big)\]
\[+i\big( -(-\sigma)\times(\xi-\eta)\big)\times
 \frac{(\xi-\eta-\sigma)\cdot(\xi-\sigma)}{|\xi-\eta-\sigma||\xi-\sigma|}\big(\overline{a_{+,-}^4(\xi-\sigma, -\eta)} - \overline{a_{+,+}^4(\xi-\sigma, -\eta)} \big)
 \]
 \[
= i\underbrace{\big(\sigma\times \xi - q(\xi-\sigma,\sigma)-q(\sigma, \xi-\sigma)\big)}_{\text{cancellation from (\ref{equation1241}) and (\ref{equation1483})}} \big(a_{+,+}^4(\xi-\eta, \eta)- a_{+,-}^4(\xi-\eta, \eta)\big) 
\]
\[ 
+ i\big(\sigma\times \xi\big)\underbrace{\big(\overline{a_{+,-}^4(\xi, -\eta)} - \overline{a_{+,+}^4}(\xi, -\eta)-a_{+,+}^4(\xi-\eta, \eta)+ a_{+,-}^4(\xi-\eta, \eta)\big) }_{\textup{cancellation from (\ref{equation12005})}} + \textup{good errors}.
\]
We only highlighted the cancellations in symbols above and omitted the detail formulas for good error terms. From above decompositions, (\ref{equation1241}), (\ref{equation1008}), (\ref{equation1483}), (\ref{equation1363}), and (\ref{equation12005}) and Lemma \ref{Snorm},  the following estimate holds when $\max\{k_2,k_3\}\leq k_1-10,$
\[
\|m_5(\xi, \eta, \sigma)\|_{\mathcal{S}^\infty_{k,k_1,k_2,k_3}} + \|m_6(\xi, \eta, \sigma)\|_{\mathcal{S}^\infty_{k,k_1,k_2,k_3}} + \|m_7(\xi, \eta, \sigma)\|_{\mathcal{S}^\infty_{k,k_1,k_2,k_3}}\lesssim 2^{3\max\{k_2,k_3\}},
\]
which further gives us the following estimate, 
\[
| J^{1,2}_{\alpha}(i\textup{Im}(\Phi))|\lesssim \|(G_1, G_2,\psi,  \Phi)\|_{X_{N_0}}^2 (\|(G_1,G_2,\psi)\|_{Z_1'} + \|\Phi\|_{Z'})^2
\]
\begin{equation}\label{equation63423}
\lesssim \|(\phi_0, \Phi)\|_{X_{N_0}}^2(\|\phi_0\|_{Z_1'} +\|\Phi\|_{Z'})^2 \lesssim (1+t)^{-1+2p_0} \epsilon_1^4\lesssim  (1+t)^{-1+2p_0} \epsilon_0^3.
\end{equation}
Hence, from (\ref{equation63421}), (\ref{equation63420}), and (\ref{equation63423}), we have
\begin{equation}\label{equation63425}
| J_{\alpha}^1|\lesssim (1+t)^{-1+2p_0}\epsilon_0^3.
\end{equation}
Similar to the procedures we did for $J_{\beta, \gamma}^2(\Phi^\gamma)$, with minor modifications, we can estimate $J_{\alpha}^2$ very similarly.  We omit the details here.

\section{Linear Decay Estimate}\label{sectiondecay}
\begin{lemma}[Linear Decay Estimate]\label{decay10}
For any $t\in \mathbb{R}$ and any suitable function $f(x)$, we have 
\begin{equation}\label{decay}
\| e^{i t |\nabla| } f \|_{Z'} \lesssim (1+ t)^{-1/2} \| f \|_{Z}  + (1+ t)^{-5/8} \big[\| f \|_{H^{N_0-1}} +
 \||x|\nabla f \|_{H^{N_1-1}}\big].
\end{equation}

\end{lemma}

\begin{proof}
We can rule out the very low frequency case first by using the size of support of $\xi$. More precisely, 
\[
\sum_{2^{k} \lesssim  (1+ t )^{-5/4}} 2^{(N_1+4)k_+}  \Big|  \int_{\R}  e^{i t |\xi| + x\cdot \xi}  \widehat{f}(\xi) \psi_{k}(\xi)  d \xi \Big| 
\leq  \sum_{2^{k}   
 \lesssim  (1+ t )^{-5/8}} 2^{k} \| P_{k} f \|_{L^{2}(\mathbb{R}^{2})}\]
 \[  \lesssim (1+ t )^{-5/8} \| f\|_{H^{N_0-1}},
\]
Then we can use the $H^{N_{0}-1}$ norm  to rule out the very high frequency case as follows,
\[
\sum_{2^{k} \gtrsim (1+t)^{5/8(N_1-5)}}   2^{(N_1+4)k_+} 
\Big| \int_{\R}  e^{i t |\xi| + x\cdot \xi}  \widehat{f}(\xi) \psi_{k}(\xi) 
 d \xi   \Big|
 \leq \sum_{2^{k} \gtrsim (1+t)^{5/8(N_1-5)}} 2^{-(N_1-5) k_+}  \| P_{k} f \|_{H^{N_0-1}} \]
 \[ \lesssim (1+ t )^{-5/8} \| f\|_{H^{N_0-1}}.
\]

From now on, we assume that $| t | \geq 1$, otherwise it would be straightforward.
 When $| x/ t| \leq 0.99 $ or $| x/ t | \geq 1.01$, we do integration by parts with respect to ``$\xi$", which gives us the following estimate, 
\[
 2^{(N_1+4)k_+} \Big|  \int_{\R}  e^{i t |\xi| + x\cdot \xi}  \widehat{f}(\xi)
 \psi_{k}(\xi)  d \xi \Big| \lesssim |t|^{-1} 2^{-k} 2^{(N_1+4)k_+} 
 \| \widehat{P_{k} f} (\xi ) \|_{L^{1}_{\xi}} + |t|^{-1} 2^{(N_1+4)k_+}  \| \p_{\xi} \widehat{P_{k} f}(\xi) \|_{L^{1}_{\xi}}
\]
\begin{equation}
\lesssim |t|^{-1} \|P_k f \|_{H^{N_0-1}} + |t|^{-1} 2^{5k_{+}}
 \| |x|\nabla f \|_{H^{N_1-1}} \lesssim |t|^{-1} \|P_k f \|_{H^{N_0-1}} + |t|^{-5/8} 
 \| |x|\nabla f \|_{H^{N_1-1}} .
\end{equation}

It remains to consider the case 
when $  |t|^{-5/4} \leq 2^{k} \leq  |t|^{5/8(N_1 -5)} $ and $0.99\leq |x|/t \leq 1.01$. Note that the phase $\Phi(\xi) = t|\xi| + x\cdot \xi$
 has a line of critical points, i.e., $\Phi^{'}(\xi) =0$ if ${\xi}/{|\xi|} = - {x}/{t} = -{x}/{|x|} =: \xi_{0} $. We first localize the angle of  $\xi$ with respect to $\xi_0$, and then use the size of support if it is close to the critical points and do integration by parts in ``$\xi$"  if it is away from the critical points. Let $\tilde{l}_{k}$ be the least integer such that  $2^{ \tilde{l}_{k}} \geq  |t|^{-1/2} 2^{-k/2}$, then 
\[
\sum_{ |t|^{-5/4} \leq 2^{k} \leq  |t|^{5/8(N_1 -5)}} 2^{(N_1+4)k_+}  \Big| 
 \int_{\R}  e^{i t |\xi| + x\cdot \xi} 
 \widehat{f}(\xi) \psi_{k}(\xi) \psi_{\leq \tilde{l}_k}(\xi/|\xi|-\xi_0)  d \xi \Big|  
\]
\[
\lesssim \sum_{  |t|^{-5/4} \leq 2^{k} \leq  |t|^{5/8(N_1 -5)} }  2^{ 2k + \tilde{l}_{k}} 2^{(N_1+4)k_+}  \| \widehat{P_{k} \psi}\|_{L^{\infty}}
\]
\begin{equation}\label{equation5205}
\lesssim \sum_{ |t|^{-5/4} \leq 2^{k} \leq |t| } 2^{ 3k/2} 2^{-2k_{+}} |t|^{-1/2} \| f\|_Z \lesssim  |t|^{-1/2} \| f\|_Z .
\end{equation}
Note that $|\nabla \Phi (\xi)| \geq |t|\,2^{l} $ when $| \xi/|\xi| - \xi_{0} | \sim 2^{l} $, hence after integration by parts in ``$\xi$", we have
\[
\sum_{ |t|^{-5/4} \leq 2^{k} \leq  |t|^{5/8(N_1 -5)}}  2^{(N_1+4)k_+} \sum_{\tilde{l}_{k} \leq l \leq2} \Big|  \int_{\R}  e^{i t |\xi| + x\cdot \xi}  \widehat{f}(\xi) \psi_{k}(\xi) \psi_{l}({\xi}/{|\xi|} - \xi_{0})  d \xi \Big| \]
\[
 \lesssim \sum_{|t|^{-5/4} \leq 2^{k} \leq  |t|^{5/8(N_1 -5)}}2^{(N_1+4)k_+}  
\sum_{\tilde{l}_{k} \leq l \leq2} \frac{1}{ |t| 2^{l} } 
\big(   2^{-(k+l)} 2^{2k+l} \| \widehat{P_{k}f}\|_{L^{\infty}} + 
2^{(2k+ l)/2} \|  \p_{\xi} \widehat{f}(\xi) \psi_{k}(\xi) \|_{L^{2}}    \big)
\]
\[
\lesssim  \sum_{|t|^{-5/4} \leq 2^{k} \leq  |t|^{5/8(N_1 -5)} } 
\sum_{\tilde{l}_{k} \leq l \leq2} \frac{1}{ |t|} 2^{- l } 2^{k-2k_{+}}\| f\|_Z
 + \frac{1}{|t|} 2^{- l/2} 2^{5 k_{+}}  \||x|\nabla f \|_{H^{N_1-1}}
\]
\[
\lesssim  \sum_{ |t|^{-5/4} \leq 2^{k} \leq  |t|^{5/8(N_1 -5)} } 2^{ 3k/2} 2^{ - 2k_{+} } \frac{1}{|t|^{1/2}} \| f\|_Z
+ \frac{1}{t^{3/4}} 2^{k/4} 2^{5 k_{+}}\||x|\nabla f \|_{H^{N_1-1}}
\]
\begin{equation}\label{equation5206}
 \lesssim \frac{1}{|t|^{1/2}} \| f\|_Z + \frac{1}{|t|^{5/8}} \||x|\nabla f \|_{H^{N_1-1}}.
\end{equation}
Combing (\ref{equation5205}) and (\ref{equation5206}), we  can see that (\ref{decay}) also holds for the remaining cases, therefore finishing the proof.
\end{proof}

\section{Proof of Proposition \ref{proposition2}}\label{improveddispersion}
\noindent The proof of Proposition \ref{proposition2} is separated into the following two steps:
\begin{enumerate}
\item[\emph{Step 1:}] Deriving the improved $Z$--norm estimates for $\Phi$, which further give us the improved dispersion estimate for $\Phi$, i.e., prove (\ref{step1}).

\item[\emph{Step 2:}] Deriving the improved estimate for $\phi_{0}$ via the bootstrap argument on the constraint, i.e., prove (\ref{step2}).
\end{enumerate}

\subsection{Improved $Z$--norm estimate and dispersion estimate for $\Phi$}\label{ZnormPhi}
Recall the equation satisfied by $\Phi$ in (\ref{mainequation10}), we can replace $\phi_0$ by $\mathcal{N}_2$ in $\tilde{Q}_{0,\mu}$ several times until the quartic terms  only depend on $\Phi$. More precisely, we can reformulate (\ref{mainequation10}) as follows, 
\begin{equation}\label{reformulateofPhi}
\p_t \Phi + i|\nabla|\Phi = Q_2 + C + Q_4 + \mathcal{R},\end{equation}
where
\[ 
Q_2= \sum_{(\mu, \nu)\in \mathcal{S}} \widetilde{Q}_{\mu, \nu}(\Phi_{\mu}, \Phi_{\nu}),\quad \mathcal{R}= \mathcal{N}_1 - Q_2 - C - Q_4,
\]
\begin{equation}\label{cubicinphi}
C= \sum_{\mu, \nu, \kappa \in \{+,-\}}\tilde{Q}_{0, \mu}(\tilde{Q}^1_{\nu, \kappa}(\Phi_{\nu}, \Phi_{\kappa}), \Phi_{\mu}),
\end{equation}
\begin{equation}\label{quarticinphi}
Q_4= \sum_{\mu,\nu, \kappa, \tau 
\in\{+,-\}} \tilde{Q}_{0,0}(\tilde{Q}^1_{\mu, \nu}(\Phi_{\mu}, \Phi_{\nu}), \tilde{Q}^1_{\kappa, \tau}(\Phi_{\nu}, \Phi_{\kappa})) + \tilde{Q}_{0, \mu}(\tilde{Q}_{0, \nu}^1(\tilde{Q}^1_{\kappa, \tau}(\Phi_{\kappa}, \Phi_{\tau}), \Phi_{\nu})),\Phi_{\mu}),
\end{equation}
and it is  not difficult to see that, in the sense of decay rate, ``$\mathcal{R}$" is of quintic and higher. Define the associated profile of ${\Phi}$ as ${g}(t) = e^{ i t |\nabla| }{\Phi}(t)$,
it follows that
\begin{equation}\label{equationforprofile}
\p_t g(t) =  e^{i t \d}[Q_2 + C + Q_4 + \mathcal{R}].
\end{equation}
Recall the normal form transformation defined in (\ref{normalformtran}), we define the associated profile of $\tilde{\Phi}$ as $\tilde{g}(t)= e^{i t\d}\tilde{\Phi}$, it follows that 
\[
\p_t \tilde{g}(t) = e^{i t \d}\Big[C + Q_4 + \mathcal{R} + \sum_{(\mu, \nu)\in \mathcal{S}} A_{\mu, \nu}(\Phi_{\mu}, P_{\nu}(  Q_2 + C + Q_4 + \mathcal{R} ))\]
\begin{equation}\label{eqn5002}
 + A_{\mu, \nu}(P_{\mu}(Q_2 + C + Q_4 + \mathcal{R}), \Phi_{\nu})\Big].
\end{equation}
From (\ref{equation8}) and the bootstrap assumption (\ref{smallness}), we have the following estimates 
\begin{equation}\label{smallnessprofile}
\sup_{t\in[0,T]}  (1+t)^{1/2}\| e^{-it|\nabla|} g\|_{Z'}+(1+t)^{-2p_0}\| g\|_{Z}\lesssim \epsilon_1,  \sup_{t\in[0,T]} (1+t)^{-p_0}\big[\| g\|_{H^{N_0}} + \| \tilde{g}\|_{H^{N_0-1}} \big]\lesssim \epsilon_0,
\end{equation}
\begin{equation}\label{eqn5001}
 \sup_{t\in[0,T]} (1+t)^{-p_0}\big[\|  \mathcal{F}^{-1}[|\xi|\nabla_{\xi}\widehat{{g}}(\xi)](\cdot)\|_{H^{N_1}} +\|  \mathcal{F}^{-1}[|\xi|\nabla_{\xi}\widehat{\tilde{g}}(\xi)](\cdot)\|_{H^{N_1-1}}\big] \lesssim \epsilon_0.
\end{equation}
We postpone the proof of estimate (\ref{eqn5001}) to the end of this section and take this estimate as granted first.

\subsubsection{Proof of \textup{(\ref{step1})}}

From the results in Lemma \ref{Zremainder}, Lemma \ref{cubicandquartic} and Lemma \ref{Zquadratic}, we have
\[
\sup_{t\in[0,T]}(1+t)^{-2p_0} \| g\|_{Z}\lesssim \| g(0)\|_{Z} + \sup_{t\in[0,T]}(1+t)^{-2p_0} \Big\|\int_{0}^t \p_t g \Big\|_{Z}\lesssim \epsilon_0,
\]
\[
\sup_{t\in[0,T]} \| \widetilde{g}(t) \|_{Z}\lesssim \|g(0)\|_{Z} + \sup_{t\in[0, T] } \| \int_{0}^t \p_t g(s) d s +\sum_{(\mu, \nu)\in \mathcal{S}}e^{i t\d} \big[A_{\mu, \nu}(\Phi_{\mu}, \Phi_{\nu})\big]\|_{Z} \lesssim \epsilon_0.
\]
From the linear decay estimate (\ref{decay}) in Lemma \ref{decay10}, we have
\[
\sup_{t\in[0,T]}(1+t)^{1/2}  \| \widetilde{\Phi}(t)\|_{Z'} \lesssim \epsilon_0 + \sup_{t\in[0,T]} (1+t)^{-1/8}\big[ \|  \widetilde{g} 
\|_{H^{N_0-1}} + \|\mathcal{F}^{-1}[|\xi|\nabla_{\xi}\widehat{\widetilde{g}}(\xi)]\|_{H^{N_1-1}} \big]\lesssim \epsilon_0,
\]
which further gives us the following estimate,
\[
\sup_{t\in[0,T]} (1+t)^{1/2}\| \Phi(t)\|_{Z'} \lesssim \sup_{t\in[0,T]} (1+t)^{1/2}\| \widetilde{\Phi}\|_{Z'} 
+ \sup_{t \in[0,T]} (1+t)^{1/2}\| \sum_{(\mu, \nu)\in\mathcal{S}}A_{\mu, \nu}(\Phi_{\mu}, \Phi_{\nu})\|_{Z'} \]
\[
\lesssim \epsilon_0  +  \sup_{t \in[0,T]} (1+t)^{1/2} \|\Phi\|_{Z'}^{3/2}\|\Phi\|_{H^{N_0}}^{1/2}\lesssim \epsilon_0.
\]
Therefore (\ref{step1}) holds.
\subsubsection{$Z$-norm estimate for the remainder terms} We can first estimate the remainder term very easily and have the following lemma:
\begin{lemma}\label{Zremainder}
Under the bootstrap assumption \textup{(\ref{smallness})} and the energy estimate (\ref{improvedenergyestimate}), we have
\begin{equation}\label{remainder}
\| \mathcal{R}\|_{Z}\lesssim (1+t)^{-3/2+2p_0}\epsilon_1^5,
\end{equation}
which further gives us
\[
\sup_{t\in[0,T]} \| \int_{0}^{t} e^{i s |\nabla|}\mathcal{R} ds\|_{Z}\lesssim \epsilon_0.
\]
\end{lemma}
\begin{proof}
Since $\mathcal{R}$ is of quintic and higher, by multilinear estimate, it is  easy to derive 
\[
\| \mathcal{R}\|_{Z}\lesssim \|  \Phi\|_{X_{N_0}}^2\big[\|\Phi\|_{Z'}^3 + \| \phi_0\|_{Z'_1} \|\Phi\|_{Z'} \big]+ \| \phi_0\|_{X_{N_0}}^2 \|\Phi\|_{Z'}\]
\[
 + \| \phi_0\|_{X_{N_0}}\|\Phi\|_{X_{N_0}}(\Phi\|_{Z'}^2+\|\phi_0\|_{Z'_1})
\lesssim (1+t)^{-3/2+2p_0}\epsilon_1^5.
\]
\end{proof}
\subsubsection{$Z$-norm estimate for the cubic and quartic terms}

Next, we proceed to estimate ${C}$ and ${Q}_4$, it turns out that we can treat one of the bilinear terms inside ${Q}_4$ as a single input and then estimate ${C}$ and ${Q}_4$ in the same way. More precisely, we have the following lemma,
\begin{lemma}\label{cubicandquartic}
Under the bootstrap assumption \textup{(\ref{smallness})} and the energy estimate (\ref{improvedenergyestimate}), we have
\begin{equation}\label{eqn3000}
\| C\|_{Z} + \|Q_4\|_{Z}\lesssim (1+t)^{-7/5-p_0} \epsilon_1^3,
\end{equation}
which further gives us the following estimate,
\[
\sup_{t\in[0,T]}\|\int_{0}^t e^{it \d} C d s\|_{Z} + \|\int_0^t e^{it \d} Q_4 d s\|_{Z}\lesssim  \epsilon_0^2.
\]
\end{lemma}
\begin{proof}
 it is  sufficient to consider the case when $t\geq 1$, otherwise, it is  trivial. To prove (\ref{eqn3000}), essentially speaking,
we only need to estimate the following trilinear form in $Z$-normed space for any possible signs $\mu, \nu, \kappa\in\{+,-\}$,
\begin{equation}\label{generaltrilinear}
T(\tilde{Q}_{\mu, \nu}^1(f_{\mu}, g_{\nu}), h_{\kappa}), \quad T\in \{\tilde{Q}_{0, +}(\cdot, \cdot), \tilde{Q}_{0, -}(\cdot, \cdot), \tilde{Q}_{0,0}(\cdot, \cdot)\},
\end{equation}
where $f$, $g$, and $h$ are well defined functions and they satisfy the following estimate,
\begin{equation}\label{smallness10}
\sup_{t\in[0,T]} (1+t)^{-p_0} \|(f,g,h)\|_{X_{N_0}} + (1+t)^{1/2}\|(f,g,h)\|_{Z'}  + (1+t)^{-2p_0} \| (f,g,h)\|_{Z}\lesssim \epsilon_1.
\end{equation}
Define the associated profiles of $f,g,$ and $h$ as  $\tilde{f}(t):=e^{it \d} f(t)$, $\tilde{g}(t):=e^{it \d} g(t)$, and $\tilde{h}(t):=e^{it \d} h(t)$ respectively. Write above trilinear form on the Fourier side, we have
\[
\mathcal{F}\big[T(\tilde{Q}_{\mu, \nu}^1(f_{\mu}, g_{\nu}), h_{\kappa})\big](\xi) =\int_{\R}\int_{\R} e^{it \Phi_{\mu, \nu,\kappa}(\xi, \eta,\sigma)} \widehat{\tilde{f}_{\mu}}(\xi-\sigma)\widehat{\tilde{g}_{\nu}}(\sigma-\eta) \widehat{\tilde{h}_{\kappa}}(\eta) m_{\mu, \nu}(\xi,\eta, \sigma) d \eta d \sigma, 
\]
where
\[
\Phi_{\mu, \nu, \kappa}(\xi, \eta, \sigma)= -\mu|\xi-\sigma|-\nu|\sigma-\eta| - \kappa |\eta|,\quad 
 m_{\mu, \nu}(\xi,\eta, \sigma)= \tilde{m}_{\mu, \nu}^1(\xi-\sigma, \sigma-\eta) t(\xi-\eta, \eta),
\]
and $t(\cdot,\cdot)$  is the associated symbol  of the bilinear operator $T(\cdot, \cdot)$.  

We can first rule out the very high frequency case by $L^2-L^2-L^\infty$ type  estimate as follows,
\[
\sup_{ 2^{k}\geq (1+t)^{1/N_1}} \big\|P_{k}\big[T(\tilde{Q}_{\mu, \nu}^1(f_{\mu}, g_{\nu}), h_{\kappa})\big]\big\|_{Z} \]
\begin{equation}\label{eqn3016}
\lesssim \sup_{2^{k}\geq (1+t)^{1/N_1}} 2^{-(N_0-N_1-10) k} \| (f,g,h)\|_{H^{N_0}}^2 \| (f,g,h)\|_{Z'} \lesssim (1+t)^{-7/5-10p_0}\epsilon_1^3.
\end{equation}
For the remaining  case, we do integration by parts in ``$\sigma$". More precisely, after integrating by parts in $\sigma$, we have
\begin{equation}\label{eqn3014}
\Big|\mathcal{F}\big[T(\tilde{Q}_{\mu, \nu}^1(P_{\mu}f, P_{\nu}g), P_{\kappa} h)\big](\xi)\Big| \lesssim \frac{1}{t}\big[ |I_{\mu, \nu,\kappa}^1(\xi) | +  |I_{\mu, \nu,\kappa}^2(\xi) |+  |I_{\mu, \nu,\kappa}^3(\xi) | \big], 
\end{equation}
where
\[
I_{\mu, \nu,\kappa}^1(\xi) = \int_{\R}\int_{\R} e^{it \Phi_{\mu, \nu,\kappa}(\xi, \eta,\sigma)} \widehat{\tilde{f}_{\mu}}(\xi-\sigma)\widehat{\tilde{g}_{\nu}}(\sigma-\eta) \widehat{\tilde{h}_{\kappa}}(\eta) \nabla_{\sigma}\cdot \hat{m}_{\mu, \nu}(\xi,\eta, \sigma) d \eta d \sigma, 
\]
\[
I_{\mu, \nu,\kappa}^2(\xi) = \int_{\R}\int_{\R} e^{it \Phi_{\mu, \nu,\kappa}(\xi, \eta,\sigma)} \widehat{\tilde{f}_{\mu}}(\xi-\sigma) \nabla_{\sigma}\cdot\widehat{\tilde{g}_{\nu}}(\sigma-\eta)  \hat{m}_{\mu, \nu}(\xi,\eta, \sigma) \widehat{\tilde{h}_{\kappa}}(\eta) d \eta d \sigma, 
\]
\[
I_{\mu, \nu,\kappa}^3(\xi) = \int_{\R}\int_{\R} e^{it \Phi_{\mu, \nu,\kappa}(\xi, \eta,\sigma)} \nabla_{\sigma}\widehat{\tilde{f}_{\mu}}(\xi-\sigma)\cdot \hat{m}_{\mu, \nu}(\xi,\eta, \sigma)\widehat{\tilde{g}_{\nu}}(\sigma-\eta) \widehat{\tilde{h}_{\kappa}}(\eta)  d \eta d \sigma, 
\]
\begin{equation}\label{equation13200}
\hat{m}_{\mu, \nu}(\xi,\eta, \sigma) = \frac{\nabla_{\sigma}\Phi_{\mu, \nu, \kappa}(\xi, \eta, \sigma)}{|\nabla_{\sigma}\Phi_{\mu, \nu, \kappa}(\xi, \eta, \sigma)|^2}{m}_{\mu, \nu}(\xi,\eta, \sigma)= \frac{\mu \frac{\xi-\sigma}{|\xi-\sigma|}-\nu \frac{\sigma-\eta}{|\sigma-\eta|}}{\big|\mu \frac{\xi-\sigma}{|\xi-\sigma|}-\nu \frac{\sigma-\eta}{|\sigma-\eta|}\big|^2} \tilde{m}_{\mu, \nu}^1(\xi-\sigma, \sigma-\eta) t(\xi-\eta, \eta).
\end{equation}
Recall that there is  a strong null structure inside the symbol $\tilde{m}^1_{\mu,\nu}(\cdot, \cdot)$ (see (\ref{eqn3002}) and (\ref{eqn3001})), in any case, we can gain one degree of angle between $\xi-\sigma$ and $\sigma-\eta$, which compensates the loss in denominator.

From Lemma \ref{boundness}, Lemma \ref{Snorm} and the discussion in subsection \ref{guide}, the following estimates for the symbol $\hat{m}_{\mu, \nu}(\xi,\eta, \sigma)$ hold,
\begin{equation}\label{eqn8900}
\|\mathcal{F}^{-1}[\hat{m}_{\mu, \nu}(\xi,\eta, \sigma)\psi_{k}(\xi)\psi_{k_1}(\xi-\eta)\psi_{k_2}(\eta)\psi_{k_1'}(\xi-\sigma)\psi_{k_2'}(\sigma-\eta)]\|_{L^1}
\lesssim 2^{\min\{k_1',k_2'\}+k_1+k_2},
\end{equation}
\begin{equation}\label{eqn8901}
\|\mathcal{F}^{-1}[\nabla_{\sigma}\cdot\hat{m}_{\mu, \nu}(\xi,\eta, \sigma)\psi_{k}(\xi)\psi_{k_1}(\xi-\eta)\psi_{k_2}(\eta)\psi_{k_1'}(\xi-\sigma)\psi_{k_2'}(\sigma-\eta)]\|_{L^1}
\lesssim 2^{\max\{k_1',k_2'\}+k_2}.
\end{equation}
Hence, from $L^2-L^\infty$ type bilinear estimate in Lemma \ref{boundness}, the following estimate holds,
\[
\sup_{2^{k}\leq (1+t)^{1/N_1} }\sum_{i=1,2,3}\| \mathcal{F}^{-1}[I_{\mu, \nu,\kappa}^i(\cdot)\psi_{k}(\cdot)]\|_{Z}\lesssim (1+t)^{10/N_1}  \big[\| (f,g,h)\|_{H^{N_0}} +  \||\xi|\nabla_{\xi}(\widehat{\tilde{f}}(\xi), \widehat{\tilde{g}}(\xi)\|_{H^{N_1}} \big]^2
\]
\begin{equation}\label{eqn3017}
 \times \|(f,g,h)\|_{Z'}\lesssim  (1+t)^{10/N_1} \|(f,g,h)\|_{X_{N_0}}^2 \| (f,g,h)\|_{Z'}\lesssim  (1+t)^{-2/5-10p_0}\epsilon_1^3.
\end{equation}
To sum up, after combining (\ref{eqn3016}), (\ref{eqn3014}) and (\ref{eqn3017}), we can see the following estimate holds under the smallness assumption (\ref{smallness10}),
\begin{equation}\label{generaltype}
\| T(\tilde{Q}_{\mu, \nu}^1(P_{\mu}f, P_{\nu}g), P_{\kappa} h)\|_{Z} \lesssim (1+t)^{-7/5-10p_0} \epsilon_1^3.
\end{equation}
From the explicit formula of ``$C$" in (\ref{cubicinphi}) and the bootstrap assumption, we can immediately derive the improved  $Z$-norm estimate for ``$C$". 

Let us  proceed to consider the quintic term $Q_4$. Recall the explicit formula of $Q_4$ in (\ref{quarticinphi}) and then let $h:=\tilde{Q}_{\kappa, \tau}^1(\Phi_{v}, \Phi_{\kappa})$, we have \[
\tilde{Q}_{0,0}(\tilde{Q}^1_{\mu, \nu}(\Phi_{\mu}, \Phi_{\nu}), \tilde{Q}^1_{\kappa, \tau}(\Phi_{\nu}, \Phi_{\kappa})) = \tilde{Q}_{0,0}(\tilde{Q}^1_{\mu, \nu}(\Phi_{\mu}, \Phi_{\nu}), h),
\]
and 
\[
\sup_{t\in[0, T]} (1+t)^{-p_0}\| h \|_{X_{N_0}} + (1+t)^{1/2}\| h \|_{Z'} + (1+t)^{-2p_0}\| h \|_{Z} \lesssim \epsilon_1.
\]

Therefore, we can use the derived general type $Z$-norm estimate (\ref{generaltype}) for the first term on the right hand side of (\ref{quarticinphi}). For the second term on  the right hand side of (\ref{quarticinphi}), we can treat the trilinear form of type (\ref{generaltrilinear}) as a input of bilinear operator $\tilde{Q}_{0, \mu}(\cdot, \cdot)$, then the $Z$-norm estimate will be straightforward. To sum up,  we have
\[
\| Q_4\|_{Z}\lesssim (1+t)^{-7/5-10p_0} \epsilon_1^3 + (1+t)^{-7/5-10p_0} \epsilon_1^3 \| \Phi\|_{H_{N_0}} \lesssim (1+t)^{-7/5-p_0} \epsilon_1^3,
\]
which implies that (\ref{eqn3000}) holds, hence finishing the proof.
\end{proof}
\subsubsection{$Z$-norm estimate for the quadratic terms}
Lastly, we consider the quadratic terms and we have the following lemma. 
\begin{lemma}\label{Zquadratic}
Under the bootstrap assumption \textup{(\ref{smallness})} and the energy estimate (\ref{improvedenergyestimate}), we have
\begin{equation}\label{quadratic}
\sup_{t\in[0,T]} (1+t)^{-2p_0} \| \int_{0}^t e^{is \d} Q_2 d  s\|_{Z} \lesssim \epsilon_0.
\end{equation}
\begin{equation}\label{anotherZestimate}
\sup_{t\in[0,T]} \| \int_{0}^{t} e^{is \d} Q_2  d s +\sum_{(\mu, \nu)\in \mathcal{S}}e^{i t\Lambda} \big[A_{\mu, \nu}(\Phi_{\mu}, \Phi_{\nu})\big]\|_{Z}\lesssim \epsilon_0.
\end{equation}

\end{lemma}
\begin{proof}
 We  write  $\mathcal{Q}_2$ on the Fourier side in terms of profile $g$ and have the following, 
\begin{equation}\label{eqn50000}
\mathcal{F}(\int_{0}^t e^{is \d} Q_2 d s )(\xi)=\sum_{(\mu, \nu)\in \mathcal{S}}\int_{0}^{t} \int_{\R} e^{is \Phi_{\mu, \nu}(\xi, \eta)}  \widehat{g_{\mu}}(s,\xi-\eta) \widehat{g_{\nu}}(s,\eta) \tilde{m}'_{\mu, \nu}(\xi-\eta, \eta) d \eta ds. 
\end{equation}
For this case, we do integration by parts in time and have the following identity,
\begin{equation}\label{eqn50001}
\mathcal{F}(\int_{0}^t e^{is \d} Q_2 d s )(\xi) = \sum_{(\mu, \nu)\in \mathcal{S}}\mathcal{J}_{1}^{\mu, \nu}(\xi) + \mathcal{J}_{2}^{\mu, \nu}(\xi) + \mathfrak{End}_{1}^{\mu, \nu}(\xi)- \mathfrak{End}_{0}^{\mu, \nu}(\xi), \end{equation}
where
\[
\mathcal{J}_{1}^{\mu, \nu}(\xi)=\int_{0}^{t} \int_{\R} e^{is \Phi_{\mu, \nu}(\xi, \eta)}  \p_s\widehat{g_{\mu}}(s,\xi-\eta) \widehat{g_{\nu}}(s,\eta) a_{\mu, \nu}(\xi-\eta, \eta)  d \eta ds,
\]
\[
\mathcal{J}_{2}^{\mu, \nu, m}(\xi)=\int_{0}^{t} \int_{\R} e^{is \Phi_{\mu, \nu}(\xi, \eta)}  \widehat{g_{\mu}}(s,\xi-\eta) \p_s\widehat{g_{\nu}}(s,\eta) a_{\mu, \nu}(\xi-\eta, \eta)  d \eta ds,
\]
\begin{equation}\label{eqn60000}
\mathfrak{End}_{1}^{\mu, \nu}(\xi) = -\int_{\R} e^{it \Phi_{\mu, \nu}(\xi, \eta)}  \widehat{g_{\mu}}(t,\xi-\eta)\widehat{g_{\nu}}(t,\eta) a_{\mu, \nu}(\xi-\eta, \eta) d \eta=- \mathcal{F}\big[ e^{i t \d}A_{\mu, \nu}(\Phi_{\mu}(t),\Phi_{\nu}(t))\big](\xi),
\end{equation}
\[
\mathfrak{End}_{0}^{\mu, \nu}(\xi) = - \mathcal{F}\big[ A_{\mu, \nu}(\Phi_{\mu}(0),\Phi_{\nu}(0))\big](\xi).
\]

 Combing (\ref{eqn60000}) and (\ref{eqn50001}), we also have the following equality, 
\[
\mathcal{F}(\int_{0}^t e^{is \d} Q_2 d s )(\xi)+  \sum_{(\mu, \nu)\in \mathcal{S}}\mathcal{F}\Big[ e^{i t\Lambda} A_{\mu, \nu}(\Phi_{\mu}(t), \Phi_{\nu}(t)) \Big](\xi)
\]
\begin{equation}\label{eqn60001}
= \sum_{(\mu, \nu)\in \mathcal{S}}\mathcal{F}\Big[ A_{\mu, \nu}(\Phi_{\mu}(0), \Phi_{\nu}(0))\Big](\xi)+\mathcal{J}_{1}^{\mu, \nu}(\xi) + \mathcal{J}_{2}^{\mu, \nu}(\xi) .
\end{equation}
From $L^2-L^2$ type bilinear estimate, we can estimate the endpoint cases as follows,
\begin{equation}\label{eqn60004}
\|\mathcal{F}^{-1}(\mathfrak{End}_{1}^{\mu, \nu}(\cdot))\|_Z \lesssim \| g(t) \|_{H^{N_0}}^2 \lesssim (1+t)^{2p_0}\epsilon_1^2 \lesssim (1+t)^{2p_0}\epsilon_0,\quad \|\mathcal{F}^{-1}(\mathfrak{End}_{0}^{\mu, \nu}(\cdot))\|_Z \lesssim \epsilon_0.
\end{equation}

Since, $\mathcal{J}^{\mu, \nu}_{2}(\xi)$ can be estimate in the same way as $\mathcal{J}^{\mu, \nu}_{1}(\xi)$, we only  estimate $\mathcal{J}^{\mu, \nu}_{1}(\xi)$ in details here. We can plug in the equation satisfied by $\p_t g$ (see (\ref{equationforprofile})) and have the following 
\[
\mathcal{J}^{\mu, \nu}_{1}(\xi) = \mathcal{C}^1+\mathcal{Q}^1_4 + \mathcal{R}^1,
\]
where
\[
\mathcal{Q}^1_4 = \int_{0}^{t} \int_{\R} e^{is \Phi_{\mu, \nu}(\xi, \eta)}  \widehat{P_{\mu}[e^{is \d} C]}(s,\xi-\eta) \widehat{g_{\nu}}(s,\eta) a_{\mu, \nu}(\xi-\eta, \eta) d \eta ds,
\]
\[
\mathcal{R}^1 = \int_{0}^{t} \int_{\R} e^{is \Phi_{\mu, \nu}(\xi, \eta)}  \widehat{P_{\mu}[e^{is \d}[Q_4+\mathcal{R}]]}(s,\xi-\eta) \widehat{g_{\nu}}(s,\eta) a_{\mu, \nu}(\xi-\eta, \eta)  d \eta ds,
\]
\[
\mathcal{C}^1= \sum_{k_1,k_2\in\mathbb{Z}}\sum_{(\kappa, \tau)\in \mathcal{S}}
\int_{0}^{t} \int_{\R} \int_{\R} e^{is \Phi_{\mu,\nu}^{\kappa, \tau}(\xi, \eta,\sigma)}\widehat{P_{\mu}[g_{\kappa}]}(s,\xi-\sigma)\times\]
\[ \widehat{P_{\mu}[g_{\tau}]}(s,\sigma-\eta) \widehat{g_{\nu}}(s,\eta) b_{\mu, \nu, \kappa, \tau}^{k_1,k_2}(\xi,\eta, \sigma) d \eta d\sigma ds,
\]
and
\[
\Phi_{\mu,\nu}^{\kappa, \tau}(\xi, \eta,\sigma)= |\xi|-\mu \kappa|\xi-\sigma|-\mu\tau |\sigma-\eta|-\nu |\eta|,
\]
\[
b_{\mu, \nu, \kappa,\tau}^{k_1,k_2}(\xi,\eta, \sigma)= \tilde{m}'_{\kappa,\tau}(\xi-\sigma, \sigma-\eta) a_{\mu, \nu}(\xi-\eta, \eta) \psi_{k_1}(\xi-\eta)\psi_{k_2}(\eta).
\]

As $\mathcal{R}^1$ is of quintic and higher, we can estimate it in the same way as we did for  $\mathcal{R}$ in Lemma \ref{remainder}.  Moreover, we can estimate  $\mathcal{Q}_4^1$ in the same way as we did for $C$ and $Q_4$ in Lemma \ref{cubicandquartic},  because of the presence of bilinear operator $\tilde{Q}_{\mu,\nu}^1(\cdot, \cdot)$ in the term  $\mathcal{Q}_4^1$. We omit the details for those cases here. 
It remains to estimate $\mathcal{C}^1.$ For this case, we will use \emph{integration by parts in ``$\sigma$"}. This method has been  used to estimate  ``$C$"  in the proof of Lemma \ref{cubicandquartic}, to estimate  $\mathcal{C}^{1}$  without any problem.

To see this point, we  consider the following symbol\[
\widehat{b}_{\mu,\nu,\kappa, \tau}^{k_1,k_2}(\xi, \eta, \sigma)= \frac{\nabla_{\sigma} \Phi_{\mu,\nu}^{\kappa, \tau}(\xi, \eta,\sigma)}{|\nabla_\sigma \Phi_{\mu,\nu}^{\kappa, \tau}(\xi, \eta,\sigma)|^2} b_{\mu, \nu, \kappa,\tau}^{k_1,k_2}(\xi,\eta, \sigma),
\]
which is the symbol inside the trilinear term after integration by parts in ``$\sigma$". The main difference between the symbol $\widehat{b}_{\mu,\nu,\kappa, \tau}^{k_1,k_2}(\xi, \eta, \sigma)$ and the symbol $\hat{m}_{\mu, \nu}(\xi,\eta, \sigma)$ in (\ref{equation13200}) is the difference between $\tilde{m}_{\mu, \nu}'(\xi-\sigma,\sigma- \eta)$ and $\tilde{m}_{\mu, \nu}^1(\xi-\sigma, \sigma-\eta)$. 

Recall the detail formulas of $\tilde{m}_{\mu, \nu}'(\xi-\eta, \eta)$ and $\tilde{m}_{\mu, \nu}^1(\xi-\eta, \eta)$ in (\ref{twosame}), (\ref{twoopposite}), (\ref{equation13202}), (\ref{equation13203}), (\ref{eqn3002}), (\ref{eqn3001}).   Note that, after using the angle between $\xi-\eta$ and $\eta$, a potential problem for $\tilde{m}_{\mu, \nu}'(\xi-\eta, \eta)$ is that the size will be very big when $|\xi|\ll |\xi-\eta|\sim |\eta|$  and $|\xi|\ll 1$. While, this problem is not a issue for $\tilde{m}_{\mu, \nu}^1(\xi-\eta, \eta)$. We will show that, after a more careful study of symbols, it is actually not a problem for $\tilde{m}_{\mu, \nu}'(\xi-\eta, \eta)$.

Correspondingly, it is sufficient  to check for the case when $|\xi-\eta|\ll |\xi-\sigma|\sim|\sigma-\eta|$ for $\widehat{b}_{\mu,\nu,\kappa, \tau}^{k_1,k_2}(\xi, \eta, \sigma)$. Note that if $\kappa$ and $\tau$ have the same sign then $\sigma-\xi$ and $\sigma-\xi+\xi-\eta=\sigma-\eta$ are almost in the same direction for the case we are considering, hence
\[
|\nabla_{\sigma} \Phi_{\mu,\nu}^{\kappa, \tau}(\xi, \eta,\sigma)| = \Big|\frac{\sigma-\xi}{|\sigma-\xi|} + \frac{\sigma-\eta}{|\sigma-\eta|}\Big| \sim 1.
\]
That is to say, if $\kappa\tau=+$, then we do not need to use the size of angle. Hence, it is not a issue when $\xi-\eta$ is very small.

For the case when $\kappa$ and $\tau$ have different sign, i.e, $(\kappa, \tau)=(+,-)$ and we have
\[
|\nabla_{\sigma} \Phi_{\mu,\nu}^{\kappa, \tau}(\xi, \eta,\sigma)| = \Big|\frac{\sigma-\xi}{|\sigma-\xi|}- \frac{\sigma-\eta}{|\sigma-\eta|}\Big|.
\]
Recall (\ref{twoopposite}). Note that,  the $(1+\cos(\xi-\sigma, \sigma-\eta))$ part of $\tilde{m}'_{+,-}(\xi-\sigma, \sigma-\eta) $ is sufficient to compensate the loss  of angle in the denominator part of $\widehat{b}_{\mu,\nu,\kappa, \tau}^{k_1,k_2}(\xi, \eta, \sigma)$. Hence the size of symbol is not big even if $\xi-\eta$ is very small. From above discussion and Lemma \ref{Snorm}, the following estimate holds,
\begin{equation}\label{eqn5000}
\|\mathcal{F}^{-1}[\widehat{b}_{\mu,\nu,\kappa, \tau}^{k_1,k_2}(\xi, \eta, \sigma) \psi_{k}(\xi)\psi_{k_1'}(\xi-\sigma)\psi_{k_2'}(\sigma-\eta)]\|_{L^1}
\lesssim 2^{k_1'+k_2' + \max\{k_1,k_2\}},
\end{equation}
\begin{equation}\label{eqn5009}
\|\mathcal{F}^{-1}[\nabla_{\sigma}\cdot\widehat{b}_{\mu,\nu,\kappa, \tau}^{k_1,k_2}(\xi, \eta, \sigma) \psi_{k}(\xi)\psi_{k_1'}(\xi-\sigma)\psi_{k_2'}(\sigma-\eta)]\|_{L^1}
\lesssim 2^{\max\{k_1',k_2'\}+ \max\{k_1,k_2\}}.
\end{equation}
Therefore, the method used in the estimate of  ``$C$"  in the proof of Lemma \ref{cubicandquartic} can be applied to estimate  $\mathcal{C}^{1}$  without any problem. 
 To sum up, we have
\[
\|\mathcal{F}^{-1}[\mathcal{J}^{\mu, \nu}_{1}(\xi)]\|_Z \lesssim \|\mathcal{F}^{-1}[ \mathcal{C}^1]\|_Z  + \|\mathcal{F}^{-1}[ \mathcal{Q}_4^1]\|_Z + \|\mathcal{F}^{-1}[  \mathcal{R}^1]\|_{Z}\lesssim 2^{-p_0 m}\epsilon_1^3 \lesssim 2^{-p_0 m} \epsilon_0^2.
\]
From above estimate and  (\ref{eqn60004}), it is easy to see our desired estimates (\ref{quadratic}) and (\ref{anotherZestimate}) hold.
\end{proof}

\subsection{Improved estimate for $\phi_{0}$ via bootstrap argument on the constraint}

\begin{lemma}\label{improvedconstraint}
With the improved estimate we have proven for $\Phi$ as follows,
\begin{equation}\label{equation1060}
\sup_{t\in[0,T]} (1+t)^{-p_{0}} \| (\phi_{0}, \Phi)\|_{X_{N_{0}}} + (1+ t)^{1/2} \|\Phi\|_{Z'} \lesssim \epsilon_{0},
\end{equation}
 we have the following improved estimate for $\phi_{0}$,
\begin{equation}\label{equation1061}
\sup_{t\in[0,T]} (1+t)^{1/2-p_{0}} \| \phi_{0} \|_{X_{N_{0}}} + (1+t) \| \phi_{0} \|_{Z'_{1}} \lesssim (\epsilon_{0} + \epsilon_{1})^{2}\lesssim \epsilon_0^2.
\end{equation}
\end{lemma}

\begin{proof}
From the constraint equation $\phi_{0} = \mathcal{N}_{2}$ and the estimates in Lemma \ref{lemmal2}, we have the following estimates for fixed $t\in[0,T],$
\[
\| \phi_{0} (t)\|_{X_{N_{0}}} \lesssim \big( \| \phi_{0}\|_{X_{N_{0}}}  + \| \Phi\|_{X_{N_{0}}}\big)(\| \phi_{0}\|_{Z_{1}'} + \| \Phi\|_{Z'}) \lesssim \frac{1}{(1+t)^{1/2-p_{0}}} (\epsilon_{0}+ \epsilon_{1})^{2}, \]
\[
\| \phi_{0}(t)\|_{Z'_{1}} \lesssim \| \Phi\|_{Z'}^{2} + \| \phi_{0}\|_{Z'_{1}} \| \Phi\|_{Z'} +  (\| \phi_{0}\|_{Z'_{1}} \| \Phi\|_{Z'})^{3/4} \| (\phi_{0}, \Phi)\|_{X_{N_{0}}}^{1/4} + \| \phi_{0}\|_{Z'_{1}}^{2} +  \| \phi_{0}\|_{Z'_{1}}^{3/2} \| \phi_{0}\|_{X_{N_{0}}}^{1/2}
\]
\[
\lesssim ( \frac{1}{1+t} + \frac{1}{(1+t)^{9/8- p_{0}/4}}) (\epsilon_{0} + \epsilon_{1})^{2} \lesssim  \frac{1}{1+t}  (\epsilon_{0} + \epsilon_{1})^{2}.
\]
Therefore (\ref{equation1061}) holds.
\end{proof}

\subsection{Proof of (\ref{eqn5001})}

We assume that $|t|\geq 1$, otherwise it is trivial.  Note that 
\begin{equation}\label{eqn3399}
\| \mathcal{F}^{-1}[|\xi|\nabla_{\xi}\widehat{g}(\xi)](\cdot)\|_{H^{N_1}} \lesssim \| S g \|_{H^{N_1}} + \| \Omega g\|_{H^{N_1}} + \| t\p_t g \|_{H^{N_1}},
\end{equation}
\[
S g = e^{it\d} S\Phi, \quad \Omega g = e^{it \d} \Omega \Phi,
\]
and a very similar estimate also holds for $\tilde{g}$. Hence it's sufficient to estimate $t\p_t g$ in $H^{N_1}$ and $t\p_t \tilde{g}$ in $H^{N_1-1}$. Recall the equation satisfied by $\tilde{g}$ in (\ref{eqn5002}), due to the cubic and higher structure, it is not difficult to derive the following estimate,
\[
\sup_{t\in[0, T]} (1+t)^{-p_0} \| \mathcal{F}^{-1}[|\xi|\nabla_\xi \widehat{\tilde{g}}(\xi)]\|_{H^{N_1-1}} \lesssim \epsilon_0.
\]

 Recall the equation satisfied by $g$ in (\ref{equationforprofile}), we have
\[
\| e^{it \d}[C+Q_4+R]\|_{H^{N_1}}\lesssim \| (\phi_0, \Phi)\|_{X_{N_0}} (\|\Phi\|_{Z'}^2+\|\phi_0\|_{Z'_1})\lesssim (1+t)^{-1+p_0} \epsilon_1^3.
\]
It remains to estimate $Q_2$ in $H^{N_1}$ norm. We can first rule out the very high frequency as follows,
\[
\sum_{2^{k}\geq (1+t)^{2/N_1} } \| P_{k}[Q_2]\|_{H^{N_1}}\lesssim (1+t)^{-1}\|\Phi\|_{H^{N_0}}\| \Phi\|_{Z'}\lesssim \frac{1}{(1+t)^{3/2}} \epsilon_1^2.
\]
For the remaining cases, we do integration by parts in ``$\eta$". Very  similar to the proof of estimate (\ref{eqn5000}) and (\ref{eqn5009}),  the following estimates hold,
\[
\Big\| \frac{\mu\frac{\xi-\eta}{|\xi-\eta|}-\nu \frac{\eta}{|\eta|}}{|\mu\frac{\xi-\eta}{|\xi-\eta|}-\nu \frac{\eta}{|\eta|}|^2}  \tilde{m}'_{\mu, \nu}(\xi-\eta, \eta)\Big\|_{\mathcal{S}^{\infty}_{k,k_1,k_2}}\lesssim 2^{k_1+k_2}, \]
\[\Big\|\nabla_{\eta}\cdot\Big( \frac{\mu\frac{\xi-\eta}{|\xi-\eta|}-\nu \frac{\eta}{|\eta|}}{|\mu\frac{\xi-\eta}{|\xi-\eta|}-\nu \frac{\eta}{|\eta|}|^2}  \tilde{m}'_{\mu, \nu}(\xi-\eta, \eta)\Big)\Big\|_{\mathcal{S}^{\infty}_{k,k_1,k_2}}\lesssim 2^{\max\{k_1,k_2\}},
\]
which further gives us the following estimate 
\[
\sum_{2^{k}\lesssim (1+t)^{2/N_1}}\| P_{k}[Q_2]\|_{H^{N_1}}\lesssim \frac{1}{t}(1+t)^{2/N_1}\big[ \|\mathcal{F}^{-1}[|\xi|\nabla_{\xi}\widehat{g}(\xi)]\|_{H^{N_1}} + \| g\|_{H^{N_1+3}}\big]\| e^{-it\d} g\|_{Z'}\]
\[\lesssim \frac{1}{(1+t)^{5/4}} \big( \epsilon_1 \|\mathcal{F}^{-1}[|\xi|\nabla_{\xi}\widehat{g}(\xi)]\|_{H^{N_1}} + \epsilon_0\big) .
\]
To sum up, we have \[
\| \mathcal{F}^{-1}[|\xi|\nabla_{\xi}\widehat{g}(\xi)](\cdot)\|_{H^{N_1}} \lesssim  t^{p_0}\epsilon_0 + \epsilon_1 \| \mathcal{F}^{-1}[|\xi|\nabla_{\xi}\widehat{g}(\xi)]\|_{H^{N_1}},
\]
which further gives us 
\[
\| \mathcal{F}^{-1}[|\xi|\nabla_{\xi}\widehat{g}(\xi)](\cdot)\|_{H^{N_1}} \lesssim t^{p_0}\epsilon_0.
\]
Now, we can see the estimate (\ref{eqn5001}) indeed holds.

\section{Asymptotic behavior of the solution}\label{asymptotic}
As a byproduct of the global existence result, we can very easily see that $\phi_{0}$ scatters to zero in $X_{N_{0}}$.  From (\ref{equationforprofile}), (\ref{eqn60001}) and the definition of $\tilde{g}(t)$, we are motivated to define \[
\tilde{g}_{\infty}= \tilde{g}(0) + \mathcal{F}^{-1}\Big[\int_{0}^{\infty} \int_{\R} e^{is \Phi_{\mu, \nu}(\xi, \eta)}  \p_s[\widehat{g_{\mu}}(s,\xi-\eta) \widehat{g_{\nu}}(s,\eta)] a_{\mu, \nu}(\xi-\eta, \eta)  d \eta ds\Big]
\]
\[
+ \int_0^\infty e^{i s \Lambda}[ C + Q_4 + \mathcal{R}] d s,
\]
then as a byproduct of the improved $Z$-norm estimate for $g$, we have
\begin{equation}\label{equation6000}
\|\tilde{g}(t) -\tilde{g}_{\infty}\|_{Z}  \lesssim  \frac{1}{(1+t)^{p_{0}}} \epsilon_{0}.\end{equation}
Thus from (\ref{equation6000}), we can easily derive the following, 
\[
\| \Phi(t) - e^{ - i t |\nabla|} \tilde{g}_{\infty}\|_{H^{N_{1}+4}} \lesssim \| \tilde{\Phi}(t) -e^{ - i t |\nabla|} \tilde{g}_{\infty}\|_{H^{N_{1}+4}}  +\sum_{(\mu, \nu)\in \mathcal{S}} \| A_{\mu, \nu}(\Phi_{\mu}, \Phi_{\nu})\|_{H^{N_1+4}}
\]
\begin{equation}
\lesssim \frac{1}{(1+t)^{p_{0}}} \epsilon_{0}
 \rightarrow 0, \quad \emph{as $|t|\rightarrow \infty$}.
\end{equation}
That is to say, $\Phi(t)$ scatters to a linear solution in a lower regularity Sobolev space.

\section*{Appendix: Deriving the system (\ref{mainequation}) from the system (\ref{mainequation1}) }

Recall that 
\[
v= (-\p_2 \psi, \p_1 \psi), \quad G_{\cdot, 1}=(-\p_2 G_1, \p_1 G_1), \quad G_{\cdot, 2}=(-\p_2 G_2, \p_1 G_2),
\]
it is  easy to see the following identities hold,
\[
\nabla v  = \Bigg( \begin{array}{cc}
 -\p_1\p_2\psi &  -\p_2^2\psi \\
\p_1^2 \psi & \p_2\p_1\psi \\
 \end{array}\Bigg), \quad G= \Bigg(\begin{array}{cc}
-\p_2 G_1 & -\p_2 G_2\\
\p_1 G_1 &  \p_1 G_2\\
\end{array}\Bigg),
\]
\[
G^{\top} =  \Bigg(\begin{array}{cc}
-\p_2 G_1 & \p_1 G_1\\
-\p_2 G_2 &  \p_1 G_2\\
\end{array}\Bigg),\quad  GG^{\top}= \Bigg( \begin{array}{cc}
(\p_2 G_1)^2 + (\p_2 G_2)^2 & -\p_2 G_1 \p_1 G_1 -\p_2 G_2 \p_1 G_2\\ 
-\p_1 G_1 \p_2 G_1 -\p_1 G_2\p_2 G_2 & (\p_1G_1)^2 + (\p_1 G_2)^2
 \end{array}\Bigg). 
\]

We can take  the first component of the first equation of the system (\ref{mainequation1}) and write it in terms of $\psi, G_1$, and $G_2$. As a result,  we have
\[
-\p_t \p_2 \psi + \p_1\p_2 G_1 + \p_2^2 G_2 = -\p_1 p - (-\p_2 \psi \p_1+\p_1 \psi\p_2) (-\p_2 \psi)\]
\[
+ \p_1\big[(\p_2 G_1)^2 + (\p_2 G_2)^2 \big] + \p_2 \big[-\p_2 G_1 \p_1 G_1 -\p_2 G_2 \p_1 G_2\big]
\]
\[
= -\p_1 p -(\p_2 \psi \p_1\p_2 \psi - \p_1 \psi \p_2^2\psi) + \p_2 G_1 \p_1 \p_2 G_1 -\p_1 G_1 \p_2^2 G_1 + \p_2 G_2 \p_1\p_2 G_2 -\p_2^2 G_2 \p_1 G_2 
\]
\begin{equation}\label{eqn100}
= -\p_1 p - Q_{1,2}(\p_2\psi, \psi) + Q_{1,2}(\p_2 G_1, G_1)  + Q_{1,2}(\p_2 G_2, G_2).
\end{equation}
Very similarly, after taking the second component of the first equation of the  system (\ref{mainequation1}), we have
\[
\p_t \p_1\psi - [\p_1 \p_1 G_1 + \p_2\p_1 G_2]= -\p_2 p -(-\p_2\psi\p_1+\p_1 \psi \p_2)(\p_1 \psi) 
\]
\[
+ \p_1\big[ -\p_1 G_1 \p_2 G_1 -\p_1 G_2\p_2 G_2\big] + \p_2 \big[ (\p_1G_1)^2 + (\p_1 G_2)^2\big]
\]
\begin{equation}\label{eqn102}
= -\p_2 p + Q_{1,2}(\p_1\psi, \psi) - Q_{1,2}(\p_1 G_1, G_1) -Q_{1,2}(\p_1 G_2, G_2).
\end{equation}

Applying  $\p_2/|\nabla|^2$ on both hands side of equation (\ref{eqn100}) and $-\p_1/|\nabla|^2$ on both hands side of equation (\ref{eqn102}), and then adding those two equations together, we have
\[
\p_t \psi - \p_1 G_1 -\p_2 G_2 = -|\nabla|^{-1}R_2\big[Q_{1,2}(\p_2\psi, \psi) -Q_{1,2}(\p_2 G_1, G_1)  - Q_{1,2}(\p_2 G_2, G_2)
 \big]
\]
\[
- |\nabla|^{-1}R_1\big[Q_{1,2}(\p_1\psi, \psi) - Q_{1,2}(\p_1 G_1, G_1) -Q_{1,2}(\p_1 G_2, G_2) \big].
\]
Therefore the first equation of the system (\ref{mainequation}) holds. Now we proceed to derive the equations satisfied by $G_1$ and $G_2$. From the equation satisfied by $G$ in (\ref{mainequation1}), we have the following four equations,
\[
-\p_t \p_2 G_1 + \p_1\p_2 \psi = -(-\p_2 \psi \p_1 +\p_1\psi \p_2) (-\p_2 G_1) + (\p_1\p_2\psi\p_2 G_1 -\p_2^2\psi \p_1 G_1)
\]
\begin{equation}\label{eqn20001}
= Q_{1,2}(\psi, \p_2 G_1) + Q_{1,2}(\p_2\psi, G_1) = \p_2\big[ Q_{1,2}(\psi, G_1)\big],
\end{equation}
\begin{equation}\label{eqn20002}
\p_t \p_1 G_1 - \p_1^2\psi= Q_{1,2}(\p_1 G_1, \psi) + Q_{1,2}(G_1, \p_1 \psi)= -\p_1\big[ Q_{1,2}(\psi, G_1)\big],
\end{equation}
\[
-\p_t \p_2 G_2 +\p_2^2\psi = -(-\p_2 \psi \p_1 +\p_1\psi \p_2) (-\p_2 G_2)  + (\p_1\p_2\psi\p_2 G_2 -\p_2^2\psi\p_1 G_2 ) 
\]
\begin{equation}\label{eqn20003}
= Q_{1,2}(\psi,\p_2 G_2 ) + Q_{1,2}(\p_2 \psi, G_2) = \p_2\big[ Q_{1,2}(\psi, G_2)\big],
\end{equation}
\begin{equation}\label{eqn20004}
\p_t \p_1 G_2 -\p_2 \p_1\psi  = Q_{1,2}(\p_1 G_2, \psi) + Q_{1,2}(G_2, \p_1 \psi) =-\p_1\big[ Q_{1,2}(\psi, G_2)\big].
\end{equation}
From (\ref{eqn20001}) and (\ref{eqn20002}), we can see the following equation holds, 
\[
\p_t G_1 -\p_1 \psi = -Q_{1,2}(\psi, G_1)= Q_{1,2}(G_1, \psi).
\]
From (\ref{eqn20003}) and (\ref{eqn20004}), we can see the following equation holds, 
\[
\p_t G_2-\p_2 \psi = -Q_{1,2}(\psi, G_2) = Q_{1,2}(G_2, \psi).
\]
To sum up, all equations in the system (\ref{mainequation}) hold.

\bibliographystyle{nabbrav}

\end{document}